\documentclass[preprint,a4paper,twoside,english,10pt]{amsart}

\usepackage{fullpage}
\usepackage{amsmath,amsfonts,amssymb}
\usepackage[all]{xy}
\usepackage[english]{babel}
\usepackage[utf8x]{inputenc}
\usepackage{url}
\usepackage{ifthen}
\usepackage{booktabs}
\usepackage{multirow}
\usepackage{array}
\usepackage{paralist}
\usepackage{enumerate}
\usepackage[nomarkers,nolists]{endfloat}

\usepackage{placeins}

 \newtheorem{thm}{Theorem}[section]
 \newtheorem{prop}[thm]{Proposition}
 \newtheorem{cor}[thm]{Corollary}
 \newtheorem{lem}[thm]{Lemma}
 
 \newtheorem{assumption}[thm]{Assumption}
 
 \newtheorem{rem}[thm]{Remark}
 \newtheorem{ex}[thm]{Example}

\newcommand{\N}{\mathbb{N}}
\newcommand{\Z}{\mathbb{Z}}

\newcommand{\R}{\mathbb{R}}
\newcommand{\C}{\mathbb{C}}

\newcommand{\st}{\;:\;}

\newcommand{\sspace}{{\cdot}}
\newcommand{\ssspace}{{\cdot\cdot}}

\newcommand{\duale}[1]{#1^*}
\newcommand{\dualee}[1]{\left(#1\right)^*}

\newcommand{\scalar}[2]{\left\langle #1 \,\middle|\, #2 \right\rangle}
\newcommand{\norma}[1]{\left\| #1 \right\|}

\newcommand{\del}{\partial}
\newcommand{\delbar}{\overline{\del}}

\newcommand{\correnti}{\mathrm{D}}
\newcommand{\correntifascio}{\mathcal{D}}

\newcommand{\n}{\mathfrak{n}}
\newcommand{\g}{\mathfrak{g}}

\newcommand{\solvmfd}{\left. \Gamma \middle\backslash G \right.}

\newcommand{\perm}[1]{\mathfrak{S}_{#1}}

\newcommand{\paragrafo}[2]{\smallskip \noindent \texttt{Step {#1}} -- {\itshape #2}\\}

\DeclareMathOperator{\id}{id}
\DeclareMathOperator{\dom}{dom}
\DeclareMathOperator{\imm}{im}

\DeclareMathOperator{\de}{d}

\DeclareMathOperator{\esp}{e}

\DeclareMathOperator{\GL}{GL}

\DeclareMathOperator{\End}{End}
\DeclareMathOperator{\Aut}{Aut}
\DeclareMathOperator{\Hom}{Hom}
\DeclareMathOperator{\Der}{Der}

\DeclareMathOperator{\Char}{Char}

\DeclareMathOperator{\Tot}{Tot}

\DeclareMathOperator{\ad}{ad}
\DeclareMathOperator{\Ad}{Ad}

\DeclareMathOperator{\diag}{diag}

\newcommand{\pd}{\textsc{pd}}

\newcommand{\gl}{\mathfrak{gl}}

\allowdisplaybreaks[1]

\title{Bott-Chern cohomology of solvmanifolds}

\author{Daniele Angella}
\address[Daniele Angella]{Dipartimento di Matematica e Informatica ``Ulisse Dini''\\
Universit\`a di Firenze\\
via Morgagni 67/A\\
50134 Firenze, Italy}
\email{daniele.angella@unifi.it}
\email{daniele.angella@gmail.com}

\author{Hisashi Kasuya}
\address[Hisashi Kasuya]{Department of Mathematics, Graduate School of Science \\
Osaka University \\
Osaka, Japan.}
\email{kasuya@math.sci.osaka-u.ac.jp}

\thanks{During the preparation of this work, the first author has been granted with a research fellowship by Istituto Nazionale di Alta Matematica INdAM, and he is supported by the Project PRIN ``Varietà reali e complesse: geometria, topologia e analisi armonica'', by the Project FIRB ``Geometria Differenziale e Teoria Geometrica delle Funzioni'', by the project SIR2014 ``Analytic aspects in complex and hypercomplex geometry'' code RBSI14DYEB, and by GNSAGA of INdAM.
The second author  is supported by JSPS Research Fellowships for Young Scientists}

\begin{document}

\begin{abstract}
We study conditions under which sub-complexes of a double complex of vector spaces allow to compute the Bott-Chern cohomology. We are especially aimed at studying the Bott-Chern cohomology of special classes of solvmanifolds, namely, complex parallelizable solvmanifolds and solvmanifolds of splitting type. More precisely, we can construct explicit finite-dimensional double complexes that allow to compute the Bott-Chern cohomology of compact quotients of complex Lie groups, respectively, of some Lie groups of the type $\C^n\ltimes_\varphi N$ where $N$ is nilpotent. As an application, we compute the Bott-Chern cohomology of the complex parallelizable Nakamura manifold and of the completely-solvable Nakamura manifold.
In particular, the latter shows that the property of satisfying the $\del\delbar$-Lemma is not strongly-closed under deformations of the complex structure.
\end{abstract}

\maketitle

\tableofcontents

\section*{Introduction}

Given a double complex $\left(A^{\bullet,\bullet},\, \del,\, \delbar\right)$ of vector spaces, both the cohomology of the associated total complex $\left(\bigoplus_{p+q=\bullet} A^{p,q},\, \del+\delbar\right)$ and the cohomologies of the rows $\left(A^{\bullet,q},\, \del\right)$ and of the columns $\left(A^{p,\bullet},\, \delbar\right)$ have been widely studied. Two other interesting cohomologies are the \emph{Bott-Chern cohomology}, namely, the cohomology of the complex
$$ \mathcal{BC}^{p,q}(A^{\bullet,\bullet}) \;:=\;
A^{p-1,q-1} \stackrel{\del\delbar}{\longrightarrow} A^{p,q} \stackrel{\del+\delbar}{\longrightarrow} A^{p+1,q}\oplus A^{p,q+1} \;, $$
and the \emph{Aeppli cohomology}, namely, the cohomology of the complex
$$ \mathcal{A}^{p,q}(A^{\bullet,\bullet}) \;:=\;
A^{p-1,q}\oplus A^{p,q-1} \stackrel{\left(\del,\, \delbar\right)}{\longrightarrow} A^{p,q} \stackrel{\del\delbar}{\longrightarrow} A^{p+1,q+1} \;.$$

For a compact complex manifold $X$, the Bott-Chern and the Aeppli cohomologies of the double complex $\left(\wedge^{\bullet,\bullet}X,\, \del,\, \delbar\right)$ have been studied by many authors in several contexts, see, {\itshape e.g.} \cite{aeppli, bott-chern, bigolin, deligne-griffiths-morgan-sullivan, varouchas, alessandrini-bassanelli, schweitzer, kooistra, bismut, bismut-book, tseng-yau-3, angella-1, angella-tomassini-3}. They appear to be a completing useful tool besides the de Rham and the Dolbeault cohomologies. In this spirit, in \cite{angella-tomassini-3}, it is shown that an inequality {\itshape à la} Fr\"olicher, involving just the dimensions of the Bott-Chern cohomology and of the de Rham cohomology, holds true on any compact complex manifold, and further allows to characterize the validity of the $\del\delbar$-Lemma (namely, the very special cohomological property that every $\del$-closed $\delbar$-closed $\de$-exact form is $\del\delbar$-exact too, see, {\itshape e.g.} \cite{deligne-griffiths-morgan-sullivan}).

A compact complex manifold satisfies the $\del\delbar$-Lemma if and only if the Bott-Chern cohomology is naturally isomorphic to the Dolbeault cohomology, \cite[Remark 5.16]{deligne-griffiths-morgan-sullivan}. Therefore, since compact K\"ahler manifolds satisfy the $\del\delbar$-Lemma because of the K\"ahler identities, \cite[Lemma 5.11]{deligne-griffiths-morgan-sullivan}, the Bott-Chern cohomology is particularly interesting in studying complex non-K\"ahler manifolds.

In non-K\"ahler geometry, a very fruitful source of examples is provided by the class of nilmanifolds and solvmanifolds, namely, compact quotients of connected simply-connected nilpotent, respectively solvable, Lie groups by co-compact discrete subgroups. For instance, the geometry of nilmanifolds can be often reduced to the study of the associated Lie algebras, \cite{console, rollenske-survey, belgun}. On the other hand, nilmanifolds do not admit too strong geometric structures, \cite{benson-gordon-nilmanifolds, hasegawa}. More precisely, on a nilmanifold, the finite-dimensional sub-complex of left-invariant forms (namely, the forms being invariant for the action of the Lie group on itself given by left-translations) suffices in computing the de Rham cohomology, \cite{nomizu, hattori}. Whenever the nilmanifold is endowed with a suitable left-invariant complex structure, also the Dolbeault cohomology, \cite{sakane, cordero-fernandez-gray-ugarte, console-fino, rollenske, rollenske-survey}, and the 
Bott-Chern cohomology, \cite{angella-1}, can be computed by means of just left-invariant forms.

Instead, for solvmanifolds, the left-invariant forms are usually not enough to recover the whole de Rham cohomology: an example is the non-completely-solvable solvmanifold provided in \cite[Corollary 4.2]{debartolomeis-tomassini}. The de Rham cohomology of solvmanifolds has been studied by several authors, {\itshape e.g.} A. Hattori \cite{hattori}, G.~D. Mostow \cite{mostow}, S. Console and A. Fino \cite{console-fino-solvmanifolds}, and the second author \cite{kasuya-jdg, kasuya-holpar}. Several results concerning the Dolbeault cohomology have been proven by the second author, \cite{kasuya-mathz, kasuya-holpar}; such results allow to study Hodge symmetry, Hodge decomposition, formality, and the Hodge and Fr\"olicher spectral sequence on solvmanifolds, \cite{kasuya-hodge, kasuya-formality, kasuya-degeneration}.

\medskip

In this note, we study the Bott-Chern cohomology of a certain class of solvmanifolds. This is done with the scope to further investigate the complex geometry of non-K\"ahler manifolds and especially its cohomological aspects. More precisely, we start by studying conditions under which the Bott-Chern cohomology of a double complex can be completely recovered by a suitable sub-complex; see Theorem \ref{surjBC} and Theorem \ref{thm:inj}. As an application, we get the following result. (For further applications to the study of the symplectic cohomologies studied by L.-S. Tseng and S.-T. Yau in \cite{tseng-yau-1, tseng-yau-2}, see \cite{angella-kasuya-3}.)

\smallskip
\noindent {\bfseries Theorem (see Theorem \ref{BCISO} and Theorem \ref{palBCISO}).\ }
{\itshape
Let $G$ be a connected simply-connected solvable Lie group admitting a co-compact discrete subgroup $\Gamma$ and endowed with a $G$-left-invariant complex structure.
If
\begin{itemize}
\item either $G$ is a semidirect product $\C^{n}\ltimes _{\phi}N$ of $\C^{n}$ and a connected simply-connected nilpotent Lie group $N$ endowed with an $N$-left-invariant complex structure satisfying some conditions (see Assumption \ref{ass:solvmanifolds}),
\item or $G$ is a complex Lie group,
\end{itemize}
then there is an explicit finite-dimensional sub-complex $C^{\bullet,\bullet}$ of  the double complex $\left(\wedge^{\bullet,\bullet}\solvmfd,\, \del,\, \delbar\right)$ which computes the Bott-Chern cohomology of the solvmanifold $\solvmfd$.
}
\smallskip

As an application, we explicitly compute the Bott-Chern cohomology of the completely-solvable Nakamura manifold and of the complex parallelizable Nakamura manifold. This gives us, as a corollary, the following result.

\smallskip
\noindent {\bfseries Theorem (see Theorem \ref{thm:deldelbar-non-closed}).\ }
{\itshape
 Satisfying the $\del\delbar$-Lemma is not a strongly-closed property under small deformations of the complex structure.
}
\smallskip

In \cite{angella-kasuya-2}, we prove (the stronger result) that satisfying the $\del\delbar$-Lemma is not a (Zariski-)closed property.

\medskip

\noindent{\sl Acknowledgments.}
The first author would like to warmly thank Adriano Tomassini for his constant support and encouragement, for his several advices, and for many inspiring conversations. The second author would like to express his gratitude to Toshitake Kohno for helpful suggestions and stimulating discussions. The authors would like also to thank Luis Ugarte for suggestions and remarks.
Thanks also to Maria Beatrice Pozzetti and to the anonymous Referees, whose suggestions improved the presentation of the paper.

\section{Computing the cohomologies of double complexes by means of sub-complexes}\label{sec:subcplx}

In this section, we study several cohomologies associated to a bounded double complex of $\C$-vector spaces; in particular, we are interested in studying when such cohomologies can be recovered by means of a suitable (possibly finite-dimensional) sub-complex.

\subsection{The cohomology of the associated total complex}

Let $\left(A^{\bullet,\bullet},\, \del,\, \delbar\right)$ be a bounded double complex of $\C$-vector spaces, namely, $\del\in \End^{1,0}\left(A^{\bullet,\bullet}\right)$ and $\delbar\in \End^{0,1}\left(A^{\bullet,\bullet}\right)$ are such that $\del^2=\delbar^2=\left[\del,\delbar\right]=0$, and $A^{p,q}=\{0\}$ but for finitely-many $(p,q)\in\Z^2$. Denote by
$$ \left(\Tot^\bullet \left( A^{\bullet,\bullet} \right) := \bigoplus_{p+q=\bullet} A^{p,q},\; \de:=\del+\delbar\right) $$
the total complex associated to $\left(A^{\bullet,\bullet},\, \del,\, \delbar\right)$. The bi-grading of $\left(A^{\bullet,\bullet},\, \del,\, \delbar\right)$ induces two natural bounded filtrations of $\left(\Tot^\bullet \left( A^{\bullet,\bullet} \right),\, \de\right)$, namely,
$$ \left\{ \left( {'F}^{p} \Tot^\bullet \left( A^{\bullet,\bullet} \right) := \bigoplus_{\substack{r+s=\bullet\\r\geq p}}A^{r,s} ,\; \de\lfloor_{{'F}^{p} \Tot^\bullet \left( A^{\bullet,\bullet} \right)} \right) \hookrightarrow \left( \Tot^\bullet\left(A^{\bullet,\bullet}\right) ,\; \de \right)\right\}_{p\in\Z} $$
and
$$ \left\{ \left( {''F}^{q} \Tot^\bullet \left( A^{\bullet,\bullet} \right) := \bigoplus_{\substack{r+s=\bullet\\s\geq q}}A^{r,s} ,\; \de\lfloor_{{''F}^{q} \Tot^\bullet \left( A^{\bullet,\bullet} \right)} \right) \hookrightarrow \left( \Tot^\bullet\left(A^{\bullet,\bullet}\right) ,\; \de \right)\right\}_{q\in\Z} \;. $$
Such filtrations induce naturally two spectral sequences, respectively, 
$$ \left\{\left({'E}_r^{\bullet,\bullet}\left(A^{\bullet,\bullet},\, \del,\, \delbar\right),\, {'\de}_r\right)\right\}_{r\in\Z} \qquad \text{ and } \qquad \left\{\left({''E}_r^{\bullet,\bullet}\left(A^{\bullet,\bullet},\, \del,\, \delbar\right),\, {''\de}_r\right)\right\}_{r\in\Z} \;,$$
such that
$$ {'E}_1^{\bullet_1,\bullet_2}\left(A^{\bullet,\bullet},\, \del,\, \delbar\right) \;\simeq\; H^{\bullet_2}\left(A^{\bullet_1,\bullet},\, \delbar\right) \;\Rightarrow\; H^{\bullet_1+\bullet_2}\left(\Tot^\bullet\left(A^{\bullet,\bullet}\right),\, \de\right) \;, $$
and 
$$ {''E}_1^{\bullet_1,\bullet_2}\left(A^{\bullet,\bullet},\, \del,\, \delbar\right) \;\simeq\; H^{\bullet_1}\left(A^{\bullet,\bullet_2},\, \del\right) \;\Rightarrow\; H^{\bullet_1+\bullet_2}\left(\Tot^\bullet\left(A^{\bullet,\bullet}\right),\, \de\right) \;, $$
(where "$\Rightarrow$" denotes convergence of the spectral sequence,)
see, {\itshape e.g.} \cite[\S2.4]{mccleary}, see also \cite[\S3.5]{griffiths-harris}, \cite[Theorem 1, Theorem 3]{cordero-fernandez-ugarte-gray}.

One gets straightforwardly the following result, providing a sufficient condition under which a sub-complex $\left(C^{\bullet,\bullet},\, \del,\, \delbar\right) \hookrightarrow \left(A^{\bullet,\bullet},\, \del,\, \delbar\right)$ allows to recover the cohomology of $\left(\Tot^\bullet\left(A^{\bullet,\bullet}\right),\, \de\right)$.
(Recall that a quasi-isomorphism is a map between complexes that induces an isomorphism in cohomology.)

\begin{prop}\label{spdb}
 Let $\left(A^{\bullet,\bullet},\, \del,\, \delbar\right)$ be a bounded double complex of $\C$-vector spaces, and let $\left(C^{\bullet,\bullet},\, \del,\, \delbar\right) \hookrightarrow \left(A^{\bullet,\bullet},\, \del,\, \delbar\right)$ be a sub-complex.
 If, for every $p\in\Z$, the induced map $\left(C^{p,\bullet},\, \delbar\right) \hookrightarrow \left(A^{p,\bullet},\, \delbar\right)$ of complexes is a quasi-isomorphism, then the induced map
  $$
    \left(\Tot^\bullet\left(C^{\bullet,\bullet}\right),\, \de\right) \hookrightarrow \left(\Tot^\bullet\left(A^{\bullet,\bullet}\right),\, \de\right)
  $$
 of complexes is a quasi-isomorphism.
\end{prop}

\begin{proof}
The inclusion $\left(C^{\bullet,\bullet},\, \del,\, \delbar\right) \hookrightarrow \left(A^{\bullet,\bullet},\, \del,\, \delbar\right)$ induces a morphism
$$ \left\{ \left( {'F}^{p} \Tot^\bullet \left( C^{\bullet,\bullet} \right) ,\; \de \right) \right\}_{p\in\Z} \to \left\{ \left( {'F}^{p} \Tot^\bullet \left( A^{\bullet,\bullet} \right) ,\; \de \right) \right\}_{p\in\Z} $$
of the associated bounded filtrations, and hence in particular a morphism
$$ \left\{ \left( {'E}_r^{\bullet,\bullet}\left(C^{\bullet,\bullet},\, \del,\, \delbar\right) ,\, {'\de}_r \right) \right\}_{r\in\Z} \to \left\{ \left( {'E}_r^{\bullet,\bullet}\left(A^{\bullet,\bullet},\, \del,\, \delbar\right) ,\, {'\de}_r \right) \right\}_{r\in\Z} $$
of the associated spectral sequences.

By the hypothesis, the inclusion induces an isomorphism at the first level,
$$
\xymatrix{
{'E}_1^{\bullet,\bullet}\left(C^{\bullet,\bullet},\, \del,\, \delbar\right) \ar@{=>}[d] \ar[r]^{\simeq} & {'E}_1^{\bullet,\bullet}\left(A^{\bullet,\bullet},\, \del,\, \delbar\right) \ar@{=>}[d] \\
H^\bullet\left(\Tot^\bullet\left(C^{\bullet,\bullet}\right) ,\, \de\right) \ar[r] & H^\bullet\left(\Tot^\bullet\left(A^{\bullet,\bullet}\right) ,\, \de\right)
}
$$
and hence, $A^{\bullet,\bullet}$ being bounded, also an isomorphism
$$ H^\bullet\left(\Tot^\bullet\left(C^{\bullet,\bullet}\right) ,\, \de\right) \stackrel{\simeq}{\to} H^\bullet\left(\Tot^\bullet\left(A^{\bullet,\bullet}\right) ,\, \de\right) $$
see, {\itshape e.g.} \cite[Theorem 3.5]{mccleary}; in particular, the induced map
$$ \left(\Tot^\bullet\left(C^{\bullet,\bullet}\right),\, \de\right) \hookrightarrow \left(\Tot^\bullet\left(A^{\bullet,\bullet}\right),\, \de\right) $$
is a quasi-isomorphism. 
\end{proof}

\subsection{The Bott-Chern cohomology}

For any $\left(p,q\right)\in\Z^2$, other than the cohomologies of $\left(\Tot^\bullet\left(A^{\bullet,\bullet}\right),\, \de\right)$, of $\left(A^{\bullet,q},\, \del\right)$, and of $\left(A^{p,\bullet},\, \delbar\right)$, one can consider also the \emph{Bott-Chern cohomology}, \cite{bott-chern}, namely, the cohomology of the complex
$$ \mathcal{BC}^{p,q}(A^{\bullet,\bullet}) \;:=\; A^{p-1,q-1} \stackrel{\del\delbar}{\longrightarrow} A^{p,q} \stackrel{\del+\delbar}{\longrightarrow} A^{p+1,q}\oplus A^{p,q+1} \;, $$
and the \emph{Aeppli cohomology}, \cite{aeppli}, namely, the cohomology of the complex
$$ \mathcal{A}^{p,q}(A^{\bullet,\bullet}) \;:=\; A^{p-1,q}\oplus A^{p,q-1} \stackrel{\left(\del,\,	 \delbar\right)}{\longrightarrow} A^{p,q} \stackrel{\del\delbar}{\longrightarrow} A^{p+1,q+1} \;. $$

\medskip

In order to study conditions under which the Bott-Chern cohomology of a double complex can be recovered by means of a suitable sub-complex, we provide the following lemma. We first look at conditions yielding a surjective map in Bott-Chern cohomology.

\begin{lem}\label{lem:sol-delbar}
 Let $\left(A^{\bullet,\bullet},\, \del,\, \delbar\right)$ be a bounded double complex of $\C$-vector spaces, and let $\left(C^{\bullet,\bullet},\, \del,\, \delbar\right) \hookrightarrow \left(A^{\bullet,\bullet},\, \del,\, \delbar\right)$ be a sub-complex.
 Suppose that, for every $p\in\Z$, the induced map $\left(C^{p,\bullet},\, \delbar\right) \hookrightarrow \left(A^{p,\bullet},\, \delbar\right)$ of complexes is a quasi-isomorphism.
 If $\phi \in A^{p,q}$ is such that $\delbar\phi \in C^{p,q+1}$, then there exist $\tilde\phi \in C^{p,q}$ and $\hat\phi \in A^{p,q-1}$ such that $\phi = \tilde\phi + \delbar\hat\phi$.
\end{lem}

\begin{proof}
One has
$$ H^{q+1}\left(C^{p,\bullet},\, \delbar\right) \;\ni\; \left(\delbar\phi \mod\imm\delbar\right) \;\mapsto\; \left(0 \mod\imm\delbar\right) \;\in\; H^{q+1}\left(A^{p,\bullet},\, \delbar\right) \:;$$
since the map $H^{q+1}\left(C^{p,\bullet},\, \delbar\right) \stackrel{\simeq}{\to} H^{q+1}\left(A^{p,\bullet},\, \delbar\right)$ is injective, one gets that $\delbar\phi \in \imm\left(\delbar\colon C^{p,q} \to C^{p,q+1}\right)$: let $\tilde\phi_1 \in C^{p,q}$ be such that
$$ \delbar\phi \;=\; \delbar\tilde\phi_1 \;.$$
Therefore,
$$ \left(\left(\phi-\tilde\phi_1\right) \mod\imm\delbar \right) \;\in\; H^{q}\left(A^{p,\bullet},\, \delbar\right) \;;$$
since the map $H^{q}\left(C^{p,\bullet},\, \delbar\right) \stackrel{\simeq}{\to} H^{q}\left(A^{p,\bullet},\, \delbar\right)$ is surjective, one gets that there exist $\tilde\phi_2 \in \ker\left(\delbar\colon C^{p,q}\to C^{p,q+1}\right)$ and $\hat\phi \in A^{p,q-1}$ such that
$$ \phi-\tilde\phi_1 \;=\; \tilde\phi_2 + \delbar\hat\phi \;,$$
that is, $\phi = \tilde\phi + \delbar\hat\phi$ where $\tilde\phi := \tilde\phi_1 + \tilde\phi_2 \in C^{p,q}$ and $\hat\phi \in A^{p,q-1}$.
\end{proof}

The following result gives a first partial answer concerning the relation between the Bott-Chern cohomology of a double complex and the Bott-Chern cohomology of a suitable sub-complex; compare it with \cite[Theorem 3.7]{angella-1}, which is in turn inspired by M. Schweitzer's computations on the Iwasawa manifold in \cite[\S1.c]{schweitzer}.

\begin{thm}\label{surjBC}
 Let $\left(A^{\bullet,\bullet},\, \del,\, \delbar\right)$ be a bounded double complex of $\C$-vector spaces, and let $\left(C^{\bullet,\bullet},\, \del,\, \delbar\right) \hookrightarrow \left(A^{\bullet,\bullet},\, \del,\, \delbar\right)$ be a sub-complex. Fix $\left(p,q\right)\in\Z^2$.
 Suppose that:
 \begin{enumerate}
  \item\label{item:thm-surg-hp-1} for every $r\in\Z$, the induced map $\left(C^{r,\bullet},\, \delbar\right) \hookrightarrow \left(A^{r,\bullet},\, \delbar\right)$ of complexes is a quasi-isomorphism,
  \item\label{item:thm-surg-hp-2} for every $s\in\Z$, the induced map $\left(C^{\bullet,s},\, \del\right) \hookrightarrow \left(A^{\bullet,s},\, \del\right)$ of complexes is a quasi-isomorphism, and
  \item\label{item:thm-surg-hp-3} the induced map
        $$ \frac{\ker\left(\de\colon \Tot^{p+q}\left(C^{\bullet,\bullet}\right) \to \Tot^{p+q+1}\left(C^{\bullet,\bullet}\right) \right) \cap C^{p,q}}{\imm\left(\de\colon \Tot^{p+q-1}\left(C^{\bullet,\bullet}\right) \to \Tot^{p+q}\left(C^{\bullet,\bullet}\right)\right)} \to \frac{\ker\left(\de\colon \Tot^{p+q}\left(A^{\bullet,\bullet}\right) \to \Tot^{p+q+1}\left(A^{\bullet,\bullet}\right) \right) \cap A^{p,q}}{\imm\left(\de\colon \Tot^{p+q-1}\left(A^{\bullet,\bullet}\right) \to \Tot^{p+q}\left(A^{\bullet,\bullet}\right)\right)} $$
        is surjective.
 \end{enumerate}
 Then the induced map $\mathcal{BC}^{p,q}(C^{\bullet,\bullet}) \hookrightarrow \mathcal{BC}^{p,q}(A^{\bullet,\bullet})$ of complexes induces a surjective map in cohomology.
\end{thm}

\begin{proof}
Up to shifting, assume that $A^{r,s}=\{0\}$ whenever $\left(r,s\right)\not\in\N^2$.

\smallskip

\paragrafo{1}{Firstly, we prove that, under the hypotheses {\ref{item:thm-surg-hp-1}} and {\ref{item:thm-surg-hp-2}}, the inclusion $\left(C^{\bullet,\bullet},\, \del,\, \delbar\right) \hookrightarrow \left(A^{\bullet,\bullet},\, \del,\, \delbar\right)$ induces, for every $\left(r,s\right)\in\Z^2$, a surjective map
$$ \frac{\imm\left(\de\colon \Tot^{r+s-1} \left(C^{\bullet,\bullet}\right) \to \Tot^{r+s} \left(C^{\bullet,\bullet}\right) \right) \cap C^{r,s}}{\imm\left(\del\delbar\colon C^{r-1,s-1} \to C^{r,s}\right)}
\to
\frac{\imm\left(\de\colon \Tot^{r+s-1} \left(A^{\bullet,\bullet}\right) \to \Tot^{r+s} \left(A^{\bullet,\bullet}\right) \right) \cap A^{r,s}}{\imm\left(\del\delbar\colon A^{r-1,s-1} \to A^{r,s}\right)} \;. $$}
Indeed, let
\begin{eqnarray*}
\left( \omega^{r,s} \mod\imm\left(\del\delbar\colon A^{r-1,s-1} \to A^{r,s}\right) \right) &:=& \left( \de\eta \mod\imm\left(\del\delbar\colon A^{r-1,s-1} \to A^{r,s}\right) \right) \\[5pt]
 &\in& \frac{\imm\left(\de\colon \Tot^{r+s-1} \left(A^{\bullet,\bullet}\right) \to \Tot^{r+s} \left(A^{\bullet,\bullet}\right) \right) \cap A^{r,s}}{\imm\left(\del\delbar\colon A^{r-1,s-1} \to A^{r,s}\right)} \;.
\end{eqnarray*}
Consider the bi-degree decomposition $\eta=:\sum_{(a,b)\in\Z^2}\eta^{a,b}$ where $\eta^{a,b}\in A^{a,b}$, for $(a,b)\in\Z^2$. Hence, consider the system
$$
\left\{
\begin{array}{ccccclcc}
 && \del\eta^{r+s-1,0} &=&0 &&& \\[5pt]
\delbar\eta^{r+s-\ell,\ell-1} &+& \del\eta^{r+s-\ell-1,\ell} &=& 0 & & \text{ for } & \ell\in\{1,\ldots,s-1\} \\[5pt]
\delbar\eta^{r,s-1} &+& \del\eta^{r-1,s} &=& \omega^{r,s} & \mod\imm\left(\del\delbar\colon A^{r-1,s-1} \to A^{r,s}\right) && \\[5pt]
\delbar\eta^{\ell,r+s-\ell-1} &+& \del\eta^{\ell-1,r+s-\ell} &=& 0 & & \text{ for } & \ell\in\{1,\ldots,r-1\} \\[5pt]
\delbar\eta^{0,r+s-1} &&&=&0 &&&
\end{array}
\right. \;.
$$

Set $\eta^{r+s,-1}:=0$, and consider the equation
$$ \delbar\eta^{r+s-\ell,\ell-1} + \del\eta^{r+s-\ell-1,\ell} \;=\; 0 \mod\imm\left(\del\delbar\colon A^{r+s-\ell-1, \ell-1} \to A^{r+s-\ell, \ell}\right)\qquad \text{ for } \ell\in\left\{0,\ldots,s-1\right\} \;. $$

If $\eta^{r+s-\tilde\ell,\tilde\ell-1} \in C^{r+s-\tilde\ell,\tilde\ell-1}$ for some $\tilde\ell\in\left\{0,\ldots,s-1\right\}$, then, by applying Lemma \ref{lem:sol-delbar} to the double complex $\left(A^{\bullet,\bullet},\, \delbar,\, \del\right)$, one gets that there exist $\tilde\eta^{r+s-\tilde\ell-1,\tilde\ell} \in C^{r+s-\tilde\ell-1,\tilde\ell}$ and $\hat\eta^{r+s-\tilde\ell-2,\tilde\ell} \in A^{r+s-\tilde\ell-2,\tilde\ell}$ such that
$$ \eta^{r+s-\tilde\ell-1,\tilde\ell} \;=\; \tilde\eta^{r+s-\tilde\ell-1,\tilde\ell} + \del\hat\eta^{r+s-\tilde\ell-2,\tilde\ell} \;;$$
therefore, when $\tilde\ell \leq s-2$, one gets the system
$$
\left\{
\begin{array}{llcc}
\del\eta^{r+s-1,0} = 0 &&& \\[5pt]
\delbar\eta^{r+s-\ell,\ell-1} + \del\eta^{r+s-\ell-1,\ell} = 0 & \text{ for } & \ell\in\{1,\ldots,\tilde\ell-1\} \\[5pt]
\delbar\eta^{r+s-\tilde\ell,\tilde\ell-1} + \del\tilde\eta^{r+s-\tilde\ell-1,\tilde\ell} = 0 & & \\[5pt]
\delbar\tilde\eta^{r+s-\tilde\ell-1,\tilde\ell} + \del\left(\eta^{r+s-\tilde\ell-2,\tilde\ell+1}-\delbar\hat\eta^{r+s-\tilde\ell-2,\tilde\ell}\right) = 0 & & \\[5pt]
\delbar\left(\eta^{r+s-\tilde\ell-2,\tilde\ell+1}-\delbar\hat\eta^{r+s-\tilde\ell-2,\tilde\ell}\right) + \del\eta^{r+s-\tilde\ell-3,\tilde\ell+2} = 0 & & \\[5pt]
\delbar\eta^{r+s-\ell,\ell-1} + \del\eta^{r+s-\ell-1,\ell} = 0 & \text{ for } & \ell\in\{\tilde\ell+3,\ldots,s-1\} \\[5pt]
\delbar\eta^{r,s-1} + \del\eta^{r-1,s} = \omega^{r,s} \mod\imm\left(\del\delbar\colon A^{r-1,s-1} \to A^{r,s}\right) & & \\[5pt]
\delbar\eta^{\ell,r+s-\ell-1} + \del\eta^{\ell-1,r+s-\ell} = 0 & \text{ for } & \ell\in\{1,\ldots,r-1\} \\[5pt]
\delbar\eta^{0,r+s-1} = 0 &&
\end{array}
\right. \;,
$$
where $\tilde\eta^{r+s-\tilde\ell-1,\tilde\ell} \in C^{r+s-\tilde\ell-1,\tilde\ell}$, and when $\tilde\ell = s-1$, one gets the system
$$
\left\{
\begin{array}{llcc}
\del\eta^{r+s-1,0} = 0 &&& \\[5pt]
\delbar\eta^{r+s-\ell,\ell-1} + \del\eta^{r+s-\ell-1,\ell} = 0 & \text{ for } & \ell\in\{1,\ldots,s-2\} \\[5pt]
\delbar\eta^{r+1,s-2} + \del\tilde\eta^{r,s-1} = 0 & & \\[5pt]
\delbar\tilde\eta^{r,s-1} + \del\eta^{r-1,s} = \omega^{r,s} \mod\imm\left(\del\delbar\colon A^{r-1,s-1} \to A^{r,s}\right) & & \\[5pt]
\delbar\eta^{\ell,r+s-\ell-1} + \del\eta^{\ell-1,r+s-\ell} = 0 & \text{ for } & \ell\in\{1,\ldots,r-1\} \\[5pt]
\delbar\eta^{0,r+s-1} = 0 &&
\end{array}
\right. \;,
$$
where $\tilde\eta^{r,s-1} \in C^{r,s-1}$.

In particular, since $\eta^{r+s,-1} = 0 \in C^{r+s,-1}$, we may assume that $\eta^{r,s-1} \in C^{r,s-1}$.

Analogously, by applying Lemma \ref{lem:sol-delbar} to the double complex $\left(A^{\bullet,\bullet},\, \del,\, \delbar\right)$, we may assume that $\eta^{r-1,s} \in C^{r-1,s}$.

Therefore
\begin{eqnarray*}
\omega^{r,s} \mod\imm\left(\del\delbar\colon A^{r-1,s-1} \to A^{r,s}\right) &=& \left(\delbar\eta^{r,s-1} + \del \eta^{r-1,s}\right) \mod\imm\left(\del\delbar\colon A^{r-1,s-1} \to A^{r,s}\right) \\[5pt]
&\in& \frac{\imm\left(\de\colon \Tot^{r+s-1} \left(C^{\bullet,\bullet}\right) \to \Tot^{r+s} \left(C^{\bullet,\bullet}\right) \right) \cap C^{r,s}}{\imm\left(\del\delbar\colon A^{r-1,s-1} \to A^{r,s}\right)} \;,
\end{eqnarray*}
that is, the induced map
$$ \frac{\imm\left(\de\colon \Tot^{r+s-1} \left(C^{\bullet,\bullet}\right) \to \Tot^{r+s} \left(C^{\bullet,\bullet}\right) \right) \cap C^{r,s}}{\imm\left(\del\delbar\colon C^{r-1,s-1} \to C^{r,s}\right)}
\to
\frac{\imm\left(\de\colon \Tot^{r+s-1} \left(A^{\bullet,\bullet}\right) \to \Tot^{r+s} \left(A^{\bullet,\bullet}\right) \right) \cap A^{r,s}}{\imm\left(\del\delbar\colon A^{r-1,s-1} \to A^{r,s}\right)} $$
is surjective.

\smallskip

\paragrafo{2}{Now, we prove that, under the additional assumption {\ref{item:thm-surg-hp-3}}, the induced map
$$ \frac{\ker\left(\del\colon C^{p,q} \to C^{p+1,q}\right) \cap \ker\left(\delbar\colon C^{p,q} \to C^{p,q+1}\right)}{\imm\left(\del\delbar \colon C^{p-1,q-1} \to C^{p,q}\right)} \to \frac{\ker\left(\del\colon A^{p,q} \to A^{p+1,q}\right) \cap \ker\left(\delbar\colon A^{p,q} \to A^{p,q+1}\right)}{\imm\left(\del\delbar \colon A^{p-1,q-1} \to A^{p,q}\right)} $$
is surjective.}
Indeed, consider the commutative diagram
$$
\xymatrix{
& 0 \ar[d] \ar@{=}[r] & 0 \ar[d] & \\
& \frac{\imm\left(\de\colon \Tot^{p+q-1} \left(C^{\bullet,\bullet}\right) \to \Tot^{p+q} \left(C^{\bullet,\bullet}\right) \right) \cap C^{p,q}}{\imm\left(\del\delbar\colon C^{p-1,q-1} \to C^{p,q}\right)} \ar[d] \ar[r]
& \frac{\imm\left(\de\colon \Tot^{p+q-1} \left(A^{\bullet,\bullet}\right) \to \Tot^{p+q} \left(A^{\bullet,\bullet}\right) \right) \cap A^{p,q}}{\imm\left(\del\delbar\colon A^{p-1,q-1} \to A^{p,q}\right)} \ar[d] \ar[r] & 0 \\
& \frac{\ker\left(\del\colon C^{p,q} \to C^{p+1,q}\right) \cap \ker\left(\delbar\colon C^{p,q} \to C^{p,q+1}\right)}{\imm\left(\del\delbar \colon C^{p-1,q-1} \to C^{p,q}\right)} \ar[r] \ar[d]
& \frac{\ker\left(\del\colon A^{p,q} \to A^{p+1,q}\right) \cap \ker\left(\delbar\colon A^{p,q} \to A^{p,q+1}\right)}{\imm\left(\del\delbar \colon A^{p-1,q-1} \to A^{p,q}\right)} \ar[d] & \\
& \frac{\ker\left(\de\colon \Tot^{p+q}\left(C^{\bullet,\bullet}\right) \to \Tot^{p+q+1}\left(C^{\bullet,\bullet}\right) \right) \cap C^{p,q}}{\imm\left(\de\colon \Tot^{p+q-1}\left(C^{\bullet,\bullet}\right) \to \Tot^{p+q}\left(C^{\bullet,\bullet}\right)\right)} \ar[d] \ar[r] \ar[d]
& \frac{\ker\left(\de\colon \Tot^{p+q}\left(A^{\bullet,\bullet}\right) \to \Tot^{p+q+1}\left(A^{\bullet,\bullet}\right) \right) \cap A^{p,q}}{\imm\left(\de\colon \Tot^{p+q-1}\left(A^{\bullet,\bullet}\right) \to \Tot^{p+q}\left(A^{\bullet,\bullet}\right)\right)} \ar[d] \ar[r] & 0 \\
& 0 \ar@{=}[r] & 0 &
}
$$
whose rows and columns are exact. By the Five Lemma, see, {\itshape e.g.} \cite[page 26]{mccleary}, the map
$$ \frac{\ker\left(\del\colon C^{p,q} \to C^{p+1,q}\right) \cap \ker\left(\delbar\colon C^{p,q} \to C^{p,q+1}\right)}{\imm\left(\del\delbar \colon C^{p-1,q-1} \to C^{p,q}\right)} \to
\frac{\ker\left(\del\colon A^{p,q} \to A^{p+1,q}\right) \cap \ker\left(\delbar\colon A^{p,q} \to A^{p,q+1}\right)}{\imm\left(\del\delbar \colon A^{p-1,q-1} \to A^{p,q}\right)} $$
is surjective, completing the proof.
\end{proof}

\medskip

We study now injectivity of maps in Bott-Chern cohomology.
In order to provide conditions under which the inclusion of a suitable sub-complex induces an injective map in Bott-Chern cohomology, we consider a further structure of Hilbert space on the double complex. (For similar results in the case of solvmanifolds, see \cite[Lemma 9]{console-fino}, \cite[Lemma 3.6]{angella-1}.)

\medskip

Let $A$ be a Hilbert space, with inner product $\scalar{\sspace}{\ssspace}\colon A\times A \to \C$. Denote by $\norma{\sspace}:=\scalar{\sspace}{\sspace}^{1/2}$ the associated norm.

Given a densely-defined linear operator $L\colon A \supseteq \dom(L) \to A$ on $A$, denote by 
$$ L^{*}_{\scalar{\sspace}{\ssspace}} \colon \dom\left(L^{*}_{\scalar{\sspace}{\ssspace}}\right) \to A $$
its $\scalar{\sspace}{\ssspace}$-adjoint operator, that is, the unique linear operator with domain
$$ \dom\left(L^{*}_{\scalar{\sspace}{\ssspace}}\right) \;:=\; \left\{ y \in A \st \scalar{L\, \sspace}{y}\colon \dom(L) \to \C \text{ is continuous} \right\} $$
and defined by
$$ \forall x \in \dom(L),\; \forall y \in \dom\left(L^{*}_{\scalar{\sspace}{\ssspace}}\right), \qquad \scalar{Lx}{y} \;=\; \scalar{x}{L^{*}_{\scalar{\sspace}{\ssspace}}y} \;.$$

Given a closed sub-space $C$ of $A$, denote the induced inner product on $C$ by $\scalar{\sspace}{\ssspace}_C:=\scalar{\sspace}{\ssspace}\lfloor_{C\times C}\colon C\times C \to \C$, and the orthogonal projection onto $C$ by $\pi_{\scalar{\sspace}{\ssspace}}^C\colon A \to C \subseteq A$. One has that
$$ \pi_{\scalar{\sspace}{\ssspace}}^C\lfloor_{C} \;=\; \id_C \qquad \text{ and } \qquad \scalar{ C }{ \left(\id_A - \pi^C_{\scalar{\sspace}{\ssspace}}\right)(A) } \;=\; \left\{0\right\} \;.$$

(To simplify notations, we do not specify the inner product $\scalar{\sspace}{\ssspace}$ in writing the projection or the adjoint, whenever it is clear from the context.) 

\medskip

We firstly record the following lemma, stating that, if $L$ commutes with $\pi^C$, then also $L^*$ does.

\begin{lem}\label{lem:adjoint-subspace}
 Let $A$ be a Hilbert space, with inner product $\scalar{\sspace}{\ssspace}$. Let $L\colon A \supseteq \dom(L) \to A$ be a densely-defined linear operator on $A$. Let $C$ be a closed sub-space of $A$ contained in $\dom(L)$ and in $\dom\left(L^*_{\scalar{\sspace}{\ssspace}}\right)$. Suppose that
 $$ \pi_{\scalar{\sspace}{\ssspace}}^C \,\circ\, L \;=\; L \,\circ\, \pi_{\scalar{\sspace}{\ssspace}}^C\colon \dom(L) \to C \;. $$
 Then
 $$ \pi^C_{\scalar{\sspace}{\ssspace}} \,\circ\, L^*_{\scalar{\sspace}{\ssspace}} \;=\; L^*_{\scalar{\sspace}{\ssspace}} \,\circ\, \pi^C_{\scalar{\sspace}{\ssspace}} \colon \dom\left(L^*_{\scalar{\sspace}{\ssspace}}\right) \to C \;;$$
 in particular, $L^*_{\scalar{\sspace}{\ssspace}}\lfloor_C \colon C \to C$, and hence $\left(L\lfloor_C\right)^*_{\scalar{\sspace}{\ssspace}_{C}} = L^*_{\scalar{\sspace}{\ssspace}}\lfloor_{C}$.
\end{lem}

\begin{proof}
It suffices to note that $\pi^C \colon A \to C \subseteq A$ is self-$\scalar{\sspace}{\ssspace}$-adjoint: for any $\alpha,\beta\in A$,
$$ \scalar{\pi^C \alpha}{\beta} \;=\; \scalar{\pi^C \alpha}{\beta-\left(\beta-\pi^C\beta\right)} \;=\; \scalar{\pi^C \alpha}{\pi^C\beta} \;=\; \scalar{\pi^C \alpha + \left(\alpha-\pi^C\alpha\right)}{\pi^C\beta} \;=\; \scalar{\alpha}{\pi^C\beta} \;. $$
It follows straightforwardly that $\pi^C \circ L^* = L^* \circ \pi^C \colon \dom\left(L^*\right) \to C$.
In particular, since $\pi^C\lfloor_C = \id_C$ and $C \subseteq \dom\left(L^*\right)$, it follows that $L^*\left(C\right) = \left(L^* \circ \pi^C\right)\left(C\right) = \left(\pi^C \circ L^*\right)\left(C\right) \subseteq C$, and hence $L^*\lfloor_{C} = \left(L\lfloor_C\right)^*_{\scalar{\sspace}{\ssspace}_{C}} \colon C \to C$.
\end{proof}

\medskip

Now, let $A^{\bullet,\bullet}$ be a bounded $\Z^2$-graded vector space with a structure of Hilbert space, with inner product $\scalar{\sspace}{\ssspace}$ such that $\scalar{ A^{p,q} }{ A^{p',q'} } = \{0\}$ for every $(p,q) \neq (p',q')$. Let
$$ \del\colon A^{\bullet,\bullet} \supseteq \dom(\del)^{\bullet,\bullet} \to A^{\bullet+1,\bullet} \qquad \text{ and } \qquad \delbar\colon A^{\bullet,\bullet} \supseteq \dom(\delbar)^{\bullet,\bullet} \to A^{\bullet,\bullet+1} $$
be densely-defined linear operators yielding a structure $\left(\left(\dom(\del)\cap\dom(\delbar)\right)^{\bullet,\bullet},\, \del,\, \delbar\right)$ of bounded double complex of $\C$-vector spaces.
Denote by
$$ \del^* \;:=\; \del^*_{\scalar{\sspace}{\ssspace}}\colon A^{\bullet,\bullet} \supseteq \dom\left(\del^*\right)^{\bullet,\bullet} \to A^{\bullet-1,\bullet} \qquad \text{ and } \qquad \delbar^* \;:=\; \delbar^*_{\scalar{\sspace}{\ssspace}}\colon A^{\bullet,\bullet} \supseteq \dom\left(\delbar^*\right)^{\bullet,\bullet} \to A^{\bullet,\bullet-1} $$
the $\scalar{\sspace}{\ssspace}$-adjoint operators of $\del$ and, respectively, $\delbar$.

Following \cite[Proposition 5]{kodaira-spencer-3}, see also \cite[\S2.b, \S2.c]{schweitzer}, define the (densely-defined) self-$\scalar{\sspace}{\ssspace}$-adjoint operator
\begin{eqnarray*}
\tilde\Delta^{BC} \;:=\; \tilde\Delta^{BC}_{\scalar{\sspace}{\ssspace}} &:=& \left(\del\delbar\right)\left(\del\delbar\right)^*+\left(\del\delbar\right)^*\left(\del\delbar\right)+\left(\delbar^*\del\right)\left(\delbar^*\del\right)^*+\left(\delbar^*\del\right)^*\left(\delbar^*\del\right)+\delbar^*\delbar+\del^*\del \\[5pt]
&\in& \Hom^{0,0}\left(\dom\left(\tilde\Delta^{BC}_{\scalar{\sspace}{\ssspace}}\right)^{\bullet,\bullet};\, A^{\bullet,\bullet}\right) \;.
\end{eqnarray*}

The following lemma states that, under a suitable decomposition hypothesis, the Bott-Chern cohomology of $\left(A^{\bullet,\bullet},\, \del,\, \delbar\right)$ is isomorphic to $\ker \tilde\Delta^{BC}$.

\begin{lem}\label{lem:hodge-isom-bc}
 Let $A^{\bullet,\bullet}$ be a bounded $\Z^2$-graded vector space with a structure of Hilbert space, with inner product $\scalar{\sspace}{\ssspace}$ such that $\scalar{ A^{p,q} }{ A^{p',q'} } = \{0\}$ for every $(p,q) \neq (p',q')$. Let $\del\colon A^{\bullet,\bullet} \supseteq \dom(\del)^{\bullet,\bullet} \to A^{\bullet+1,\bullet}$ and $\delbar\colon A^{\bullet,\bullet} \supseteq \dom(\delbar)^{\bullet,\bullet} \to A^{\bullet,\bullet+1}$ be densely-defined linear operators yielding a structure $\left(\left(\dom(\del)\cap\dom(\delbar)\right)^{\bullet,\bullet},\, \del,\, \delbar\right)$ of bounded double complex of $\C$-vector spaces.
 Suppose that the operator $\tilde\Delta^{BC}_{\scalar{\sspace}{\ssspace}} \in \Hom^{0,0}\left(\dom\left(\tilde\Delta^{BC}_{\scalar{\sspace}{\ssspace}}\right)^{\bullet,\bullet};\, A^{\bullet,\bullet}\right)$ induces the decomposition
 $$ \dom\left(\tilde\Delta^{BC}_{\scalar{\sspace}{\ssspace}}\right) \;=\; \ker \tilde\Delta^{BC}_{\scalar{\sspace}{\ssspace}} \oplus \imm \tilde\Delta^{BC}_{\scalar{\sspace}{\ssspace}} \;. $$
 Then, for every $\left(p,q\right)\in\Z^2$, the induced map
 $$
 \left( 0 \to \ker \tilde\Delta^{BC}_{\scalar{\sspace}{\ssspace}} \cap A^{p,q} \to 0 \right) \hookrightarrow \mathcal{BC}^{p,q}(A^{\bullet,\bullet})
 $$
 is a quasi-isomorphism.
\end{lem}

\begin{proof}
  Note that, for every $\eta\in \dom\left(\tilde\Delta^{BC}\right)$, one has
 $$ \scalar{\tilde\Delta^{BC} \eta}{\eta} \;=\; \norma{\left(\del\delbar\right)^*\eta}^2 + \norma{\del\delbar\eta}^2 + \norma{\del^*\delbar\eta}^2 + \norma{\delbar^*\del\eta}^2 + \norma{\delbar\eta}^2 + \norma{\del\eta}^2 \;, $$
 hence
 $$ \ker\tilde\Delta^{BC} \;=\; \ker\del \cap \ker\delbar \cap \ker\left(\del\delbar\right)^* \;.$$

 On the other hand, since $\imm\tilde\Delta^{BC} \subseteq \imm\del\delbar \oplus \left(\imm\del^* + \imm\delbar^*\right)$ and $\left(\imm\del^* + \imm\delbar^*\right) \cap \left(\ker\del\cap\ker\delbar\right) = \{0\}$, one has
 $$ \imm\tilde\Delta^{BC} \cap \left(\ker\del\cap\ker\delbar\right) \;\subseteq\; \imm\del\delbar \;. $$
 It follows that
 $$
 \ker\tilde\Delta^{BC} \cap A^{p,q}
 \stackrel{\simeq}{\to}
 \frac{\ker\tilde\Delta^{BC} \cap A^{p,q} + \imm\del\delbar\cap A^{p,q}}{\imm\left(\del\delbar\colon A^{p-1,q-1}\to A^{p,q}\right)}
 \;\simeq\;
 \frac{\ker\left(\del+\delbar \colon A^{p,q}\to A^{p+1,q}\oplus A^{p,q+1}\right)}{\imm\left(\del\delbar\colon A^{p-1,q-1}\to A^{p,q}\right)}
  \;,
 $$
 completing the proof.
\end{proof}

We have now the following result.

\begin{thm}\label{thm:inj}
 Let $A^{\bullet,\bullet}$ be a bounded $\Z^2$-graded vector space with a structure of Hilbert space, with inner product $\scalar{\sspace}{\ssspace}$ such that $\scalar{ A^{p,q} }{ A^{p',q'} } = \{0\}$ for every $(p,q) \neq (p',q')$. Let $\del\colon A^{\bullet,\bullet} \supseteq \dom(\del)^{\bullet,\bullet} \to A^{\bullet+1,\bullet}$ and $\delbar\colon A^{\bullet,\bullet} \supseteq \dom(\delbar)^{\bullet,\bullet} \to A^{\bullet,\bullet+1}$ be densely-defined linear operators yielding a structure $\left(\left(\dom(\del)\cap\dom(\delbar)\right)^{\bullet,\bullet},\, \del,\, \delbar\right)$ of bounded double complex of $\C$-vector spaces.
 Let
 $$ j\colon \left(C^{\bullet,\bullet},\, \del,\, \delbar\right) \hookrightarrow \left(\left(\dom(\del)\cap\dom(\delbar)\right)^{\bullet,\bullet},\, \del,\, \delbar\right) $$
 be a sub-complex.
 Suppose that:
 \begin{enumerate}
  \item\label{item:thm-inj-1} the operator $\tilde\Delta^{BC}_{\scalar{\sspace}{\ssspace}} \in \Hom^{0,0}\left(\dom\left(\tilde\Delta^{BC}_{\scalar{\sspace}{\ssspace}}\right)^{\bullet,\bullet};\, A^{\bullet,\bullet}\right)$ induces the decomposition
  $$ \dom\left(\tilde\Delta^{BC}_{\scalar{\sspace}{\ssspace}}\right) \;=\; \ker \tilde\Delta^{BC}_{\scalar{\sspace}{\ssspace}} \oplus \imm \tilde\Delta^{BC}_{\scalar{\sspace}{\ssspace}} \;;$$
  \item\label{item:thm-inj-2} it holds that
  $$ \del^*_{\scalar{\sspace}{\ssspace}}\lfloor_{C^{\bullet,\bullet}} \;=\; \left(\del\lfloor_{C^{\bullet,\bullet}}\right)^*_{\scalar{\sspace}{\ssspace}_{C^{\bullet,\bullet}}} \colon \dom\left( \del^*_{\scalar{\sspace}{\ssspace}}\lfloor_{C^{\bullet,\bullet}} \right)^{\bullet,\bullet} \to C^{\bullet-1,\bullet} $$
  and
  $$ \delbar^*_{\scalar{\sspace}{\ssspace}}\lfloor_{C^{\bullet,\bullet}} \;=\; \left(\delbar\lfloor_{C^{\bullet,\bullet}}\right)^*_{\scalar{\sspace}{\ssspace}_{C^{\bullet,\bullet}}} \colon \dom\left( \delbar^*_{\scalar{\sspace}{\ssspace}}\lfloor_{C^{\bullet,\bullet}} \right)^{\bullet,\bullet} \to C^{\bullet,\bullet-1} \;; $$
  in particular, it follows that
  $$ \tilde\Delta^{BC}_{\scalar{\sspace}{\ssspace}}\lfloor_{C^{\bullet,\bullet}} \;=\; \tilde\Delta^{BC}_{\scalar{\sspace}{\ssspace}_{C^{\bullet,\bullet}}} \;\in\; \Hom^{0,0}\left(\dom\left(\tilde\Delta^{BC}\lfloor_{C^{\bullet,\bullet}}\right)^{\bullet,\bullet};\, C^{\bullet,\bullet}\right) \;; $$
  \item\label{item:thm-inj-3} the operator $\tilde\Delta^{BC}_{\scalar{\sspace}{\ssspace}}\lfloor_{C^{\bullet,\bullet}} \in \Hom^{0,0}\left(\dom\left(\tilde\Delta^{BC}_{\scalar{\sspace}{\ssspace}}\lfloor_{C^{\bullet,\bullet}}\right)^{\bullet,\bullet};\, C^{\bullet,\bullet}\right)$ induces the decomposition
  $$ \dom\left(\tilde\Delta^{BC}_{\scalar{\sspace}{\ssspace}}\lfloor_{C^{\bullet,\bullet}}\right) \;=\; \ker \tilde\Delta^{BC}_{\scalar{\sspace}{\ssspace}}\lfloor_{C^{\bullet,\bullet}} \oplus \imm \tilde\Delta^{BC}_{\scalar{\sspace}{\ssspace}}\lfloor_{C^{\bullet,\bullet}} \;.$$
 \end{enumerate}
 Then, for every $\left(p,q\right)\in\Z^2$, the induced map $j \colon \mathcal{BC}^{p,q}(C^{\bullet,\bullet})  \hookrightarrow \mathcal{BC}^{p,q}(A^{\bullet,\bullet})$
 of complexes induces an injective map $j^*$ in cohomology.
\end{thm}

\begin{proof}
 By Lemma \ref{lem:hodge-isom-bc} and under the hypotheses \ref{item:thm-inj-1}, \ref{item:thm-inj-2}, and \ref{item:thm-inj-3}, one gets that both
 $$
 \left( 0 \to \ker \tilde\Delta^{BC} \cap A^{p,q} \to 0 \right) \hookrightarrow \mathcal{BC}^{p,q}(A^{\bullet,\bullet})
 $$
 and
 $$
 \left( 0 \to \ker \tilde\Delta^{BC}\lfloor_{C^{\bullet,\bullet}} \cap C^{p,q} = \ker\tilde\Delta^{BC}_{\scalar{\sspace}{\ssspace}_{C^{\bullet,\bullet}}} \cap C^{p,q} \to 0 \right) \hookrightarrow \mathcal{BC}^{p,q}(C^{\bullet,\bullet})
 $$
 are quasi-isomorphisms.

 Hence, one has the commutative diagram
 $$
 \xymatrix{
  \ker \tilde\Delta^{BC}\lfloor_{C^{\bullet,\bullet}} \cap C^{p,q} \ar[r]^{\simeq} \ar@{^{(}->}[d]^{j} & \frac{\ker\left(\del+\delbar \colon C^{p,q}\to C^{p+1,q}\oplus C^{p,q+1}\right)}{\imm\left(\del\delbar\colon C^{p-1,q-1}\to C^{p,q}\right)} \ar[d]^{j^*} \\
  \ker \tilde\Delta^{BC} \cap A^{p,q} \ar[r]^{\simeq} & \frac{\ker\left(\del+\delbar \colon A^{p,q}\to A^{p+1,q}\oplus A^{p,q+1}\right)}{\imm\left(\del\delbar\colon A^{p-1,q-1}\to A^{p,q}\right)}
 }
 $$
 getting that $j^*$ is injective.
\end{proof}

By using Lemma \ref{lem:adjoint-subspace}, one gets the following corollary of Theorem \ref{thm:inj}, concerning closed sub-complexes.

\begin{cor}\label{cor:inj-closed}
 Let $A^{\bullet,\bullet}$ be a bounded $\Z^2$-graded vector space with a structure of Hilbert space, with inner product $\scalar{\sspace}{\ssspace}$ such that $\scalar{ A^{p,q} }{ A^{p',q'} } = \{0\}$ for every $(p,q) \neq (p',q')$. Let $\del\colon A^{\bullet,\bullet} \supseteq \dom(\del)^{\bullet,\bullet} \to A^{\bullet+1,\bullet}$ and $\delbar\colon A^{\bullet,\bullet} \supseteq \dom(\delbar)^{\bullet,\bullet} \to A^{\bullet,\bullet+1}$ be densely-defined linear operators yielding a structure $\left(\left(\dom(\del)\cap\dom(\delbar)\right)^{\bullet,\bullet},\, \del,\, \delbar\right)$ of bounded double complex of $\C$-vector spaces.
 Let $j\colon \left(C^{\bullet,\bullet},\, \del,\, \delbar\right) \hookrightarrow \left(\left(\dom(\del)\cap\dom(\delbar)\right)^{\bullet,\bullet},\, \del,\, \delbar\right)$ be a closed sub-complex.
 Suppose that:
 \begin{enumerate}
  \item the operator $\tilde\Delta^{BC}_{\scalar{\sspace}{\ssspace}} \in \Hom^{0,0}\left(\dom\left(\tilde\Delta^{BC}_{\scalar{\sspace}{\ssspace}}\right)^{\bullet,\bullet};\, A^{\bullet,\bullet}\right)$ induces the decomposition
  $$ \dom\left(\tilde\Delta^{BC}_{\scalar{\sspace}{\ssspace}}\right) \;=\; \ker \tilde\Delta^{BC}_{\scalar{\sspace}{\ssspace}} \oplus \imm \tilde\Delta^{BC}_{\scalar{\sspace}{\ssspace}} \;; $$
  \item $C^{\bullet,\bullet}\subseteq \dom(\del)\cap\dom(\delbar)\cap\dom\left(\del^*_{\scalar{\sspace}{\ssspace}}\right)\cap\dom\left(\delbar^*_{\scalar{\sspace}{\ssspace}}\right)$, and $\pi^{C^{\bullet,\bullet}} \circ \del = \del \circ \pi^{C^{\bullet,\bullet}} \colon \dom(\del)^{\bullet,\bullet} \to C^{\bullet+1,\bullet}$ and $\pi^{C^{\bullet,\bullet}} \circ \delbar = \delbar \circ \pi^{C^{\bullet,\bullet}} \colon \dom(\delbar)^{\bullet,\bullet} \to C^{\bullet,\bullet+1}$.
 \end{enumerate}
 Then, for every $\left(p,q\right)\in\Z^2$, the induced map $j \colon \mathcal{BC}^{p,q}(C^{\bullet,\bullet}) \hookrightarrow \mathcal{BC}^{p,q}(A^{\bullet,\bullet})$ of complexes induces an injective map $j^*$ in cohomology.
\end{cor}

\begin{proof}
 By Lemma \ref{lem:adjoint-subspace}, one has $\pi^{C^{\bullet,\bullet}} \circ \del^* = \del^* \circ \pi^{C^{\bullet,\bullet}} \colon \dom\left(\del^*\right)^{\bullet,\bullet} \to C^{\bullet-1,\bullet}$ and $\pi^{C^{\bullet,\bullet}} \circ \delbar^* = \delbar^* \circ \pi^{C^{\bullet,\bullet}} \colon \dom\left(\delbar^*\right)^{\bullet,\bullet} \to C^{\bullet,\bullet-1}$, and hence in particular $\del^*\lfloor_{C^{\bullet,\bullet}} = \left(\del\lfloor_{C^{\bullet,\bullet}}\right)^*_{\scalar{\sspace}{\ssspace}_{C^{\bullet,\bullet}}} \colon C^{\bullet,\bullet} \to C^{\bullet-1,\bullet}$
 and $\delbar^*\lfloor_{C^{\bullet,\bullet}} = \left(\delbar\lfloor_{C^{\bullet,\bullet}}\right)^*_{\scalar{\sspace}{\ssspace}_{C^{\bullet,\bullet}}} \colon C^{\bullet,\bullet} \to C^{\bullet,\bullet-1}$.

 Furthermore, it follows that $\pi^{C^{\bullet,\bullet}} \circ \tilde\Delta^{BC} = \tilde\Delta^{BC} \circ \pi^{C^{\bullet,\bullet}} \colon \dom\left(\tilde\Delta^{BC}\right)^{\bullet,\bullet} \to C^{\bullet,\bullet}$. In particular, it follows that
 $$ \pi^{C^{\bullet,\bullet}} \left(\ker\tilde\Delta^{BC}\right) \;=\; \ker\tilde\Delta^{BC}\lfloor_{C^{\bullet,\bullet}} \qquad \text{ and } \qquad \pi^{C^{\bullet,\bullet}} \left(\imm\tilde\Delta^{BC}\right) \;=\; \imm\tilde\Delta^{BC}\lfloor_{C^{\bullet,\bullet}} \;, $$
 and hence one gets the decomposition
 \begin{eqnarray*}
 \dom\left(\tilde\Delta^{BC}\lfloor_{C^{\bullet,\bullet}}\right)^{\bullet,\bullet} &=& \pi^{C^{\bullet,\bullet}}\left(\dom\left(\tilde\Delta^{BC}\right)^{\bullet,\bullet}\right) \;=\; \pi^{C^{\bullet,\bullet}}\left(\ker\tilde\Delta^{BC}\right) + \pi^{C^{\bullet,\bullet}} \left(\imm\tilde\Delta^{BC}\right) \\[5pt]
 &=& \ker\tilde\Delta^{BC} \lfloor_{C^{\bullet,\bullet}} \oplus \imm\tilde\Delta^{BC} \lfloor_{C^{\bullet,\bullet}} \;.
 \end{eqnarray*}
 Hence the hypotheses of Theorem \ref{thm:inj} are satisfied, completing the proof.
\end{proof}

Note that hypothesis \ref{item:thm-inj-3} in Theorem \ref{thm:inj} is satisfied whenever the sub-complex $C^{\bullet,\bullet}$ is finite-dimensional.

\begin{cor}\label{cor:inj-fin-dim}
 Let $A^{\bullet,\bullet}$ be a bounded $\Z^2$-graded vector space with a structure of Hilbert space, with inner product $\scalar{\sspace}{\ssspace}$ such that $\scalar{ A^{p,q} }{ A^{p',q'} } = \{0\}$ for every $(p,q) \neq (p',q')$. Let $\del\colon A^{\bullet,\bullet} \supseteq \dom(\del)^{\bullet,\bullet} \to A^{\bullet+1,\bullet}$ and $\delbar\colon A^{\bullet,\bullet} \supseteq \dom(\delbar)^{\bullet,\bullet} \to A^{\bullet,\bullet+1}$ be densely-defined linear operators yielding a structure $\left(\left(\dom(\del)\cap\dom(\delbar)\right)^{\bullet,\bullet},\, \del,\, \delbar\right)$ of bounded double complex of $\C$-vector spaces.
 Let $j\colon \left(C^{\bullet,\bullet},\, \del,\, \delbar\right) \hookrightarrow \left(\left(\dom(\del)\cap\dom(\delbar)\right)^{\bullet,\bullet},\, \del,\, \delbar\right)$ be a sub-complex.
 Suppose that:
 \begin{enumerate}
  \item\label{item:cor-inj-fin-dim-1} the operator $\tilde\Delta^{BC}_{\scalar{\sspace}{\ssspace}} \in \Hom^{0,0}\left(\dom\left(\tilde\Delta^{BC}_{\scalar{\sspace}{\ssspace}}\right)^{\bullet,\bullet};\, A^{\bullet,\bullet}\right)$ induces the decomposition
  $$ \dom\left(\tilde\Delta^{BC}_{\scalar{\sspace}{\ssspace}}\right)^{\bullet,\bullet} \;=\; \ker \tilde\Delta^{BC}_{\scalar{\sspace}{\ssspace}} \oplus \imm \tilde\Delta^{BC}_{\scalar{\sspace}{\ssspace}} \;;$$
  \item\label{item:cor-inj-fin-dim-2} $C^{\bullet,\bullet}$ is finite-dimensional;
  \item\label{item:cor-inj-fin-dim-3} it holds that
  $$ \del^*_{\scalar{\sspace}{\ssspace}}\lfloor_{C^{\bullet,\bullet}} \;=\; \left(\del\lfloor_{C^{\bullet,\bullet}}\right)^*_{\scalar{\sspace}{\ssspace}_{C^{\bullet,\bullet}}} \colon C^{\bullet,\bullet} \to C^{\bullet-1,\bullet} $$
  and
  $$ \delbar^*_{\scalar{\sspace}{\ssspace}}\lfloor_{C^{\bullet,\bullet}} \;=\; \left(\delbar\lfloor_{C^{\bullet,\bullet}}\right)^*_{\scalar{\sspace}{\ssspace}_{C^{\bullet,\bullet}}} \colon C^{\bullet,\bullet} \to C^{\bullet,\bullet-1} \;. $$
 \end{enumerate}
 Then, for every $\left(p,q\right)\in\Z^2$, the induced map $j \colon \mathcal{BC}^{p,q}(C^{\bullet,\bullet}) \hookrightarrow \mathcal{BC}^{p,q}(A^{\bullet,\bullet})$
 of complexes induces an injective map $j^*$ in cohomology.
\end{cor}

\begin{proof}
 Note that, if $C^{\bullet,\bullet}\subseteq \left(\dom\del\cap\dom\delbar\right)^{\bullet,\bullet}$ is finite-dimensional, as in \ref{item:cor-inj-fin-dim-2}, then the $\C$-linear operators $\del\lfloor_{C^{\bullet,\bullet}} \colon C^{\bullet,\bullet}\to C^{\bullet+1,\bullet}$ and $\delbar\lfloor_{C^{\bullet,\bullet}} \colon C^{\bullet,\bullet} \to C^{\bullet,\bullet+1}$ are continuous, and hence $\dom \left(\del\lfloor_{C^{\bullet,\bullet}}\right)^*_{\scalar{\sspace}{\ssspace}_{C^{\bullet,\bullet}}} = \dom \left(\del^*\lfloor_{C^{\bullet,\bullet}}\right) = C^{\bullet,\bullet}$ and $\dom \left(\delbar\lfloor_{C^{\bullet,\bullet}}\right)^*_{\scalar{\sspace}{\ssspace}_{C^{\bullet,\bullet}}} = \dom \left(\delbar^*\lfloor_{C^{\bullet,\bullet}}\right) = C^{\bullet,\bullet}$. By hypothesis \ref{item:cor-inj-fin-dim-3}, it follows that
 $\tilde\Delta^{BC}\lfloor_{C^{\bullet,\bullet}} = \tilde\Delta^{BC}_{\scalar{\sspace}{\ssspace}_{C^{\bullet,\bullet}}} \in \End^{0,0}\left(C^{\bullet,\bullet}\right)$.
 In particular, $\dom \tilde\Delta^{BC}_{\scalar{\sspace}{\ssspace}_{C^{\bullet,\bullet}}} = \dom\tilde\Delta^{BC}\lfloor_{C^{\bullet,\bullet}} =  C^{\bullet,\bullet}$.

 Hence, in order to apply Theorem \ref{thm:inj}, it suffices to show that, given a finite-dimensional $\C$-vector space $C$ endowed with an inner product $\scalar{\sspace}{\ssspace}$, any self-$\scalar{\sspace}{\ssspace}$-adjoint endomorphism $L\in\Hom(C)$ yields a decomposition
 $$ C \;=\; \ker L \oplus \imm L \;.$$
 Indeed, take $\ker L \subseteq C$ and let $V\subseteq C$ be the $\C$-vector sub-space of $C$ being $\scalar{\sspace}{\ssspace}$-orthogonal to $\ker L$; in particular, $C = \ker L \stackrel{\perp}{\oplus} V$. It suffices to show that $V = \imm L$. Since $L$ is self-$\scalar{\sspace}{\ssspace}$-adjoint, then $\scalar{\imm L}{\ker L}=\{0\}$, and hence $\imm L \subseteq V$. Since $\dim_\C C = \dim_\C \imm L+ \dim_\C \ker L < +\infty$, it follows that $V=\imm L$.
\end{proof}

\begin{rem}\label{rem:other-cohom-inj}
 Obviously, Theorem \ref{thm:inj}, as well as its corollaries, holds, with straightforward modifications, also for the cohomologies associated to the operators $\Delta_{\scalar{\sspace}{\ssspace}} := \left[\de,\, \de^*\right]$, and $\square_{\scalar{\sspace}{\ssspace}} := \left[\del,\, \del^*\right]$, and $\overline\square_{\scalar{\sspace}{\ssspace}} := \left[\delbar,\, \delbar^*\right]$, and $\tilde\Delta^{A}_{\scalar{\sspace}{\ssspace}} := \del\del^*+\delbar\delbar^*+\left(\del\delbar\right)^*\left(\del\delbar\right)+\left(\del\delbar\right)\left(\del\delbar\right)^*+\left(\delbar\del^*\right)^*\left(\delbar\del^*\right)+\left(\delbar\del^*\right)\left(\delbar\del^*\right)^*$.
\end{rem}

\section{Applications}

We are now interested in applying the general results of the previous section to suitable sub-complexes of the double complex $\left(\wedge^{\bullet,\bullet}X,\, \del,\, \delbar\right)$, where $X$ is a compact complex manifold. We are especially interested in the case when $X$ is a solvmanifold.

\subsection{Complexes of PD-type}

Let $\left( A^{\bullet,\bullet} ,\, \del ,\, \delbar \right)$ be a double complex of $\C$-vector spaces. Suppose that $A^{\bullet,\bullet}$ have a structure $\wedge$ of $\C$-algebra being compatible with the $\Z^2$-grading (namely, $A^{p,q}\wedge A^{p',q'}\subseteq A^{p+p',q+q'}$ for every $(p,q),(p',q')\in\Z^2$), and with respect to which $\de := \del + \delbar$ satisfies the Leibniz rule, namely,
$$ \text{for every } a\in \Tot^{\hat a}A^{\bullet,\bullet} \;, \qquad  \left[\de, \, a\wedge\sspace\right] \;=\; \de a \wedge \sspace \;\in\; \End^{\hat a + 1}\left(\Tot^\bullet A^{\bullet,\bullet}\right) \;. $$

Following the notation introduced in \cite[\S2]{kasuya-degeneration} by the second author, $\left( A^{\bullet,\bullet} ,\, \del ,\, \delbar \right)$ is said to be a \emph{bi-differential $\Z^2$-graded algebra of \pd-type} if
\begin{enumerate}
 \item whenever $p<0$ or $q<0$, then $A^{p,q}=\{0\}$, and $H^{0}(\Tot^\bullet A^{\bullet,\bullet}) = \C \left\langle 1 \right\rangle$;
 \item there exists $n\in\N$ such that, whenever $p>n$ or $q>n$, then $A^{p,q} = \{0\}$, and $H^{2n}(\Tot^\bullet A^{\bullet,\bullet}) = \C \left\langle v \right\rangle$; (call $n$ the \emph{\pd-dimension of $A^{\bullet,\bullet}$};)
 \item for every $\left(h,k\right)\in\{0,\ldots,n\}^2$, the bi-$\C$-linear map $A^{h,k} \times A^{n-h,n-k} \to A^{n,n} \stackrel{\simeq}{\to} \C$ induced by $\wedge$ is non-degenerate;
 \item $\de \Tot^0 A^{\bullet,\bullet} = \{0\}$ and $\de \Tot^{2n-1} A^{\bullet,\bullet} = \{0\}$.
\end{enumerate}

Given a bi-differential $\Z^2$-graded algebra $\left(A^{\bullet,\bullet},\, \del,\, \delbar\right)$ of \pd-type, let $\scalar{\sspace}{\ssspace}$ be an inner product on $A^{\bullet,\bullet}$ being compatible with the $\Z^2$-grading, namely, $\scalar{A^{p,q}}{A^{p',q'}}=\{0\}$ whenever $\left(p,q\right) \neq \left(p',q'\right)$, and being compatible with the \pd-type structure, namely, $\scalar{v}{v}=1$.
Define the $\C$-anti-linear map
$$ \bar*_{\scalar{\sspace}{\ssspace}} \colon A^{\bullet_1,\bullet_2} \to A^{n-\bullet_1,n-\bullet_2} \qquad \text{ such that } \qquad \text{for every }\alpha,\beta\in A^{\bullet,\bullet}\;, \quad \alpha \wedge \bar*_{\scalar{\sspace}{\ssspace}}\beta \;=\; \scalar{\alpha}{\beta} \cdot v $$
(as above, we will understand the scalar product $\scalar{\sspace}{\ssspace}$ whenever it is clear from the context).

By considering the Hilbert space given by the $\scalar{\sspace}{\ssspace}$-completion of $A^{\bullet,\bullet}$, one has that the operators
$$ \del^* \;:=\; -\bar*_{\scalar{\sspace}{\ssspace}} \, \del \, \bar*_{\scalar{\sspace}{\ssspace}} \colon A^{\bullet,\bullet} \to A^{\bullet-1,\bullet} \qquad \text{ and } \qquad \delbar^* \;:=\; -\bar*_{\scalar{\sspace}{\ssspace}} \, \delbar \, \bar*_{\scalar{\sspace}{\ssspace}} \colon A^{\bullet,\bullet} \to A^{\bullet,\bullet-1} $$
are in fact the $\scalar{\sspace}{\ssspace}$-adjoint operators $\del^*_{\scalar{\sspace}{\ssspace}}$, respectively $\delbar^*_{\scalar{\sspace}{\ssspace}}$, of $\del\colon A^{\bullet,\bullet}\to A^{\bullet+1,\bullet}$, respectively $\delbar\colon A^{\bullet,\bullet}\to A^{\bullet,\bullet+1}$, and the operator
$$ \de^* \;:=\; -\bar*_{\scalar{\sspace}{\ssspace}} \, \de \, \bar*_{\scalar{\sspace}{\ssspace}} \;=\; \del^* + \delbar^* \colon \Tot^\bullet A^{\bullet,\bullet} \to \Tot^{\bullet-1} A^{\bullet,\bullet} $$
is in fact the $\scalar{\sspace}{\ssspace}$-adjoint operator $\de^*_{\scalar{\sspace}{\ssspace}}$ of $\de := \del+\delbar \colon \Tot^\bullet A^{\bullet,\bullet} \to \Tot^{\bullet+1} A^{\bullet,\bullet}$, \cite[Lemma 2.4]{kasuya-degeneration}.

The following result is an application of Corollary \ref{cor:inj-fin-dim} to the case of bi-differential $\Z^2$-graded algebras of \pd-type.

\begin{prop}\label{prop:inj-pd-type}
 Let $\left(A^{\bullet,\bullet},\, \del,\, \delbar\right)$ be a bi-differential $\Z^2$-graded algebra of \pd-type of \pd-dimension $n$. Let $\scalar{\sspace}{\ssspace}$ be an inner product on $A^{\bullet,\bullet}$ being compatible with the $\Z^2$-grading and with the \pd-type structure. Consider the Hilbert space given by the $\scalar{\sspace}{\ssspace}$-completion of $A^{\bullet,\bullet}$, and suppose that the operator $\tilde\Delta^{BC}_{\scalar{\sspace}{\ssspace}} \in \End^{0,0}\left(A^{\bullet,\bullet}\right)$ induces the decomposition
 $$ A^{\bullet,\bullet} \;=\; \ker \tilde\Delta^{BC}_{\scalar{\sspace}{\ssspace}} \oplus \imm \tilde\Delta^{BC}_{\scalar{\sspace}{\ssspace}} \;.$$
 Let $\left(C^{\bullet,\bullet},\, \del,\, \delbar\right) \hookrightarrow \left(A^{\bullet,\bullet},\, \del,\, \delbar\right)$ be a finite-dimensional sub-complex of $\left(A^{\bullet,\bullet},\, \del,\, \delbar\right)$ having a structure of bi-differential $\Z^2$-graded algebra of \pd-type of \pd-dimension $n$ induced by $A^{\bullet,\bullet}$. Suppose that
 $$ \bar*_{\scalar{\sspace}{\ssspace}} \lfloor_{C^{\bullet,\bullet}}\colon C^{\bullet,\bullet} \to C^{n-\bullet,n-\bullet} \;. $$

 Then, for any $\left(p,q\right)\in\Z^2$, the induced inclusions
 $$ \left(\Tot^\bullet \left(C^{\bullet,\bullet}\right),\, \del+\delbar\right) \hookrightarrow \left(\Tot^\bullet A^{\bullet,\bullet},\, \del+\delbar\right) \;,$$
and 
$$ \left(C^{\bullet,q},\, \del\right) \hookrightarrow \left(A^{\bullet,q},\, \del\right) \;, \qquad \left(C^{p,\bullet},\, \delbar\right) \hookrightarrow \left(A^{p,\bullet},\, \delbar\right) \;,$$
 and
 $$
   \mathcal{BC}^{p,q}(C^{\bullet,\bullet}) \hookrightarrow \mathcal{BC}^{p,q}(A^{\bullet,\bullet}) \;, \qquad
   \mathcal{A}^{p,q}(C^{\bullet,\bullet}) \hookrightarrow \mathcal{A}^{p,q}(A^{\bullet,\bullet})
 $$
 induce injective maps in cohomology.
\end{prop}

\begin{proof}
 Note that also
 $$ A^{\bullet,\bullet} \;=\; \ker \tilde\Delta^{A}_{\scalar{\sspace}{\ssspace}} \oplus \imm \tilde\Delta^{A}_{\scalar{\sspace}{\ssspace}} \;,$$
 since $\bar*_{\scalar{\sspace}{\ssspace}} \, \tilde\Delta^{A}_{\scalar{\sspace}{\ssspace}} = \tilde\Delta^{BC}_{\scalar{\sspace}{\ssspace}} \, \bar*_{\scalar{\sspace}{\ssspace}}$.

By the hypothesis that $\bar*_{\scalar{\sspace}{\ssspace}} \lfloor_{C^{\bullet,\bullet}}\colon C^{\bullet,\bullet} \to C^{n-\bullet,n-\bullet}$, one gets that
$$ \bar*_{\scalar{\sspace}{\ssspace}}\lfloor_{C^{\bullet,\bullet}} \;=\; \bar*_{\scalar{\sspace}{\ssspace}_{C^{\bullet,\bullet}}} $$
(indeed, let $\alpha\in C^{\bullet,\bullet}$; then, for any $\beta\in C^{\bullet,\bullet}$, it holds that $\left(\bar*_{\scalar{\sspace}{\ssspace}_{C^{\bullet,\bullet}}}\alpha-\bar*_{\scalar{\sspace}{\ssspace}}\alpha\right)\wedge\beta=0$; by taking $\beta=\bar*_{\scalar{\sspace}{\ssspace}}\left(\bar*_{\scalar{\sspace}{\ssspace}_{C^{\bullet,\bullet}}}\alpha-\bar*_{\scalar{\sspace}{\ssspace}}\alpha\right) \in C^{\bullet,\bullet}$, one gets hence that $\bar*_{\scalar{\sspace}{\ssspace}_{C^{\bullet,\bullet}}}\alpha=\bar*_{\scalar{\sspace}{\ssspace}}\alpha$).
In particular, it follows that
\begin{eqnarray*}
\del^*_{\scalar{\sspace}{\ssspace}}\lfloor_{C^{\bullet,\bullet}} &=& \left( - \bar*_{\scalar{\sspace}{\ssspace}} \, \del \, \bar*_{\scalar{\sspace}{\ssspace}} \right)\lfloor_{C^{\bullet,\bullet}} \;=\; - \bar*_{\scalar{\sspace}{\ssspace}_{C^{\bullet,\bullet}}} \, \del\lfloor_{C^{\bullet,\bullet}} \, \bar*_{\scalar{\sspace}{\ssspace}_{C^{\bullet,\bullet}}} \\[5pt]
&=& \left(\del\lfloor_{C^{\bullet,\bullet}}\right)^*_{\scalar{\sspace}{\ssspace}_{C^{\bullet,\bullet}}} \colon C^{\bullet,\bullet} \to C^{\bullet-1,\bullet}
\end{eqnarray*}
and
\begin{eqnarray*}
\delbar^*_{\scalar{\sspace}{\ssspace}}\lfloor_{C^{\bullet,\bullet}} &=& \left( - \bar*_{\scalar{\sspace}{\ssspace}} \, \delbar \, \bar*_{\scalar{\sspace}{\ssspace}} \right)\lfloor_{C^{\bullet,\bullet}} \;=\; - \bar*_{\scalar{\sspace}{\ssspace}_{C^{\bullet,\bullet}}} \, \delbar\lfloor_{C^{\bullet,\bullet}} \, \bar*_{\scalar{\sspace}{\ssspace}_{C^{\bullet,\bullet}}} \\[5pt]
&=& \left(\delbar\lfloor_{C^{\bullet,\bullet}}\right)^*_{\scalar{\sspace}{\ssspace}_{C^{\bullet,\bullet}}} \colon C^{\bullet,\bullet} \to C^{\bullet,\bullet-1} \;.
\end{eqnarray*}

Hence Corollary \ref{cor:inj-fin-dim}, see also Remark \ref{rem:other-cohom-inj}, applies.
\end{proof}

\subsection{Compact complex manifolds}

Let $X$ be a compact complex manifold of complex dimension $n$ endowed with a Hermitian metric $g$. (Note that all manifolds are assumed to have no boundary.)

By considering the ($\C$-anti-linear) Hodge-$*$-operator
$$ \bar*_g \colon \wedge^{\bullet_1, \bullet_2}X \to \wedge^{n-\bullet_1, n-\bullet_2}X $$
and the inner product
$$ \scalar{\sspace}{\ssspace} \;:=\; \int_X \sspace \wedge \bar*_g (\ssspace) \;, $$
one gets that the double complex $\left(\wedge^{\bullet,\bullet}X,\, \del,\, \delbar\right)$ has a structure of bi-differential $\Z^2$-graded algebra of \pd-type of \pd-dimension $n$, such that $\scalar{\sspace}{\ssspace}$ is compatible with the $\Z^2$-grading and with the \pd-type structure of $\wedge^{\bullet,\bullet}X$.

The $2^\text{nd}$ order self-$\scalar{\sspace}{\ssspace}$-adjoint elliptic differential operators
$$ \Delta_g \;:=\; \left[\de,\, \de^*\right] \;\in\; \End^0\left(\wedge^\bullet X\otimes\C\right) \;, $$
and
$$
\square_g \;:=\; \left[\del,\, \del^*\right] \;\in\; \End^{0,0}\left(\wedge^{\bullet,\bullet} X\right) \;,
\qquad
\overline\square_g \;:=\; \left[\delbar,\, \delbar^*\right] \;\in\; \End^{0,0}\left(\wedge^{\bullet,\bullet} X\right) \;,
$$
and the $4^\text{th}$ order self-$\scalar{\sspace}{\ssspace}$-adjoint elliptic differential operators,  \cite[Proposition 5]{kodaira-spencer-3}, \cite[\S2.b, \S2.c]{schweitzer},
$$ \tilde \Delta^{BC}_g \;:=\; \left(\del\delbar\right)\left(\del\delbar\right)^*+\left(\del\delbar\right)^*\left(\del\delbar\right)+\left(\delbar^*\del\right)\left(\delbar^*\del\right)^*+\left(\delbar^*\del\right)^*\left(\delbar^*\del\right)+\delbar^*\delbar+\del^*\del \;\in\; \End^{0,0}\left(\wedge^{\bullet, \bullet}X\right) $$
and
$$ \tilde\Delta^{A}_g \;:=\; \del\del^*+\delbar\delbar^*+\left(\del\delbar\right)^*\left(\del\delbar\right)+\left(\del\delbar\right)\left(\del\delbar\right)^*+\left(\delbar\del^*\right)^*\left(\delbar\del^*\right)+\left(\delbar\del^*\right)\left(\delbar\del^*\right)^* \;\in\; \End^{0,0}\left(\wedge^{\bullet, \bullet}X\right) \;, $$
(from now on, the metric $g$ will be understood whenever it is clear from the context,) induce the $\scalar{\sspace}{\ssspace}$-orthogonal decompositions, \cite[page 450]{kodaira},
$$ \wedge^\bullet X\otimes_\R\C \;=\; \ker \Delta \oplus \imm \Delta \;=\; \ker \Delta \oplus \imm\de \oplus \imm\de^* $$
and
\begin{eqnarray*}
\wedge^{\bullet,\bullet}X &=& \ker\square \oplus \imm\square \;=\; \ker\square \oplus \imm\del \oplus \imm\del^* \\[5pt]
&=& \ker\overline\square \oplus \imm\overline\square \;=\; \ker\overline\square \oplus \imm\delbar \oplus \imm\delbar^* \;,
\end{eqnarray*}
and, \cite[Théorème 2.2, \S2.c]{schweitzer},
\begin{eqnarray*}
\wedge^{\bullet,\bullet}X &=& \ker\tilde\Delta^{BC} \oplus \imm\tilde\Delta^{BC} \;=\; \ker\tilde\Delta^{BC} \oplus \imm\del\delbar \oplus \left(\imm\del^* + \imm\delbar^*\right) \\[5pt]
 &=& \ker\tilde\Delta^{A} \oplus \imm\tilde\Delta^{A} \;=\; \ker\tilde\Delta^{A} \oplus \left(\imm\del+\imm\delbar\right) \oplus \imm\left(\del\delbar\right)^* \;.
\end{eqnarray*}

In particular, by arguing as in Lemma \ref{lem:hodge-isom-bc}, it follows that
$$
H^\bullet_{dR}(X;\C) \;:=\; \frac{\ker\de}{\imm\de} \;\simeq\; \ker\Delta \;,
\qquad
H^{\bullet,\bullet}_{\del}(X) \;:=\; \frac{\ker\del}{\imm\del} \;\simeq\; \ker\square \;,
\qquad
H^{\bullet,\bullet}_{\delbar}(X) \;:=\; \frac{\ker\delbar}{\imm\delbar} \;\simeq\; \ker\overline\square \;,
$$
and, \cite[Corollaire 2.3, \S2.c]{schweitzer},
$$
H^{\bullet,\bullet}_{BC}(X) \;:=\; \frac{\ker\del \cap \ker\delbar}{\imm\del\delbar} \;\simeq\; \ker\tilde\Delta^{BC} \;,
\qquad
H^{\bullet,\bullet}_{A}(X) \;:=\; \frac{\ker\del\delbar}{\imm\del+\imm\delbar} \;\simeq\; \ker\tilde\Delta^{A} \;.
$$

Note that $\bar*_g \circ \tilde\Delta^{BC} = \tilde\Delta^{A} \circ \bar*_g$, and hence the Hodge-$*$-operator induces the isomorphism
$$ H^{\bullet,\bullet}_{BC}(X) \stackrel{\simeq}{\to} H^{n-\bullet,n-\bullet}_{A}(X) \;.$$

In particular, by Proposition \ref{prop:inj-pd-type}, one gets straightforwardly the following result, which provides a condition under which the Bott-Chern cohomology of a finite-dimensional sub-complex of $\wedge^{\bullet,\bullet}X$ is a subgroup of $H^{\bullet,\bullet}_{BC}(X)$. Such a result will be applied in the next section with the aim to study the Bott-Chern cohomology of a certain class of solvmanifolds.

\begin{prop}\label{injCo}
 Let $X$ be a compact complex manifold of complex dimension $n$ endowed with a Hermitian metric $g$. Let $\left(C^{\bullet,\bullet},\, \del,\, \delbar\right) \hookrightarrow \left(\wedge^{\bullet,\bullet}X,\, \del,\, \delbar\right)$ be a finite-dimensional sub-complex of $\left(\wedge^{\bullet,\bullet}X,\, \del,\, \delbar\right)$ having a structure of bi-differential $\Z^2$-graded algebra of \pd-type of \pd-dimension $n$ induced by $\wedge^{\bullet,\bullet}X$.
 Suppose that
 $$ \bar*_g\lfloor_{C^{\bullet,\bullet}}\colon C^{\bullet,\bullet} \to C^{n-\bullet,n-\bullet} \;. $$
 Then, for any $\left(p,q\right)\in\Z^2$, the induced inclusions
 $$ \left(\Tot^\bullet \left(C^{\bullet,\bullet}\right),\, \del+\delbar\right) \hookrightarrow \left(\wedge^\bullet X \otimes_\R\C,\, \de\right) \;, $$
and
$$ \left(C^{\bullet,q},\, \del\right) \hookrightarrow \left(\wedge^{\bullet,q}X,\, \del\right) \;, \qquad \left(C^{p,\bullet},\, \delbar\right) \hookrightarrow \left(\wedge^{p,\bullet}X,\, \delbar\right) \;,$$
 and
 $$
   \mathcal{BC}^{p,q}(C^{\bullet,\bullet})
   \hookrightarrow
   \mathcal{BC}^{p,q}(\wedge^{\bullet,\bullet}X)
\;, \qquad
   \mathcal{A}^{p,q}(C^{\bullet,\bullet})
   \hookrightarrow
   \mathcal{A}^{p,q}(\wedge^{\bullet,\bullet}X)
 $$
 induce injective maps in cohomology.
\end{prop}

\begin{proof}
 The proof follows straightforwardly by \cite[Théorème 2.2, \S2.c]{schweitzer} and \cite[page 450]{kodaira}, and by Proposition \ref{prop:inj-pd-type}.
\end{proof}

\begin{rem}
 By applying Corollary \ref{cor:inj-closed} to the $\scalar{\sspace}{\ssspace}$-completion of $\wedge^{\bullet,\bullet}X$, the same conclusion of Proposition \ref{injCo} holds true for a (possibly non-finite-dimensional) closed sub-complex $\left(C^{\bullet,\bullet},\, \del,\, \delbar\right) \hookrightarrow \left(\wedge^{\bullet,\bullet}X,\, \del,\, \delbar\right)$ such that $\pi^{C^{\bullet,\bullet}} \circ \del = \del \circ \pi^{C^{\bullet,\bullet}} \colon \wedge^{\bullet,\bullet}X \to C^{\bullet,\bullet}$ and $\pi^{C^{\bullet,\bullet}} \circ \delbar = \delbar \circ \pi^{C^{\bullet,\bullet}} \colon \wedge^{\bullet,\bullet}X \to C^{\bullet,\bullet}$.
\end{rem}

\medskip

In order to study cohomologies of solvmanifolds, we need also the following result.

To simplify the notation, we say that a sub-complex $\left(C^{\bullet,\bullet},\, \del,\, \delbar\right) \hookrightarrow \left(\wedge^{\bullet,\bullet}X,\, \del,\, \delbar\right)$ \emph{suffices in computing the de Rham, respectively conjugate Dolbeault, respectively Dolbeault, respectively Bott-Chern, respectively Aeppli cohomology of $X$} if the induced inclusion
$$
\left(\Tot^\bullet C^{\bullet,\bullet},\, \del+\delbar\right) \hookrightarrow \left(\wedge^\bullet X \otimes_\R\C,\, \de\right) \;, $$
respectively, for any $q\in\N$,
$$\left(C^{\bullet,q},\, \del\right) \hookrightarrow \left(\wedge^{\bullet,q},\, \del\right) \;, $$
respectively, for any $p\in\N$,
$$ \left(C^{p,\bullet},\, \delbar\right) \hookrightarrow \left(\wedge^{p,\bullet},\, \delbar\right) \;,$$
respectively, for any $(p,q)\in\Z^2$,
$$
  \mathcal{BC}^{p,q}(C^{\bullet,\bullet})
  \hookrightarrow
  \mathcal{BC}^{p,q}(\wedge^{\bullet,\bullet}X)
$$
respectively, for any $(p,q)\in\Z^2$,
$$
  \mathcal{A}^{p,q}(C^{\bullet,\bullet})
  \hookrightarrow
  \mathcal{A}^{p,q}(\wedge^{\bullet,\bullet}X)
$$
is a quasi-isomorphism.

\begin{prop}\label{subiso}
 Let $X$ be a compact complex manifold of complex dimension $n$ endowed with a Hermitian metric $g$. Let $\left(C^{\bullet,\bullet},\, \del,\, \delbar\right) \hookrightarrow \left(\wedge^{\bullet,\bullet}X,\, \del,\, \delbar\right)$ be a finite-dimensional sub-complex of $\left(\wedge^{\bullet,\bullet}X,\, \del,\, \delbar\right)$ having a structure of bi-differential $\Z^2$-graded algebra of \pd-type of \pd-dimension $n$ induced by $\wedge^{\bullet,\bullet}X$ and such that
 $$ \bar*_g\lfloor_{C^{\bullet,\bullet}} \colon C^{\bullet,\bullet} \to C^{n-\bullet,n-\bullet} \;. $$
 Let $\left(B^{\bullet,\bullet},\, \del,\, \delbar\right) \hookrightarrow \left(C^{\bullet,\bullet},\, \del,\, \delbar\right)$ be a sub-complex of $\left(C^{\bullet,\bullet},\, \del,\, \delbar\right)$ having a structure of bi-differential $\Z^2$-graded algebra of \pd-type of \pd-dimension $n$ induced by $C^{\bullet,\bullet}$ and such that
 $$ \bar*_g\lfloor_{B^{\bullet,\bullet}}\colon B^{\bullet,\bullet} \to B^{n-\bullet,n-\bullet} \;. $$
 If $\left(B^{\bullet,\bullet},\, \del,\, \delbar\right)$ suffices in computing the cohomologies of $X$, then also $\left(C^{\bullet,\bullet},\, \del,\, \delbar\right)$ suffices in computing the corresponding cohomologies of $X$.
\end{prop}

\begin{proof}
 By Proposition \ref{prop:inj-pd-type} and Proposition \ref{injCo}, both the inclusions $B^{\bullet,\bullet}\hookrightarrow C^{\bullet,\bullet}$ and $C^{\bullet,\bullet}\hookrightarrow \wedge^{\bullet,\bullet}X$ induce injective maps in cohomology, whose composition is an isomorphism by the hypothesis.
\end{proof}

\subsection{Complex nilmanifolds}

Let $X = \left. \Gamma \right\backslash G$ be a \emph{solvmanifold} (respectively, a \emph{nilmanifold}), namely, a compact quotient of a connected simply-connected solvable (respectively, nilpotent) Lie group $G$ by a co-compact discrete subgroup $\Gamma$, endowed with a $G$-left-invariant (almost-)complex structure $J$.
We recall that a solvmanifold is called \emph{completely-solvable} if, for any $g\in G$, all the eigenvalues of $\Ad_g:=\de\left(\psi_g\right)_e\in\Aut(\g)$ are real, equivalently, for any $X\in\g$, all the eigenvalues of $\ad_X:=\left[X,\sspace\right]\in\End(\g)$ are real, where $\psi\colon G \ni g \mapsto \left( \psi_g \colon h \mapsto g\,h\,g^{-1} \right) \in \Aut(G)$ and $e$ is the identity element of $G$.

Recall that, by J. Milnor's Lemma \cite[Lemma 6.2]{milnor}, $G$ is unimodular (that is, ${\det} (\Ad _{g})=1 $ for any $g\in G$), and hence, in particular, there exists a $G$-bi-invariant volume form $\eta$ on $X$ such that $\int_X\eta=1$. Therefore, consider the \emph{F.~A. Belgun symmetrization map} in \cite[Theorem 7]{belgun}, namely,
$$ \mu\colon \wedge^\bullet X \otimes_\R \C \to \wedge^\bullet\dualee{\mathfrak{g} \otimes_\R\C} \;,\qquad \mu(\alpha)\;:=\;\int_X \alpha\lfloor_x \, \eta(x) \;.$$
Note, \cite[Theorem 7]{belgun}, that $\mu$ commutes with $\de$ and with $J$, and hence also with $\del$ and $\delbar$, and that $\mu\lfloor_{\wedge^\bullet\dualee{\mathfrak{g}\otimes_\R\C}} = \id_{\wedge^\bullet\dualee{\mathfrak{g}\otimes_\R\C}}$.

\begin{lem}\label{wedco}
 Let $\solvmfd$ be a solvmanifold, and consider the F.~A. Belgun symmetrization map $\mu\colon \wedge^\bullet X \otimes_\R \C \to \wedge^\bullet\dualee{\mathfrak{g} \otimes_\R\C}$ in \cite[Theorem 7]{belgun}.
 For a $G$-left-invariant differential form $\theta$ on $\solvmfd$ and for a differential form $\omega$ on $\solvmfd$, we have
\[\mu(\theta\wedge\omega)=\theta\wedge\mu(\omega).\]
\end{lem}
\begin{proof}
Suppose that $\theta$ is a $G$-left-invariant $1$-form on $\solvmfd$. Let $\omega$ be a $p$-form on $\solvmfd$.
Then for $X_{1},\dots,X_{p+1}\in \g$, since $\theta(X_{j})$ is constant for every $j\in\{1,\ldots,p+1\}$, we have
\begin{eqnarray*}
\mu(\theta\wedge\omega)(X_{1},\dots, X_{p+1}) &=& \int_{\solvmfd} \sum_{\sigma\in\perm{p+1}} \theta_{x}\left(X_{\sigma(1)}\right)\cdot \omega\left(X_{\sigma(2)},\dots, X_{\sigma(p+1)}\right) \, \eta(x) \\[5pt]
&=& \sum_{\sigma\in\perm{p+1}} \theta\left(X_{\sigma(1)}\right) \cdot \int_{\solvmfd}\omega_{x}\left(X_{\sigma(2)},
\dots, X_{\sigma(p+1)}\right) \, \eta(x) \\[5pt]
&=& \left(\theta\wedge \mu (\omega)\right)\left(X_{1},\dots, X_{p+1}\right) \;,
\end{eqnarray*}
where $\perm{p+1}$ is the set of permutations of $p+1$ elements.
Hence, in this case, the lemma holds.
We can easily check that the lemma holds in the general case.
\end{proof}

\begin{lem}[{see \cite[Proposition 5.4]{angella-tomassini-zhang}}]\label{lem:cond-3-surj-solvmfd}
 Let $X = \left. \Gamma \right\backslash G$ be a completely-solvable solvmanifold endowed with a $G$-left-invariant complex structure $J$. Consider the sub-complex
 $$ j\colon \left(\wedge^\bullet\dualee{\mathfrak{g}\otimes_\R\C},\, \de\right) \hookrightarrow \left(\wedge^{\bullet}X\otimes_\R \C,\, \de\right) \;, $$
 which is a quasi-isomorphism by A. Hattori's theorem \cite[Corollary 4.2]{hattori}. The induced map
 \begin{eqnarray*}
 \lefteqn{ j\colon \frac{\ker\left(\de\colon \wedge^{p+q}\dualee{\mathfrak{g}\otimes_\R\C} \to \wedge^{p+q+1}\dualee{\mathfrak{g}\otimes_\R\C} \right) \cap \wedge^{p,q}\dualee{\mathfrak{g}\otimes_\R\C}}{\imm\left(\de\colon \wedge^{p+q-1}\dualee{\mathfrak{g}\otimes_\R\C} \to \wedge^{p+q}\dualee{\mathfrak{g}\otimes_\R\C}\right)} } \\[5pt]
 &\to& \frac{\ker\left(\de\colon \wedge^{p+q}X\otimes_\R \C \to \wedge^{p+q+1}X \otimes_\R \C \right) \cap \wedge^{p,q}X}{\imm\left(\de\colon \wedge^{p+q-1}X \otimes_\R \C \to \wedge^{p+q}X\otimes_\R \C \right)}
 \end{eqnarray*}
 is an isomorphism.
\end{lem}

\begin{proof}
For the sake of completeness, we recall here the argument of the proof (note that the statement holds, more in general, in the almost-complex setting).

The F.~A. Belgun symmetrization map $\mu\colon \wedge^\bullet X \otimes_\R \C \to \wedge^\bullet\dualee{\mathfrak{g}\otimes_\R\C}$ induces the map
\begin{eqnarray*} 
 \lefteqn{ \mu \colon \frac{\ker\left(\de\colon \wedge^{p+q}X\otimes_\R \C \to \wedge^{p+q+1}X \otimes_\R \C \right) \cap \wedge^{p,q}X}{\imm\left(\de\colon \wedge^{p+q-1}X \otimes_\R \C \to \wedge^{p+q}X\otimes_\R \C \right)} } \\[5pt]
 &\to& \frac{\ker\left(\de\colon \wedge^{p+q}\dualee{\mathfrak{g}\otimes_\R\C} \to \wedge^{p+q+1}\dualee{\mathfrak{g}\otimes_\R\C} \right) \cap \wedge^{p,q}\dualee{\mathfrak{g}\otimes_\R\C}}{\imm\left(\de\colon \wedge^{p+q-1}\dualee{\mathfrak{g}\otimes_\R\C} \to \wedge^{p+q}\dualee{\mathfrak{g}}\otimes_\R\C\right)} \;.
\end{eqnarray*}

Hence, one gets the commutative diagram
$$
\xymatrix{
\frac{\ker\left(\de\colon \wedge^{p+q}\dualee{\mathfrak{g}\otimes_\R\C} \to \wedge^{p+q+1}\dualee{\mathfrak{g}\otimes_\R\C} \right) \cap \wedge^{p,q}\dualee{\mathfrak{g}\otimes_\R\C}}{\imm\left(\de\colon \wedge^{p+q-1}\dualee{\mathfrak{g}\otimes_\R\C} \to \wedge^{p+q}\dualee{\mathfrak{g}\otimes_\R\C}\right)}
\ar[d]^{j} \ar@/_8.5pc/[dd]_{\id}
\\
\frac{\ker\left(\de\colon \wedge^{p+q}X\otimes_\R \C \to \wedge^{p+q+1}X \otimes_\R \C \right) \cap \wedge^{p,q}X}{\imm\left(\de\colon \wedge^{p+q-1}X \otimes_\R \C \to \wedge^{p+q}X\otimes_\R \C \right)}
\ar[d]^\mu
\\
\frac{\ker\left(\de\colon \wedge^{p+q}\dualee{\mathfrak{g}\otimes_\R\C} \to \wedge^{p+q+1}\dualee{\mathfrak{g}\otimes_\R\C} \right) \cap \wedge^{p,q}\dualee{\mathfrak{g}\otimes_\R\C}}{\imm\left(\de\colon \wedge^{p+q-1}\dualee{\mathfrak{g}\otimes_\R\C} \to \wedge^{p+q}\dualee{\mathfrak{g}\otimes_\R\C}\right)}
} \;,
$$
from which one gets that $j$ is injective, and that $\mu$ is surjective.

Moreover, since $j\colon \left(\wedge^\bullet\dualee{\mathfrak{g}\otimes_\R\C},\, \de\right) \hookrightarrow \left(\wedge^\bullet X\otimes_\R\C,\, \de\right)$ is a quasi-isomorphism by A. Hattori's theorem \cite[Theorem 4.2]{hattori}, one gets that $\mu\colon H^\bullet_{dR}(X;\C) \to H^\bullet\left(\wedge^\bullet\dualee{\mathfrak{g}\otimes_\R\C},\, \de\right)$ is in fact the identity map, and hence
\begin{eqnarray*}
\lefteqn{ \mu \colon
\frac{\ker\left(\de\colon \wedge^{p+q}X\otimes_\R \C \to \wedge^{p+q+1}X \otimes_\R \C \right) \cap \wedge^{p,q}X}{\imm\left(\de\colon \wedge^{p+q-1}X \otimes_\R \C \to \wedge^{p+q}X\otimes_\R \C \right)} } \\[5pt]
&\to&
\frac{\ker\left(\de\colon \wedge^{p+q}\dualee{\mathfrak{g}\otimes_\R\C} \to \wedge^{p+q+1}\dualee{\mathfrak{g}\otimes_\R\C} \right) \cap \wedge^{p,q}\dualee{\mathfrak{g}\otimes_\R\C}}{\imm\left(\de\colon \wedge^{p+q-1}\dualee{\mathfrak{g}\otimes_\R\C} \to \wedge^{p+q}\dualee{\mathfrak{g}\otimes_\R\C}\right)}
\end{eqnarray*}
is also injective.

Since $X$ is compact, the dimension of $H^\bullet_{dR}(X;\C)$ is finite, and hence $\mu$ is in fact an isomorphism.
\end{proof}

As an application of Theorem \ref{surjBC} and Proposition \ref{injCo}, one recovers the following results, concerning the Bott-Chern cohomology of nilmanifolds. (We refer to \cite{wang, nakamura, barberis-dottimiatello-miatello, andrada-barberis-dottimiatello, cordero-fernandez-gray-ugarte, console-fino, rollenske, salamon} for definitions and notation.)

\begin{cor}[{\cite[Theorem 3.8]{angella-1}}]
 Let $X = \left. \Gamma \right\backslash G$ be a nilmanifold endowed with a $G$-left-invariant complex structure $J$, and denote the Lie algebra naturally associated to $G$ by $\mathfrak{g}$.
 Suppose that one of the following conditions holds:
  \begin{itemize}
    \item $X$ is complex parallelizable;
    \item $J$ is an Abelian complex structure;
    \item $J$ is a nilpotent complex structure;
    \item $J$ is a rational complex structure;
    \item $\mathfrak{g}$ admits a torus-bundle series compatible with $J$ and with the rational structure induced by $\Gamma$;
    \item $\dim_\R\mathfrak{g}=6$ and $\mathfrak{g}$ is not isomorphic to $\mathfrak{h}_7:=\left(0^3,\, 12,\, 13,\, 23\right)$.
  \end{itemize}
 Then the inclusion $j\colon \left(\wedge^{\bullet,\bullet}\dualee{\mathfrak{g}\otimes_\R\C},\, \del,\, \delbar\right) \hookrightarrow \left(\wedge^{\bullet,\bullet}X,\, \del,\, \delbar\right)$ induces the isomorphisms
  $$ H^{\bullet,\bullet}_{BC}(X) \;\simeq\; \frac{\ker\left(\de\colon \wedge^{\bullet,\bullet}\dualee{\mathfrak{g}\otimes_\R\C}\to \wedge^{\bullet+\bullet+1}\dualee{\mathfrak{g}\otimes_\R\C}\right)}{\imm\left(\del\delbar\colon \wedge^{\bullet-1,\bullet-1}\dualee{\mathfrak{g}\otimes_\R\C}\to \wedge^{\bullet,\bullet}\dualee{\mathfrak{g}\otimes_\R\C}\right)} $$
 and
  $$ H^{\bullet,\bullet}_{A}(X) \;\simeq\; \frac{\ker\left(\del\delbar\colon \wedge^{\bullet,\bullet}\dualee{\mathfrak{g}\otimes_\R\C}\to \wedge^{\bullet+1,\bullet+1}\dualee{\mathfrak{g}\otimes_\R\C}\right)}{\imm\left(\del\colon \wedge^{\bullet-1,\bullet}\dualee{\mathfrak{g}\otimes_\R\C}\to \wedge^{\bullet,\bullet}\dualee{\mathfrak{g}\otimes_\R\C}\right) + \imm\left(\delbar\colon \wedge^{\bullet,\bullet-1}\dualee{\mathfrak{g}\otimes_\R\C}\to \wedge^{\bullet,\bullet}\dualee{\mathfrak{g}\otimes_\R\C}\right)} \;. $$
\end{cor}

\begin{proof}
 Choose a $G$-left-invariant Hermitian metric $g$ on $X$.
 The sub-complex $\left(\wedge^{\bullet,\bullet}\dualee{\mathfrak{g}\otimes_\R\C},\, \del,\, \delbar\right)$ being finite-dimensional, the induced maps in Bott-Chern, respectively Aeppli cohomologies are injective by Proposition \ref{injCo}.

 Under the hypothesis, by \cite[Theorem 1]{sakane}, \cite[Main Theorem]{cordero-fernandez-gray-ugarte}, \cite[Theorem 2, Remark 4]{console-fino}, \cite[Theorem 1.10]{rollenske}, and \cite[Corollary 3.10]{rollenske-survey}, one has that, for any fixed $p\in\N$, the induced map
 $$ j\colon \left(\wedge^{p,\bullet}\dualee{\mathfrak{g}\otimes_\R\C},\, \delbar\right) \hookrightarrow \left(\wedge^{p,\bullet}X,\, \delbar\right) $$
 is a quasi-isomorphism. By conjugation, one has also that, for any fixed $q\in\N$, the induced map
 $$ j\colon \left(\wedge^{\bullet,q}\dualee{\mathfrak{g}\otimes_\R\C},\, \del\right) \hookrightarrow \left(\wedge^{\bullet,q}X,\, \del\right) $$
 is a quasi-isomorphism.
 Lastly, condition \ref{item:thm-surg-hp-3} in Theorem \ref{surjBC} is satisfied by Lemma \ref{lem:cond-3-surj-solvmfd}. Hence, by Theorem \ref{surjBC}, the induced map in Bott-Chern cohomology is surjective.

 As regards Aeppli cohomologies, it suffices to note that the Hodge-$*$-operator $\bar*_g$ induces the isomorphisms $H^{\bullet,\bullet}_{BC}(X) \stackrel{\simeq}{\to} H^{n-\bullet,n-\bullet}_{A}(X)$ and $\frac{\ker\de\lfloor_{\wedge^{\bullet,\bullet}\dualee{\mathfrak{g}\otimes_\R\C}}}{\imm\del\delbar} \stackrel{\simeq}{\to} \frac{\ker\del\delbar\lfloor_{\wedge^{n-\bullet,n-\bullet}\dualee{\mathfrak{g}\otimes_\R\C}}}{\imm\del + \imm\delbar}$, where $n$ is the complex dimension of $X$.
\end{proof}

The previous result can be used to compute the cohomology of the left-invariant complex structures classified by M. Ceballos, A. Otal, L. Ugarte, and R. Villacampa in \cite{ceballos-otal-ugarte-villacampa}, as in \cite{angella-franzini-rossi} and \cite{latorre-ugarte-villacampa}.

\subsection{Complex solvmanifolds}\label{subsec:solvmfd}

Let $G$ be a connected simply-connected $n$-dimensional solvable Lie group admitting a discrete co-compact subgroup $\Gamma$, and denote by $\g$ the (solvable) Lie algebra of $G$. Set $\g_\C := \g\otimes_\R\C$.

Consider the \emph{adjoint action}
$$ \ad \colon \g \to \gl(\g) \;, \qquad \ad_X \;:=\; \left[X,\sspace\right] \;;$$
by denoting by $\Der(\g):=\left\{D\in\gl(\g) \st \forall X\in\g,\; \left[D,\ad_X\right]=\ad_{DX}\right\}$ the $\R$-vector space of {\em derivations} on $\g$, one has that $\ad(\g)\subseteq\Der(\g)$. One has that every derivation $\ad_X$, for $X\in\g$, admits a unique {\em Jordan decomposition}, see, {\itshape e.g.} \cite[I\hspace{-.1em}I.1.10]{DER}, namely,
$$ \ad_X \;=\; \left(\ad_X\right)_{\mathrm{s}}+\left(\ad_X\right)_{\mathrm{n}} \;, $$
where $\left(\ad_X\right)_{\mathrm{s}}\in\gl(\g)$ is \emph{semi-simple} (that is, each $\left(\ad_X\right)_{\mathrm{s}}$-invariant sub-space of $\g$ admits an $\left(\ad_X\right)_{\mathrm{s}}$-invariant complementary sub-space in $\g$), and $\left(\ad_X\right)_{\mathrm{n}}\in\gl(\g)$ is \emph{nilpotent} (that is, there exists $N\in\N$ such that $\left(\ad_X\right)_{\mathrm{n}}^N=0$).

Let $\n$ be the {\em nilradical} of $\g$, that is, the maximal nilpotent ideal in $\g$.
Since $\g$ is solvable, there exists an $\R$-vector sub-space $V$ (which is not necessarily a Lie algebra) of $\g$ so that
\begin{inparaenum}[\itshape (i)]
\item $\g=V\oplus \n$ as the direct sum of $\R$-vector spaces, and,
\item for any  $A,B\in V$, it holds that $\left(\ad_A\right)_{\mathrm{s}}(B)=0$,
\end{inparaenum}
see, {\itshape e.g.} \cite[Proposition I\hspace{-.1em}I\hspace{-.1em}I.1.1] {DER}.
Hence, one can define the  map
$$ \ad_{\mathrm{s}}\colon \g\to \Der(\g) \;, \qquad \g = V \oplus \n \ni \left(A,X\right) \mapsto \left(\ad_{\mathrm{s}}\right)_{A+X} := (\ad_{A})_{\mathrm{s}} \in \Der(\g) \;. $$
Moreover, one has that
\begin{inparaenum}[\itshape (i)]\setcounter{enumi}{2}
\item $\left[\ad_{\mathrm{s}}(\g), \ad_{\mathrm{s}}(\g)\right]=\{0\}$, and
\item $\ad_{\mathrm{s}} \colon \g \to \gl(\g)$ is $\R$-linear,
\end{inparaenum}
see, {\itshape e.g.} \cite[Proposition {{I\hspace{-.1em}I\hspace{-.1em}I}.1.1}] {DER}.

Since we have $[\g,\g]\subseteq \n$, see, {\itshape e.g.} \cite[I\hspace{-.1em}I.1.9]{DER}, and $\ad_{\mathrm{s}}(\n)=\{0\}$, the  map $\ad_{\mathrm{s}}\colon \g\to \gl(\g)$ is a representation of $\g$, whose image $\ad_{\mathrm{s}}(\g)$ is Abelian and consists of semi-simple elements.
Hence, denote by
$$ \Ad_{\mathrm{s}}\colon G\to {\Aut}(\g) \;, \qquad \text{respectively } \Ad_{\mathrm{s}}\colon G\to {\Aut}\left(\g_\C\right) \;, $$
the unique representation which lifts ${\ad}_{\mathrm{s}}\colon \g \to \gl(\g)$, see, {\itshape e.g.} \cite[Theorem 3.27]{warner}, respectively the natural $\C$-linear extension.

\medskip
The following arguments on characters of $G$ are very useful.
For $\alpha\in \Hom\left(G;\C^{\ast}\right)$, since we have $\alpha(g_{1}g_{2})=\alpha(g_{1})\alpha(g_{2})$ for any $g_{1}, g_{2}$, we can easily check that $\frac{\de\alpha}{\alpha}$ is $G$-left-invariant.
For a  $G$-left-invariant differential form $\omega$, we have 
\[\de(\alpha\omega)=\de\alpha\wedge \omega +\alpha \de\omega=\alpha\left(\frac{\de\alpha}{\alpha}\wedge \omega+ \de\omega\right)
\]
and hence $\de(\alpha\omega)$ is also a product of  $\alpha$ and a  $G$-left-invariant differential form.

Let $T$ be  the Zariski-closure of ${\Ad}_{\mathrm{s}}(G)$ in ${\Aut}(\g_\C)$.
Denote by $\Char(T):=\Hom(T;\C^*)$ the set of all $1$-dimensional algebraic group representations of $T$.
Set 
\[{\mathcal C}_{\Gamma} \;:=\; \left\{ \beta\circ \Ad_{\mathrm{s}} \in \Hom\left(G;\C^{\ast}\right) \st \beta\in  \Char(T),\; \left(\beta\circ \Ad_{\mathrm{s}}\right)\lfloor_\Gamma=1 \right\} \;. \]
By the above arguments on characters of $G$,
we have the differential graded sub-algebra
\[\bigoplus _{\alpha\in {\mathcal C}_{\Gamma}}\alpha\cdot \wedge^\bullet \g^{\ast}_\C
\]
of $\wedge^{\bullet}\solvmfd \otimes_\R \C$.
(Note that we have used left-translations on $G$ to identify the elements of $\wedge^\bullet\duale{\g_\C}$ with the $G$-left-invariant complex forms in $\wedge^{\bullet} \solvmfd \otimes_\R\C$, namely, the complex forms being invariant for the action of the Lie group $G$ on $\solvmfd$ given by left-translations.)
By $\Ad_{\mathrm{s}}(G) \subseteq \Aut(\g_\C)$ we have the $\Ad_{\mathrm{s}}(G)$-action on the differential graded algebra $\bigoplus_{\alpha\in {\mathcal C}_{\Gamma}}\alpha\cdot \wedge^\bullet \g_\C^{\ast}$.
We denote by $A^{\bullet}_{\Gamma}$ the space consisting of the $\Ad_{\mathrm{s}}(G)$-invariant elements of $\bigoplus _{\alpha\in {\mathcal C}_{\Gamma}}\alpha\cdot \wedge^\bullet \g^{\ast}_\C$, namely,
\begin{equation}\label{eq:def-a1}
 A^{\bullet}_{\Gamma} \;:=\; \left\{ \varphi \in \bigoplus _{\alpha\in {\mathcal C}_{\Gamma}}\alpha\cdot \wedge^\bullet \g^{\ast}_\C \st \left(\Ad_{\mathrm{s}}\right)_g(\varphi)=\varphi \text{ for every } g\in G \right\} \;. 
\end{equation}
Since the action commutes with the structure of the  differential graded algebra, $A^{\bullet}_{\Gamma}$ is also a  differential graded algebra.
Now we consider the inclusion
$$ A^{\bullet}_{\Gamma} \subseteq \wedge^\bullet \solvmfd \otimes_\R \C $$
of differential graded algebras.
We have the following result.

\begin{thm}[{\cite[Corollary 7.6]{kasuya-jdg}}]\label{isoc}
Let $\solvmfd$ be a solvmanifold, and consider $A^\bullet_\Gamma$ as defined in \eqref{eq:def-a1}.
Then the inclusion
\[ \left( A^{\bullet}_\Gamma,\, \de \right) \hookrightarrow \left(\wedge^\bullet \solvmfd \otimes_\R \C,\, \de\right) \]
of differential graded algebras induces an isomorphism in cohomology.
\end{thm}

Note that ${\Ad}_{\mathrm{s}}(G)\subseteq\Aut(\g_\C)$ consists of simultaneously diagonalizable elements.
Let $\left\{ X_{1},\ldots ,X_{n} \right\}$ be a basis of $\g_{\C}$ with respect to which
$$ \Ad_{\mathrm{s}} \;=\; \diag\left(\alpha_1, \ldots, \alpha_n\right) \colon G \to \Aut(\g_\C) $$
for some characters
$$ \alpha_1 \in \Hom(G;\C^*) , \ldots , \alpha_n \in \Hom(G;\C^*) \;. $$
Let $\left\{ x_{1},\ldots,x_{n}\right\}$ be the dual basis of $\duale{\g_\C}$ of $\left\{X_{1},\ldots ,X_{n}\right\}$.
For the basis $\left\{ x_{i_{1}}\wedge \dots \wedge x_{i_{p}} \right\}_{1\le i_{1}<i_{2}<\dots <i_{p}\le n}$ of $\wedge^\bullet \g^{\ast}_\C$, for $\alpha\in  {\mathcal C}_{\Gamma}$, we have
\[ \left({\Ad}_{\mathrm{s}}\right)_{g} \left(\alpha \, x_{i_{1}}\wedge \dots \wedge x_{i_{p}} \right) \;=\; \alpha(g) \, \alpha_{i_{1}\cdots i_{p}}^{-1}(g) \, \alpha \, x_{i_{1}}\wedge \dots \wedge x_{i_{p}} \;, \]
where we have shortened $\alpha_{i_{1}\cdots i_{p}} := \alpha_{i_{1}}\cdot \cdots \cdot \alpha_{i_{p}} \in \Hom\left(G;\C^*\right)$. 
Then the basis
$$ \left\{ \alpha\, x_{i_{1}}\wedge \dots \wedge x_{i_{p}} \;\middle\vert\; 1\le i_{1}<i_{2}<\dots <i_{p}\le n \text{ and } \alpha\in {\mathcal C}_{\Gamma} \right\} $$
of $\bigoplus_{\alpha\in {\mathcal C}_{\Gamma}}\alpha\cdot \wedge^\bullet \g_\C^{\ast}$ diagonalizes the $\Ad_{\mathrm{s}}(G)$-action, 
and $\alpha\, x_{i_{1}}\wedge \dots \wedge x_{i_{p}} \in A^{\bullet}_\Gamma$ if and only if $\alpha= \alpha_{i_{1}\cdots i_{p}}$ and $\alpha_{i_{1}\cdots i_{p}}\lfloor_\Gamma = 1$.
Hence the differential graded algebra $ A^{\bullet}_\Gamma$ is written as
\begin{equation}\label{eq:def-a}
A^{p}_\Gamma
\;=\;
\C\left\langle \alpha_{i_{1}\cdots i_{p}}\, x_{i_{1}}\wedge \dots \wedge x_{i_{p}} \;\middle\vert\; 1\le i_{1}<i_{2}<\dots <i_{p}\le n \; \text{ such that } \alpha_{i_{1}\cdots i_{p}}\lfloor_\Gamma = 1\right\rangle \;.
\end{equation}

In fact, the following result holds.

\begin{thm}\label{p,q-co}
Let $\solvmfd$ be a solvmanifold.
Let $\left\{ X_{1},\ldots ,X_{n} \right\}$ be a basis of the $\C$-vector space $\g_{\C}$ with respect to which $\Ad_{\mathrm{s}} = \diag\left(\alpha_1, \ldots, \alpha_n\right)$ for some characters $\alpha_1 , \ldots , \alpha_n \in \Hom(G;\C^*)$.
Consider the finite set of characters
\[{\mathcal A}_\Gamma \;:=\; \left\{ 
\alpha_{i_{1}\cdots i_{p}} \in \Hom(G;\C^*) \st 1\le i_{1}<i_{2}<\dots <i_{p}\le n \; \text{ such that } \alpha_{i_{1}\cdots i_{p}}\lfloor_\Gamma = 1
\right\} \;.\]
Then the sub-complex
$$ \iota\colon \left(\bigoplus_{\alpha\in {\mathcal A}_\Gamma} \alpha\cdot \wedge^\bullet\duale{\g_{\C}} ,\, \de\right) \hookrightarrow \left(\wedge^{\bullet}\solvmfd\otimes_\R\C ,\, \de\right) $$
induces an isomorphism in cohomology.

Suppose furthermore that $G$ is endowed with a $G$-left-invariant complex structure.
Consider the bi-graded $\C$-vector sub-space
$$ \iota \colon \bigoplus_{\alpha\in {\mathcal A}_\Gamma} \alpha\cdot \wedge^{\bullet,\bullet} \duale{\g_{\C}} \hookrightarrow \wedge^{\bullet,\bullet}\solvmfd \;; $$
then $\iota$ induces, for any $(p,q)\in\Z^2$, the isomorphism
\[
\iota^* \colon \frac{\ker \de\lfloor_{\bigoplus_{\alpha\in {\mathcal A}_\Gamma} \alpha\cdot \wedge^{p,q} \duale{\g_{\C}}}}{\de\left(\bigoplus_{\alpha\in {\mathcal A}_\Gamma} \alpha\cdot \wedge^{p+q-1}\duale{\g_{\C}}\right)} \stackrel{\simeq}{\longrightarrow} \frac{\ker\de\lfloor_{\wedge^{p,q}\solvmfd}}{\de\left(\wedge^{p+q-1}\solvmfd\otimes_\R\C\right)} \;.
\]
\end{thm}

\begin{proof}
Consider the $G$-left-invariant Hermitian metric
\[ g \;:=\; \sum_{j=1}^{n} x_{j}\odot\bar x_{j} \]
on $\solvmfd$, and the associated $\C$-anti-linear Hodge-$*$-operator $\bar*_g\colon \wedge^{\bullet}\solvmfd \otimes_\R \C \to \wedge^{n-\bullet} \solvmfd \otimes_\R \C$, where $n$ is the dimension of $\solvmfd$.
If the restriction of a character $\alpha$ of $G$ on $\Gamma$ is trivial, then $\alpha$ induces a function on $\solvmfd$ and  the image $\alpha(G)$ is a compact subgroup of $\C^{\ast}$, and hence $\alpha$ is unitary.
For $\alpha_{i_{1}\cdots i_{p}} := \alpha_{i_{1}} \cdot \cdots \cdot \alpha_{i_{p}} \in {\mathcal A}_\Gamma$, since $G$ is unimodular, \cite[Lemma 6.2]{milnor}, for the complement $\{j_{1},\dots, j_{n-p}\}:=\{1,\dots,n\} \setminus \{i_{1},\dots, i_{p}\}$ we have
\[\bar \alpha_{i_{1}\dots i_{p}}=\alpha^{-1}_{i_{1}\cdots i_{p}}=\alpha_{j_{1}\dots j_{n-p}}.
\]
By this, we have
\[\bar \ast_g \left(\alpha_{i_{1}\cdots i_{p}}\cdot\wedge^\bullet \g^{\ast}_{\C}\right) \;=\; \alpha_{j_{1}\dots j_{n-p}}\cdot\wedge^{n-\bullet} \g^{\ast}_{\C}\]
and, for $\alpha_{i_{1}\dots i_{p}} \, x_{i_{1}}\wedge\dots \wedge x_{i_{p}}\in A_{\Gamma}^{\bullet}$, we have
\[\bar\ast_g \left(\alpha_{i_{1}\dots i_{p}} \, x_{i_{1}}\wedge\dots \wedge x_{i_{p}}\right) \;=\; \alpha_{j_{1}\dots j_{n-p}} \, x_{j_{1}}\wedge\dots \wedge x_{j_{n-p}}\in A_{\Gamma}^{n-\bullet} \;.\]
Hence
the sub-complexes
$$ \left(A^\bullet_\Gamma,\, \de\right) \hookrightarrow \left(\bigoplus_{\alpha\in {\mathcal A}_\Gamma} \alpha\cdot \wedge^\bullet\duale{\g_{\C}} ,\, \de\right) \hookrightarrow \left(\wedge^\bullet \solvmfd \otimes_\R\C,\, \de\right) $$
are such that
$$ \bar*_g\lfloor_{A^{\bullet}_\Gamma} \colon A^{\bullet}_{\Gamma} \to A^{n-\bullet}_{\Gamma}
\qquad \text{ and } \qquad
\bar*_g\lfloor_{\bigoplus_{\alpha\in {\mathcal A}_\Gamma} \alpha\cdot \wedge^\bullet\duale{\g_{\C}}} \colon \bigoplus_{\alpha\in {\mathcal A}_\Gamma} \alpha\cdot \wedge^\bullet\duale{\g_{\C}} \to \bigoplus_{\alpha\in {\mathcal A}_\Gamma} \alpha\cdot \wedge^{n-\bullet} \duale{\g_{\C}}
\;, $$
therefore the first assertion follows from Theorem \ref{isoc} and Proposition \ref{subiso}.

Consider the F.~A. Belgun symmetrization map $\mu\colon \wedge^\bullet \solvmfd \otimes_\R \C \to \wedge^\bullet\duale{\mathfrak{g}_\C}$, \cite[Theorem 7]{belgun}.
For $\alpha\in {\mathcal A}_\Gamma$, we define the map
$$ \varphi_{\alpha}\colon \wedge^{\bullet}\solvmfd \otimes_\R\C \to \alpha \cdot \wedge^\bullet \g^{\ast}_{\C} \;, \qquad \varphi _{\alpha}(\omega) \;:=\; \alpha\cdot \mu\left(\frac{\omega}{\alpha}\right) \;. $$
By the definition of $\mu$, for a $G$-left-invariant differential form $\theta$ on $\solvmfd$ and for a differential form $\omega$ on $\solvmfd$, we have $\mu(\theta\wedge\omega)=\theta\wedge\mu(\omega)$, see Lemma \ref{wedco}.
By this we have, for any $\alpha\in\mathcal{A}_\Gamma$,
\begin{eqnarray*}
\varphi_{\alpha}(\de\omega) &=&
\alpha\cdot \mu\left(\frac{\de\omega}{\alpha}\right) \;=\;
\alpha\cdot\mu \left(\de\left(\frac{\omega}{\alpha}\right)+\frac{\de\alpha}{\alpha}\wedge\frac{\omega}{\alpha}\right) \\[5pt]
&=&
\alpha\cdot \de\mu\left(\frac{\omega}{\alpha}\right)+\de\alpha\wedge\mu\left(\frac{\omega}{\alpha}\right) \;=\;
\de\left(\alpha\cdot \mu\left(\frac{\omega}{\alpha}\right)\right) \\[5pt]
&=&
\de\varphi_{\alpha}(\omega) \;,
\end{eqnarray*}
and hence $\varphi _{\alpha}$ is a morphism of cochain complexes. Furthermore, for $\alpha\in\mathcal{A}_\Gamma$, by considering the inclusion
$$ \iota_{\alpha}\colon \alpha\cdot \wedge^\bullet\g^{\ast}_{\C} \hookrightarrow \wedge^{\bullet}\solvmfd \otimes_\R\C \;, $$
we have that
$$ \varphi_{\alpha} \circ \iota_{\alpha} \;=\; \id_{\alpha\cdot \wedge^\bullet \g^{\ast}_{\C}} \;. $$
For distinct characters $\alpha, \alpha^{\prime}\in {\mathcal A}_\Gamma$, for the $G$-left-invariant form $\frac{\alpha^{\prime}}{\alpha}\de\left(\frac{\alpha}{\alpha^{\prime}}\right)$, since $\eta$ is a $G$-left-invariant volume form, we can choose $\lambda \in  \wedge^{\dim G-1} \g^{\ast}_{\C}$ such that $\frac{\alpha^{\prime}}{\alpha}\de\left(\frac{\alpha}{\alpha^{\prime}}\right) \wedge \lambda=\eta$.
Then we have
$$ \de\left(\frac{\alpha}{\alpha^{\prime}}\lambda\right) \;=\; \frac{\alpha}{\alpha^{\prime}}\,\frac{\alpha^{\prime}}{\alpha}\, \de \left(\frac{\alpha}{\alpha^{\prime}}\right) \wedge \lambda \;=\; \frac{\alpha}{\alpha^{\prime}}\, \eta \;.
$$
By this, using Stokes' theorem, for $\alpha\,\omega\in \alpha\cdot \wedge^p \g^{\ast}_{\C}$ and for $X_{1},\dots,X_{p}\in \g \otimes_\R\C$, we have
\begin{eqnarray*}
\mu \left(\frac{\alpha}{\alpha^{\prime}}\,\omega\right) \left(X_{1},\dots,X_{p}\right)
&=& \int_{\solvmfd} \frac{\alpha(x)}{\alpha^{\prime}(x)}\,\omega\lfloor_{x} \left(X_{1}\lfloor_{x},\dots, X_{p}\lfloor_{x}\right)\, \eta(x)
\;=\; \omega \left(X_{1},\dots X_{p}\right)\, \int_{\solvmfd} \frac{\alpha(x)}{\alpha^{\prime}(x)}\, \eta(x) \\[5pt]
&=& \omega \left(X_{1},\dots X_{p}\right)\, \int_{\solvmfd} \de \left(\frac{\alpha}{\alpha^{\prime}}\,\lambda\right) \;=\; 0 \;
\end{eqnarray*}
and hence we have
$$ \varphi_{\alpha^{\prime}}\circ \iota_{\alpha} \;=\; 0 \;. $$
By the definition and since the complex structure on $\solvmfd$ is $G$-left-invariant, we have that, for any $\alpha\in\mathcal{A}_\Gamma$, for any $(p,q)\in\Z^2$,
$$ \varphi_{\alpha}\left(\wedge^{p,q}\solvmfd\right) \subseteq \alpha\cdot \wedge^{p,q} \g^{\ast}_{\C} \;. $$
By noting that the set ${\mathcal A}_\Gamma$ is finite, we define the map
$$ \Phi \;:=\; \sum_{\alpha\in{\mathcal A}_\Gamma} \varphi_{\alpha} \colon \wedge^{\bullet,\bullet}\solvmfd \to \bigoplus_{\alpha\in {\mathcal A}_\Gamma} \alpha\cdot \wedge^{\bullet,\bullet} \g^{\ast}_{\C} \;; $$
note that $\Phi$ is a morphism of cochain complexes and we have, for any $(p,q)\in\Z^2$,
$$ \Phi\left(\wedge^{p,q}\solvmfd\right) \subseteq \bigoplus_{\alpha\in {\mathcal A}_\Gamma} \alpha\cdot \wedge^{p,q}\g^{\ast}_{\C} \qquad \text{ and } \qquad \Phi\circ \iota \;=\; \id_{\bigoplus_{\alpha\in {\mathcal A}_\Gamma}\alpha\cdot \wedge^{p,q}\g^{\ast}_{\C}} \;, $$
where $\iota$ denotes the inclusion $\iota\colon \bigoplus_{\alpha \in {\mathcal A}_\Gamma} \alpha\cdot \wedge^{\bullet,\bullet} \g^{\ast}_{\C} \hookrightarrow \wedge^{\bullet,\bullet}\solvmfd$.
Consider the induced maps
$$ \iota^{\ast}\colon H^{\bullet}\left( \Tot^\bullet \bigoplus_{\alpha\in {\mathcal A}_\Gamma} \alpha\cdot \wedge^{\bullet,\bullet}\g^{\ast}_{\C} ,\, \de \right) \to H^{\bullet}_{dR}\left(\solvmfd; \C\right) $$
and
$$ \Phi^{\ast}\colon H^{\bullet}_{dR}\left(\solvmfd; \C\right) \to H^{\bullet}\left( \Tot^\bullet \bigoplus_{\alpha\in {\mathcal A}_\Gamma} \alpha\cdot \wedge^{\bullet,\bullet} \g^{\ast}_{\C} ,\, \de \right) \;. $$
Since $\iota^{\ast}$ is an isomorphism by the first assertion and $\Phi^\ast\circ\iota^\ast=\id$, then $\Phi^{\ast}$ is the inverse of $\iota^{\ast}$.
By $\Phi\left(\wedge^{p,q}\solvmfd\right) \subseteq \bigoplus_{\alpha\in {\mathcal A}_\Gamma} \alpha\cdot \wedge^{p,q} \g^{\ast}_{\C}$, we have
\[\Phi^{\ast}\left(\frac{\ker \de\lfloor_{\wedge^{p,q}\solvmfd}}{\de\left(\wedge^{p+q-1}\solvmfd\otimes_\R\C\right)}\right) \;\subseteq\; \frac{\ker \de\lfloor_{\bigoplus_{\alpha\in {\mathcal A}_\Gamma} \alpha\cdot \wedge^{p,q}\g^{\ast}_{\C}}}{\de\left(\bigoplus_{\alpha\in {\mathcal A}_\Gamma} \alpha\cdot \wedge^{p+q-1} \g^{\ast}_{\C} \right)} \;.
\]
Hence the restriction of $\Phi^{\ast}$ to $\frac{\ker \de\lfloor_{\wedge^{p,q}\solvmfd}}{\de\left(\wedge^{p+q-1}\solvmfd\right)}$ is the inverse of the restriction of $\iota^{\ast}$ to $\frac{\ker \de\lfloor_{\bigoplus_{\alpha\in {\mathcal A}_\Gamma} \alpha\cdot \wedge^{p,q}\g^{\ast}_{\C} }}{\de\left(\bigoplus_{\alpha\in {\mathcal A}_\Gamma} \alpha\cdot \wedge^{p+q-1} \g^{\ast}_{\C} \right)}$, which is hence an isomorphism. Therefore the second assertion follows.
\end{proof}

\begin{cor}\label{AG}
Let $\solvmfd$ be a solvmanifold.
Let $J$ be a $G$-left-invariant complex structure on $G$ satisfying, for all $g\in G$,
$$ J\circ \left(\Ad_{\mathrm{s}}\right)_g \;=\; \left(\Ad_{\mathrm{s}}\right)_g\circ J \;. $$
Then, by setting $A^{p,q}_{\Gamma}:=A_{\Gamma}^{\bullet}\cap \wedge^{p,q} \solvmfd$ for any $(p,q)\in\Z^2$, we have that the differential graded sub-algebra $\left(A^\bullet_\Gamma,\, \de\right) \hookrightarrow \left(\wedge^\bullet\solvmfd\otimes_\R\C,\, \de\right)$ defined in \eqref{eq:def-a1} is actually $\Z^2$-graded,
$$ A_{\Gamma}^{\bullet} \;=\; \bigoplus_{p+q=\bullet} A^{p,q}_{\Gamma} \;, $$
and the inclusion $A_{\Gamma}^{\bullet,\bullet} \subset  \wedge^{\bullet,\bullet}\solvmfd$ induces the isomorphism
\[\frac{\ker \de\lfloor_{A^{p,q}_{\Gamma}}}{\de\left(A^{p+q-1}_{\Gamma}\right)} \stackrel{\simeq}{\to} \frac{\ker \de\lfloor_{\wedge^{p,q} \solvmfd}}{\de\left(\wedge^{p+q-1}\solvmfd\otimes_\R\C\right)} \;.
\]
\end{cor}

\begin{proof}
Consider the $\Ad_{\mathrm{s}}(G)$-action on  $\bigoplus_{\alpha\in {\mathcal A}_\Gamma} \alpha\cdot \wedge^{\bullet,\bullet} \g^{\ast}_{\C}$.
Then $A_{\Gamma}^{\bullet,\bullet}$ is the sub-complex that consists of the elements of $\bigoplus_{\alpha\in {\mathcal A}_\Gamma} \alpha\cdot \wedge^{\bullet,\bullet} \g^{\ast}_{\C}$ fixed by this action.
Since $\Ad_{\mathrm{s}}$ is diagonalizable, we have the decomposition
\[ \bigoplus_{\alpha\in {\mathcal A}_\Gamma} \alpha\cdot \wedge^{\bullet} \g^{\ast}_{\C} \;=\; A_{\Gamma}^{\bullet} \oplus D^{\bullet}
\]
such that $D^{\bullet}$ is a sub-complex and this decomposition is a direct sum of cochain complexes.
By the assumption $J\circ \left(\Ad_{\mathrm{s}}\right)_g = \left(\Ad_{\mathrm{s}}\right)_g \circ J$ for any $g\in G$, the $\Ad_{\mathrm{s}}(G)$-action is compatible with the bi-grading $\bigoplus_{\alpha\in {\mathcal A}_\Gamma} \alpha\cdot \wedge^{\bullet,\bullet} \g^{\ast}_{\C}$.
Hence we have in fact
 \[\bigoplus_{\alpha\in {\mathcal A}_\Gamma} \alpha\cdot \wedge^{\bullet,\bullet} \g^{\ast}_{\C} \;=\; A_{\Gamma}^{\bullet,\bullet}\oplus D^{\bullet,\bullet} \;.
\]
Consider the projection $p\colon \bigoplus_{\alpha\in {\mathcal A}_\Gamma} \alpha\cdot \wedge^{\bullet,\bullet} \g^{\ast}_{\C} \to A_{\Gamma}^{\bullet,\bullet}$ and the inclusion $\iota \colon A_{\Gamma}^{\bullet,\bullet} \hookrightarrow \bigoplus_{\alpha\in {\mathcal A}_\Gamma} \alpha\cdot \wedge^{\bullet,\bullet} \g^{\ast}_{\C}$.
Then we have $p\circ \iota = {\id}_{A_{\Gamma}^{\bullet,\bullet}}$.
As similar to the proof of Theorem \ref{p,q-co}, we have that $\iota$ induces, for any $(p,q)\in\Z^2$, the isomorphism
\[
\iota^\ast \colon \frac{{\ker} \de\lfloor_{A^{p,q}_{\Gamma}}}{\de\left( A^{p+q-1}_{\Gamma}\right)} \stackrel{\simeq}{\to} \frac{\ker \de\lfloor_{\bigoplus_{\alpha\in {\mathcal A}_\Gamma} \alpha\cdot \wedge^{p,q}\g^{\ast}_{\C} }}{\de\left(\bigoplus_{\alpha\in {\mathcal A}_\Gamma} \alpha\cdot \wedge^{p+q-1} \g^{\ast}_{\C} \right)} \;.\]
Hence the corollary follows from Theorem \ref{p,q-co}.
\end{proof}

\subsection{Complex solvmanifolds of splitting type}\label{spl}

We consider now solvmanifolds of the following type.
\begin{assumption}\label{ass:solvmanifolds}
Consider a solvmanifold $X = \solvmfd$ endowed with a $G$-left-invariant complex structure $J$. Assume that $G$ is the semi-direct product $\C^{n}\ltimes_{\phi}N$ so that:
\begin{enumerate}
 \item\label{item:ass-1} $N$ is a connected simply-connected $2m$-dimensional nilpotent Lie group endowed with an $N$-left-invariant complex structure $J_N$; (denote the Lie algebras of $\C^{n}$ and $N$ by $\mathfrak{a}$ and, respectively, $\n$;)
 \item\label{item:ass-2} for any $t\in \C^{n}$, it holds that $\phi(t)\in \GL(N)$ is a holomorphic automorphism of $N$ with respect to $J_N$;
 \item\label{item:ass-3} $\phi$ induces a semi-simple action on $\n$;
 \item\label{item:ass-4} $G$ has a lattice $\Gamma$; (then $\Gamma$ can be written as $\Gamma = \Gamma_{\C^n} \ltimes_{\phi} \Gamma_{N}$ such that $\Gamma_{\C^n}$ and $\Gamma_{N}$ are  lattices of $\C^{n}$ and, respectively, $N$, and, for any $t\in \Gamma_{\C^n}$, it holds $\phi(t)\left(\Gamma_N\right)\subseteq\Gamma_N$;)
 \item\label{item:ass-5} the inclusion $\wedge^{\bullet,\bullet}\dualee{\n\otimes_\R\C} \hookrightarrow \wedge^{\bullet,\bullet}\left(\left. \Gamma_{N} \right\backslash N\right)$ induces the isomorphism
$$ H^{\bullet}\left(\wedge^{\bullet,\bullet}\left(\n\otimes_\R\C\right)^*,\, \delbar\right) \stackrel{\simeq}{\to} H^{\bullet,\bullet}_{\bar\partial }\left(\left.\Gamma_{N}\middle\backslash N\right.\right) \;. $$
\end{enumerate}
\end{assumption}

Consider the standard basis $\left\{ X_{1},\, \dots,\, X_{n} \right\}$ of $\C^n$.
Consider the decomposition $\n\otimes_\R{\C}=\n^{1,0}\oplus \n^{0,1}$ induced by $J_N$.
By the condition {\ref{item:ass-2}}, this decomposition is a direct sum of $\C^{n}$-modules.
By the condition {\ref{item:ass-3}}, we have a basis $\left\{Y_{1},\, \dots,\, Y_{m}\right\}$ of $\n^{1,0}$ and characters $\alpha_1,\ldots,\alpha_m\in\Hom(\C^n;\C^*)$ such that the induced action $\phi$ on $\n^{1,0}$ is represented by
$$ \C^n \ni t \mapsto \phi(t) \;=\; \diag \left( \alpha_{1}(t),\, \dots,\, \alpha_{m} (t) \right) \in \GL(\n^{1,0}) \;. $$
For any $j\in\{1,\ldots,m\}$, since $Y_{j}$ is an $N$-left-invariant $(1,0)$-vector field on $N$,
the $(1,0)$-vector field $\alpha_{j}Y_{j}$ on $\C^{n}\ltimes _{\phi} N$ is $G$-left-invariant.
Consider the Lie algebra $\g$ of $G$ and the decomposition $\g_\C := \g \otimes_\R \C = \g^{1,0} \oplus \g^{0,1}$ induced by $J$.
Hence we have a basis $\left\{ X_{1},\, \dots,\, X_{n},\, \alpha_{1}Y_{1},\, \dots,\, \alpha_{m}Y_{m}\right\}$ of $\g^{1,0}$, and let $\left\{ x_{1},\, \dots,\, x_{n},\, \alpha^{-1}_{1}y_{1},\, \dots,\, \alpha_{m}^{-1}y_{m} \right\}$ be its dual basis of $\wedge^{1,0}\duale{\g_\C}$.
Then we have 
$$ \wedge^{p,q}\duale{\g_\C} \;=\; \wedge^{p} \left\langle x_{1},\, \dots,\, x_{n},\, \alpha^{-1}_{1}y_{1},\, \dots,\,\alpha^{-1}_{m}y_{m} \right\rangle \otimes \wedge^{q} \left\langle \bar x_{1},\, \dots,\, \bar x_{n},\, \bar\alpha^{-1}_{1}\bar y_{1},\, \dots,\, \bar\alpha^{-1}_{m}\bar y_{m} \right\rangle \;.$$

\medskip

The following lemma holds.

\begin{lem}[{\cite[Lemma 2.2]{kasuya-mathz}}]\label{charr}
Let $X = \solvmfd$ be a solvmanifold endowed with a $G$-left-invariant complex structure $J$ as in Assumption \ref{ass:solvmanifolds}.
Consider a basis $\left\{Y_{1},\, \dots,\, Y_{m}\right\}$ of $\n^{1,0}$ such that the induced action $\phi$ on $\n^{1,0}$ is represented by $\phi(t) = \diag \left( \alpha_{1}(t),\, \dots,\, \alpha_{m} (t) \right)$ for $\alpha_1,\ldots,\alpha_m\in\Hom(\C^n;\C^*)$ characters of $\C^n$.
For any $j\in\{1,\ldots,m\}$, there exist unique unitary characters $\beta_{j}\in\Hom(\C^n;\C^*)$ and $\gamma_{j}\in\Hom(\C^n;\C^*)$ on $\C^{n}$ such that $\alpha_{j}\beta_{j}^{-1}$ and $\bar\alpha_{j}\gamma^{-1}_{j}$ are holomorphic.
\end{lem}

We recall the following result by the second author.

\begin{thm}{\rm (\cite[Corollary 4.2]{kasuya-mathz})}\label{CORR}
Let $X = \solvmfd$ be a solvmanifold endowed with a $G$-left-invariant complex structure $J$ as in Assumption \ref{ass:solvmanifolds}.
Consider the standard basis $\left\{ X_{1},\, \dots,\, X_{n} \right\}$ of $\C^n$.
Consider a basis $\left\{Y_{1},\, \dots,\, Y_{m}\right\}$ of $\n^{1,0}$ such that the induced action $\phi$ on $\n^{1,0}$ is represented by $\phi(t) = \diag \left( \alpha_{1}(t),\, \dots,\, \alpha_{m} (t) \right)$ for $\alpha_1,\ldots,\alpha_m\in\Hom(\C^n;\C^*)$ characters of $\C^n$.
Let $\left\{ x_{1},\, \dots,\, x_{n},\, \alpha^{-1}_{1}y_{1},\, \dots,\, \alpha_{m}^{-1}y_{m} \right\}$ be the basis of $\wedge^{1,0}\duale{\g_\C}$ which is dual to $\left\{ X_{1},\, \dots,\, X_{n},\, \alpha_{1}Y_{1},\, \dots,\, \alpha_{m}Y_{m}\right\}$.
For any $j\in\{1,\ldots,m\}$, let $\beta_{j}$ and $\gamma_{j}$ be the unique unitary characters on $\C^{n}$ such that $\alpha_{j}\beta_{j}^{-1}$ and $\bar\alpha_{j}\gamma^{-1}_{j}$ are holomorphic, as in Lemma \ref{charr}.
Define the differential bi-graded sub-algebra $B^{\bullet,\bullet}_{\Gamma} \subset \wedge^{\bullet,\bullet} \solvmfd$, for $(p,q)\in\Z^2$, as
\begin{eqnarray}\label{eq:def-b}
B^{p,q}_{\Gamma} &:=& \C\left\langle x_{I} \wedge \left(\alpha^{-1}_{J}\beta_{J}\right)\, y_{J} \wedge \bar x_{K} \wedge \left( \bar\alpha^{-1}_{L}\gamma_{L} \right) \, \bar y_{L}
\;\middle\vert\;
\left|I\right| + \left|J\right|=p \text{ and } \left|K\right| + \left|L\right| = q \right. \\[5pt]
\nonumber
&& \left. \text{ such that } \left( \beta_{J}\gamma_{L} \right)\lfloor_{\Gamma} = 1
\right\rangle \;.
\end{eqnarray}
Then the inclusion $B^{\bullet,\bullet}_{\Gamma}\subset \wedge^{\bullet,\bullet}\solvmfd$ induces the cohomology isomorphism
\[
H^{\bullet,\bullet}\left(B^{\bullet,\bullet}_{\Gamma},\, \delbar\right) \stackrel{\simeq}{\to} H^{\bullet,\bullet}_{\delbar}\left(\solvmfd\right) \;.
\]
\end{thm}

\medskip

As a straightforward consequence, by means of conjugation, we get the following result.

\begin{cor}\label{bar}
Let $X = \solvmfd$ be a solvmanifold endowed with a $G$-left-invariant complex structure $J$ as in Assumption \ref{ass:solvmanifolds}.
Consider $B^{\bullet,\bullet}_{\Gamma}$ as in \eqref{eq:def-b}, and let
\begin{equation}\label{eq:def-b-bar}
\bar B^{\bullet,\bullet}_{\Gamma} \;:=\; \left\{ \bar \omega \in \wedge^{\bullet,\bullet}\solvmfd \st \omega\in B^{\bullet,\bullet}_{\Gamma} \right\} \;.
\end{equation}
The inclusion $\bar B_{\Gamma}^{\bullet,\bullet} \hookrightarrow \wedge^{\bullet,\bullet}\solvmfd$ induces the cohomology isomorphism
\[ H^{\bullet,\bullet}\left(\bar B_{\Gamma}^{\bullet,\bullet},\, \del\right) \stackrel{\simeq}{\to} H^{\bullet,\bullet}_{\del}\left(\solvmfd\right) \;. \]
\end{cor}

Hence we get the following result.

\begin{cor}\label{DBARD}
Let $\solvmfd$ be a solvmanifold endowed with a $G$-left-invariant complex structure $J$ as in Assumption \ref{ass:solvmanifolds}.
Consider $B^{\bullet,\bullet}_{\Gamma}$ as in \eqref{eq:def-b}, and $\bar B^{\bullet,\bullet}_{\Gamma}$ as in \eqref{eq:def-b-bar}.
Let
\begin{equation}\label{eq:def-c}
C^{\bullet,\bullet}_{\Gamma} \;:=\; B^{\bullet,\bullet}_{\Gamma} + \bar B_{\Gamma}^{\bullet,\bullet} \;.
\end{equation}
Then we have
\begin{enumerate}
 \item\label{item:DBARD-1} the inclusion $C^{\bullet,\bullet}_{\Gamma} \hookrightarrow \wedge^{\bullet,\bullet}\solvmfd$ induces the cohomology isomorphism
 \[ H^{\bullet,\bullet}\left(C^{\bullet,\bullet}_{\Gamma},\, \del\right) \stackrel{\simeq}{\to} H^{\bullet,\bullet}_{\del}\left(\solvmfd\right) \;;\]
 \item\label{item:DBARD-2} the inclusion $C^{\bullet,\bullet}_{\Gamma} \hookrightarrow \wedge^{\bullet,\bullet}\solvmfd$ induces the cohomology isomorphism
 \[ H^{\bullet,\bullet}\left(C^{\bullet,\bullet}_{\Gamma},\, \delbar\right) \stackrel{\simeq}{\to} H^{\bullet,\bullet}_{\delbar}\left(\solvmfd\right) \;; \]
 \item\label{item:DBARD-3} for any $(p,q)\in\Z^2$, the inclusion $C^{\bullet,\bullet}_{\Gamma} \hookrightarrow \wedge^{\bullet,\bullet}\solvmfd$ induces the surjective map
 \[ \frac{\ker \de\lfloor_{C^{p,q}_\Gamma}}{\de\left(\Tot^{p+q-1} C_{\Gamma}^{\bullet,\bullet}\right)} \to \frac{\ker \de\lfloor_{\wedge^{p,q}\solvmfd}}{\de\left(\wedge^{p+q-1}\solvmfd\otimes_\R \C\right)} \;. \]
\end{enumerate}
\end{cor}

\begin{proof}
Let $g$ be the $G$-left-invariant Hermitian metric on $G$ defined by 
\[ g \;:=\; \sum_{j=1}^{n} x_{j}\odot \bar x_{j} + \sum_{k=1}^{m} \alpha_{k}^{-1}\bar\alpha_{k}^{-1}y_{k}\odot\bar y_{k} \;, \]
and consider its associated $\C$-anti-linear Hodge-$*$-operator $\bar*_g\colon \wedge^{\bullet}\solvmfd \to \wedge^{2N-\bullet}\solvmfd$, where $2N := 2n+2m = \dim_\R\solvmfd$.
Then for multi-indices $I, J \subset \{1,\dots, n\}$ and $K ,L\subset \{1,\dots, m\}$, and their complements $I^{\prime},J^{\prime}\subset \{1,\dots, n\}$ and $K^{\prime} ,L^{\prime}\subset \{1,\dots, m\}$, we have
\[\bar*_g\left(
x_{I} \wedge \left(\alpha^{-1}_{J}\beta_{J}\right)\, y_{J} \wedge \bar x_{K} \wedge \left( \bar\alpha^{-1}_{L}\gamma_{L} \right)\, \bar y_{L}
\right)=x_{I^{\prime}} \wedge \left(\alpha^{-1}_{J^{\prime}}\bar\beta_{J}\right)\, y_{J^{\prime}} \wedge \bar x_{K^{\prime}} \wedge \left( \bar\alpha^{-1}_{L^{\prime}}\bar\gamma_{L} \right) \, \bar y_{L^{\prime}}.
\]
Since $G$ is unimodular by the existence of a lattice, \cite[Lemma 6.2]{milnor}, we have
$\alpha_{J}\alpha_{J^{\prime}}\bar\alpha_{L}\bar\alpha_{L^{\prime}}=1$ and so we have $\beta_{J^{\prime}}\gamma_{L^{\prime}}=\beta^{-1}_{J}\gamma^{-1}_{L}=\bar\beta_{J}\bar\gamma_{L}$.
This implies 
\[x_{I^{\prime}} \wedge \left(\alpha^{-1}_{J^{\prime}}\bar\beta_{J}\right)\, y_{J^{\prime}} \wedge \bar x_{K^{\prime}} \wedge \left( \bar\alpha^{-1}_{L^{\prime}}\bar\gamma_{L} \right) \, \bar y_{L^{\prime}}=x_{I^{\prime}} \wedge \left(\alpha^{-1}_{J^{\prime}}\beta_{J^{\prime}}\right)\, y_{J^{\prime}} \wedge \bar x_{K^{\prime}} \wedge \left( \bar\alpha^{-1}_{L^{\prime}}\gamma_{L^{\prime}} \right) \, \bar y_{L^{\prime}}\in B_{\Gamma}^{\bullet,\bullet}.
\]
Then we have $\bar\ast_g \left( B_{\Gamma}^{\bullet,\bullet}\right) \subseteq B^{N-\bullet,N-\bullet}_{\Gamma}$ and so also
$$ \bar\ast_g \left( C_{\Gamma}^{\bullet,\bullet} \right) \;\subseteq\; C^{N-\bullet,N-\bullet}_{\Gamma} \;. $$
Hence \ref{item:DBARD-1}, respectively \ref{item:DBARD-2}, follows from Theorem \ref{CORR}, respectively Corollary \ref{bar}, and Proposition \ref{subiso}.

We consider the sub-complex $A^{\bullet}_{\Gamma} \subseteq \wedge^{\bullet}\solvmfd \otimes_\R\C$ defined in \eqref{eq:def-a1}.
Consider the standard basis $\left\{ X_{1},\, \dots,\, X_{n} \right\}$ of $\C^n$.
Consider a basis $\left\{Y_{1},\, \dots,\, Y_{m}\right\}$ of $\n^{1,0}$ such that the induced action $\phi$ on $\n^{1,0}$ is represented by $\phi(t) = \diag \left( \alpha_{1}(t),\, \dots,\, \alpha_{m} (t) \right)$ for $\alpha_1,\ldots,\alpha_m\in\Hom(\C^n;\C^*)$ characters of $\C^n$.
Then, with respect to the basis $\left\{ X_{1},\, \dots,\, X_{n},\,\bar X_{1},\, \dots,\, \bar X_{n},\, \alpha_{1}Y_{1},\, \dots,\, \alpha_{m}Y_{m},\, \bar\alpha_{1}{\bar Y}_{1},\, \dots,\, \bar\alpha_{m}\bar Y_{m}\right\}$ of $\g_{\C}=\g^{1,0}\oplus \g^{0,1}$, we have, for $(t,n)\in G=\C^{n}\ltimes_{\phi}N$,
\begin{eqnarray*}
\left(\Ad_{\mathrm{s}}\right)_{(t,n)}
&=&
\left(
\begin{array}{c|c}
{\id}_{(\C^{n})^{1,0}\oplus (\C^{n})^{0,1}}& 0 \\
\hline
0&    \phi_{\ast}\lfloor_{\n^{1,0}\oplus \n^{0,1}}(t)
\end{array}
\right) \\[5pt]
&=& \diag\left(\underbrace{1, \, \dots, \, 1}_{2n \text{ times}}, \, \alpha_{1}(t), \, \dots, \, \alpha_{m}(t), \, {\bar \alpha}_{1}(t), \, \dots, \, \bar \alpha_{m}(t)\right) \;.
\end{eqnarray*}
Hence we have $J\circ \left(\Ad_{\mathrm{s}}\right)_{(t,n)} = \left(\Ad_{\mathrm{s}}\right)_{(t,n)}\circ J$, and 
we can easily see that $A_{\Gamma}^{\bullet,\bullet}\subseteq C_{\Gamma}^{\bullet,\bullet} \subseteq \wedge^{\bullet,\bullet}\solvmfd$.
Since the composition
\[\frac{\ker\de\lfloor_{A^{p,q}_{\Gamma}}}{\de\left(A^{p+q-1}_{\Gamma}\right)} \to \frac{\ker\de\lfloor_{C^{p,q}}}{\de\left(\Tot^{p+q-1} C_{\Gamma}^{\bullet,\bullet}\right)}\to \frac{\ker\de\lfloor_{\wedge^{p,q}\solvmfd}}{\de\left(\wedge^{p-q-1}\solvmfd \otimes_\R\C\right)}
\] 
is an isomorphism, then \ref{item:DBARD-3} of the corollary follows.
\end{proof}

Finally we get the following theorem.
\begin{thm}\label{BCISO}
Let $\solvmfd$ be a solvmanifold endowed with a $G$-left-invariant complex structure $J$ as in Assumption \ref{ass:solvmanifolds}.
Consider $C^{\bullet,\bullet}_{\Gamma}$ as in \eqref{eq:def-c}.
For any $(p,q)\in\Z^2$, the inclusion $C^{\bullet,\bullet}_{\Gamma}\subseteq \wedge^{\bullet,\bullet}\solvmfd$ induces the isomorphism
\[ H \left( C_{\Gamma}^{p-1,q-1} \stackrel{\del\delbar}{\to} C_{\Gamma}^{p,q} \stackrel{\del+\delbar}{\to} C_{\Gamma}^{p+1,q} \oplus C_{\Gamma}^{p,q+1} \right) \stackrel{\simeq}{\to} H^{p,q}_{BC}\left(\solvmfd\right) \;. \]
\end{thm}

\begin{proof}
By Corollary \ref{DBARD}, the surjectivity follows from Theorem \ref{surjBC}.
The injectivity follows from Proposition \ref{injCo}.
\end{proof}

As an application, we will study the completely-solvable Nakamura manifold in Example \ref{CSN}.

\medskip

Given a property depending on the complex structure, one says that it is \emph{open under small deformations} (respectively, \emph{strongly-closed under small deformations}) if, for any complex-analytic families of compact complex manifolds parametrized by $\mathcal{B}$, the set of parameters for which the property holds is open (respectively, closed) in the strong topology of $\mathcal{B}$.

We recall that satisfying the $\del\delbar$-Lemma is an open property under small deformations, see \cite[Proposition 9.21]{voisin}, \cite[Theorem 5.12]{wu}, \cite[\S B]{tomasiello}, \cite[Corollary 2.7]{angella-tomassini-3}. On the other hand, as pointed out by Luis Ugarte, the completely-solvable Nakamura manifold provides a counterexample to the strongly-closedness of the property of satisfying the $\del\delbar$-Lemma: indeed, complex structures in class \ref{item:nakamura-3} satisfy the $\del\delbar$-Lemma while complex structures in classes \ref{item:nakamura-1} and \ref{item:nakamura-2} do not. We have hence the following theorem.

\begin{thm}\label{thm:deldelbar-non-closed}
 Satisfying the $\del\delbar$-Lemma is not a strongly-closed property under small deformations of the complex structure.
\end{thm}

\begin{rem}
 Actually, as remarked by Luis Ugarte, in defining closedness for deformations, one usually considers the Zariski topology, see, {\itshape e.g.} \cite{popovici-annsns}: namely, a property $\mathcal{P}$ is said to be \emph{(Zariski-)closed} if, for any family $\left\{X_t\right\}_{t\in\Delta}$ of compact complex manifolds such that $\mathcal{P}$ holds for any $t\in\Delta\setminus\{0\}$ in the punctured-disk, then $\mathcal{P}$ holds also for $X_0$. In \cite{angella-kasuya-2}, a family of deformations of the complex parallelizable Nakamura manifold is studied in order to prove that satisfying the $\del\delbar$-Lemma is also non-(Zariski-)closed.
\end{rem}

\subsection{Complex parallelizable solvmanifolds}
Let $G$ be a connected simply-connected complex solvable Lie group admitting a lattice $\Gamma$, and denote by $2n$ the real dimension of $G$. Denote the Lie algebra naturally associated to $G$ by $\g$.
We use the following lemma.

\begin{lem}\label{uniim}
Let $\alpha$, $\beta$ be holomorphic characters of a connected simply-connected complex solvable Lie group $G$.
If $\alpha\bar\beta$ is a unitary character, then $\alpha=\beta^{-1}$.
\end{lem}
\begin{proof}
Since we have $\alpha([G,G])= [\alpha(G),\alpha(G)]=1$ and $\beta([G,G])= [\beta(G),\beta(G)]=1$, we can regard $\alpha$ and $\beta$ as characters of $G/[G,G]$. Since $G$ is connected simply-connected, $\left. G \middle/ [G,G] \right.$ is also connected simply-connected, see \cite[Theorem 3.5]{dek}.
Since $\left. G \middle/ [G,G] \right.$ is Abelian, it is sufficient to show the lemma in the case $G=\C^{n}$.
For the coordinate set $\left( z_{1}, \dots, z_{n} \right)$ of $\C^{n}$, we write
$\alpha=\exp \left(\sum_{j=1}^{n} a_{j}z_{j}\right)$ and $\beta=\exp \left(\sum_{j=1}^{n} b_{j}z_{j}\right)$, for some $a_1,\ldots,a_n,b_1,\ldots,b_n\in\C$.
If $\alpha\bar\beta$ is unitary, then we have
$\Re \left(\sum_{j=1}^{n} \left( a_{j}z_{j}+\bar b_{j}\bar z_{j}\right) \right)=0$.
By simple computations, we have $a_{j}=-b_{j}$ for any $j\in\{1,\ldots,n\}$.
Hence the lemma follows.
\end{proof}

Denote by $ \g_{+}$ (respectively, $ \g_{-}$) the Lie algebra of the $G$-left-invariant holomorphic (respectively, anti-holomorphic) vector fields on $G$.
As a (real) Lie algebra, we have an isomorphism  $ \g_{+}\simeq \g_{-}$ by means of the complex conjugation.

Let $N$ be the nilradical of $G$.
We can take a connected simply-connected complex nilpotent subgroup $C\subseteq G$  such that $G=C\cdot N$, see, {\itshape e.g.} \cite[Proposition 3.3]{dek}.
Since $C$ is nilpotent, the map
\[C\ni c \mapsto ({\Ad}_{c})_{\mathrm{s}}\in {\Aut}(\g_{+})\]
is a homomorphism, where $({\Ad}_{c})_{\mathrm{s}}$ is the semi-simple part of the Jordan decomposition of ${\Ad}_{c}$.
Let $\mathfrak{c}$ be the Lie algebra of $C$; we take a subspace $V\subseteq \mathfrak{c}$ such that $\g=V\oplus \n$.
Then the diagonalizable representation $\Ad_{\mathrm{s}}$ constructed above, \S\ref{subsec:solvmfd}, is identified with the map
\[G=C\cdot N\ni c\cdot n\mapsto ({\rm Ad}_{c})_{s} \in{\rm Aut}(\g),
\] see \cite[Remark 4]{kasuya-holpar}.

We have a basis $\left\{ X_{1},\dots,X_{n} \right\}$ of $\g_{+}$ such that, for $c \in C$,
$$ ({\Ad}_{c})_{\mathrm{s}} \;=\; {\diag} \left(\alpha_{1}(c),\dots,\alpha_{n}(c) \right) \;, $$
for some characters $\alpha_1,\dots,\alpha_n$ of $C$.
By $G=C\cdot N$, we have $G/N=C/C\cap N$ and regard $\alpha_1,\dots,\alpha_n$ as characters of $G$.
Let $\left\{ x_{1},\dots, x_{n} \right\}$ be the basis of $\g^{\ast}_{+}$ which is dual to $\left\{ X_{1},\dots ,X_{n} \right\}$.

\begin{thm}{\rm (\cite[Corollary 6.2 and its proof]{kasuya-holpar})}\label{MMTT}
Let $G$ be a connected simply-connected complex solvable Lie group admitting a lattice $\Gamma$. Denote the Lie algebra naturally associated to $G$ by $\g$.
Consider a basis $\left\{ X_{1},\dots,X_{n} \right\}$ of the Lie algebra $\g_+$ of the $G$-left-invariant holomorphic vector fields on $G$ with respect to which $({\Ad}_{c})_{\mathrm{s}} = {\diag} \left(\alpha_{1}(c),\dots,\alpha_{n}(c) \right)$ for some characters $\alpha_1,\dots,\alpha_n$ of $C$.
Regard $\alpha_1,\dots,\alpha_n$ as characters of $G$.
Let $B^{\bullet}_{\Gamma}$ be the sub-complex of $\left( \wedge^{0,\bullet} \solvmfd, \, \delbar \right) $ defined as
\begin{equation}\label{eq:def-b-holpar}
B^{\bullet}_{\Gamma} \;:=\; \left\langle \frac{\bar\alpha_{I}}{\alpha_{I} }\, \bar x_{I} \;\middle\vert\; I \subseteq \{1,\ldots,n\} \text{ such that } \left.\left(\frac{\bar\alpha_{I}}{\alpha_{I}}\right)\right\lfloor_{\Gamma}=1 \right\rangle \;, 
\end{equation}
(where we shorten, {\itshape e.g.} $\alpha_I:=\alpha_{i_1}\cdot\cdots\cdot\alpha_{i_{k}}$ for a multi-index $I=\left(i_1,\dots,i_k\right)$).
Then the inclusion $B^{\bullet}_{\Gamma} \hookrightarrow \wedge^{0,\bullet}\solvmfd$ induces the isomorphism
\[H^{\bullet}\left(B^{\bullet}_{\Gamma},\, \delbar\right) \stackrel{\simeq}{\to} H_{\bar\partial}^{0,\bullet}(\solvmfd) \;.\]
\end{thm}
By this theorem, since $\solvmfd$ is complex parallelizable, for the differential bi-graded algebra $\left( \wedge^\bullet \g_{+}^{\ast}\otimes_\C B^{\bullet}_{\Gamma}, \,\bar\partial \right)$, the inclusion $\wedge^{\bullet_1} \g^{\ast}_{+}\otimes_\C B^{\bullet_2}_{\Gamma} \hookrightarrow \wedge^{\bullet_1,\bullet_2} \solvmfd$ induces the isomorphism
\[ \wedge^{\bullet_1} \g^{\ast}_{+}\otimes_\C H^{\bullet_2}_{\bar\partial}(B^{\bullet}_{\Gamma}) \stackrel{\simeq}{\to} H^{\bullet_1,\bullet_2}_{\bar\partial}(\solvmfd) \;. \]

Consider the $G$-left-invariant Hermitian metric
\[g \;:=\; \sum_{j=1}^{n} x_{j} \odot \bar x_{j} \;. \]
Then, for $x_{I}\wedge\frac{\bar\alpha_{K}}{\alpha_{K} }\,\bar x_{K}\in \wedge^{|I|} \g_{+}^{\ast}\otimes_\C B^{|K|}_{\Gamma}$,
since $G$ is unimodular, \cite[Lemma 6.2]{milnor}, 
we have
\[\bar\ast_{g} \left( x_{I}\wedge\frac{\bar\alpha_{K}}{\alpha_{K} }\,\bar x_{K}\right) \;=\; x_{I^{\prime}}\wedge\frac{\alpha_{K}}{\bar\alpha_{K} }\,\bar x_{K^{\prime}} \;=\; x_{I^{\prime}}\wedge\frac{\bar\alpha_{K^{\prime}}}{\alpha_{K^{\prime}} }\,\bar x_{K^{\prime}} \;\in\; \wedge^{n-|I|}\g_{+}^{\ast}\otimes_\C B^{n-|K|}_{\Gamma} \;,
\]
where $I^{\prime} := \{1,\ldots,n\} \setminus I$ and $K^{\prime} := \{1,\ldots,n\} \setminus K$ are the complements of $I$ and $K$ respectively.
Hence we have $\bar\ast_{g}(\wedge^\bullet \g_{+}^{\ast}\otimes_\C B^{\bullet}_{\Gamma}) \subseteq \wedge^{n-\bullet} \g_{+}^{\ast}\otimes_\C B^{n-\bullet}_{\Gamma}$.

We consider the space
\[\bar B^{\bullet}_{\Gamma} \;=\; \left\langle \frac{\alpha_{I}}{\bar\alpha_{I} } \, x_{I} \;\middle\vert\; I \subseteq \{1,\ldots,n\} \text{ such that } \left.\left(\frac{\alpha_{I}}{\bar\alpha_{I}}\right)\right\lfloor_{{\Gamma}}=1\right\rangle \;.\]
Then the inclusion $\bar B^{\bullet_1}_{\Gamma}\otimes_\C \wedge^{\bullet_2} \g_{-}^{\ast}\subseteq \wedge^{\bullet_1,\bullet_2}\solvmfd $ induces the isomorphism in $\partial$-cohomology
$$ H^{\bullet_1}\left(\bar B^{\bullet}_{\Gamma}\otimes_\C \wedge^{\bullet_2} \g_{-}^{\ast} ,\, \del \right) \stackrel{\simeq}{\to} H^{\bullet_1,\bullet_2}_{\del}\left(\solvmfd\right) \;. $$
Consider
\begin{equation}\label{eq:def-c-cplx-par}
C^{\bullet_1,\bullet_2} \;:=\; \wedge^{\bullet_1} \g_{+}^{\ast} \otimes_\C B^{\bullet_2}_{\Gamma} + \bar B^{\bullet_1}_{\Gamma} \otimes_\C \wedge^{\bullet_2} \g_{-}^{\ast} \;.
\end{equation}
Then we have $\bar\ast_g \left(C^{\bullet_1,\bullet_2}\right)\subseteq C^{n-\bullet_1,n-\bullet_2}$.

As similar to Corollary \ref{DBARD}, we can show the following result.
\begin{cor}
Let $G$ be a connected simply-connected complex solvable Lie group admitting a lattice $\Gamma$. Denote the Lie algebra naturally associated to $G$ by $\g$. Consider the sub-complex $C^{\bullet,\bullet}_{\Gamma}\subseteq \wedge^{\bullet,\bullet}\solvmfd$ as defined in \eqref{eq:def-c-cplx-par}.
\begin{enumerate}
\item The inclusion $C^{\bullet,\bullet}_{\Gamma} \hookrightarrow \wedge^{\bullet,\bullet}\solvmfd$ induces the $\partial$-cohomology isomorphism
\[H^{\bullet,\bullet}(C^{\bullet,\bullet}_{\Gamma},\, \del) \stackrel{\simeq}{\to} H^{\bullet,\bullet}_{\del}(\solvmfd) \;.\]
\item The inclusion $C_{\Gamma}^{\bullet,\bullet} \hookrightarrow \wedge^{\bullet,\bullet}\solvmfd$ induces the $\delbar$-cohomology isomorphism
\[H^{\bullet,\bullet}(C_{\Gamma}^{\bullet,\bullet},\, \delbar) \stackrel{\simeq}{\to} H^{\bullet,\bullet}_{\delbar}(\solvmfd) \;.\]
\item The inclusion $C_{\Gamma}^{\bullet,\bullet} \hookrightarrow \wedge^{\bullet,\bullet}\solvmfd$ induces, for any $(p,q)\in\Z^2$, the surjection
\[\frac{\ker \de\lfloor_{C^{p,q}}}{\de\left(\Tot^{p+q-1} C_{\Gamma}^{\bullet,\bullet}\right)}\to \frac{\ker \de\lfloor_{\wedge^{p,q}\solvmfd}}{\de \left( \wedge^{p+q-1}\solvmfd\otimes_\R \C\right)} \;.
\]
\end{enumerate}
\end{cor}
\begin{proof}
By $\bar\ast_g \left(C^{\bullet_1,\bullet_2}\right)\subseteq C^{n-\bullet_1,n-\bullet_2}$, the first and second assertions follow as similar to the proof of Corollary \ref{DBARD}.

By denoting the complex structure by $J$, for any $c\in C$, since we have ${\Ad}_{c}\circ J=J\circ {\Ad}_{c}$, we have $({\Ad}_{c})_{\mathrm{s}}\circ J=J\circ ({\Ad}_{c})_{\mathrm{s}}$, and hence we have $\left( \Ad_{\mathrm{s}}\right)_g \circ J = J\circ \left(\Ad_{\mathrm{s}}\right)_g$ for any $g\in G$.
We consider the sub-complex $A^{\bullet}_{\Gamma}\subseteq \wedge^{\bullet}\solvmfd \otimes_\R\C$ as in \eqref{eq:def-a1}, see Theorem \ref{isoc}.
By Corollary \ref{AG}, the inclusion $A_{\Gamma}^{\bullet,\bullet} \hookrightarrow \wedge^{\bullet,\bullet}\solvmfd$ induces the isomorphism
\[\frac{\ker \de\lfloor_{A^{p,q}_{\Gamma}}}{\de\left(A^{p+q-1}_{\Gamma}\right)} \stackrel{\simeq}{\to} \frac{\ker \de\lfloor_{\wedge^{p,q} \solvmfd}}{\de\left(\wedge^{p+q-1}\solvmfd\otimes_\R\C\right)} \;.
\]
We have
\[A_{\Gamma}^{\bullet} \;=\; \left\langle \alpha_{I}\, \bar\alpha_{J}\, x_{I} \wedge \bar x_{J} \;\middle\vert\; I, J \subseteq \{1, \ldots, n\} \text{ such that } \left(\alpha_{I}\,\bar\alpha_{J}\right)\lfloor_{\Gamma}=1\right\rangle \;. \]
For $\left(\alpha_{I}\,\bar\alpha_{J}\right)\lfloor_{\Gamma}=1$, since we can regard $\alpha_{I}\,\bar\alpha_{J}$ as a function on $\solvmfd$, the image of $\alpha_{I}\,\bar\alpha_{J}$ is compact and hence it is unitary.
By Lemma \ref{uniim}, we have $\alpha_{I}\,\bar\alpha_{J}=\frac{\bar\alpha_{J}}{\alpha_{J}}$.
Hence we have the inclusion $A_{\Gamma}^{\bullet}\subseteq \Tot^\bullet \wedge^\bullet \g_{+}^{\ast}\otimes B^{\bullet}_{\Gamma}$ and so we have the inclusion $A_{\Gamma}^{\bullet,\bullet}\subseteq C_{\Gamma}^{\bullet,\bullet} \subseteq \wedge^{\bullet,\bullet}\solvmfd$.
Since the composition
\[\frac{\ker\de\lfloor_{A^{p,q}_{\Gamma}}}{\de\left(A^{p+q-1}_{\Gamma}\right)} \to \frac{\ker\de\lfloor_{C^{p,q}}}{\de\left(\Tot^{p+q-1} C_{\Gamma}^{\bullet,\bullet}\right)}\to \frac{\ker\de\lfloor_{\wedge^{p,q}\solvmfd}}{\de\left(\wedge^{p-q-1}\solvmfd\right)}
\] 
is an isomorphism, then the third assertion of the corollary follows.
\end{proof}

By this, we get the following result.
\begin{thm}\label{palBCISO}
Let $G$ be a connected simply-connected complex solvable Lie group admitting a lattice $\Gamma$. Consider the sub-complex $C^{\bullet,\bullet}_{\Gamma}\subseteq \wedge^{\bullet,\bullet}\solvmfd$ as defined in \eqref{eq:def-c-cplx-par}.
The inclusion $C^{\bullet,\bullet}_{\Gamma} \hookrightarrow \wedge^{\bullet,\bullet}\solvmfd$ induces the isomorphism
\[ H\left( C_{\Gamma}^{\bullet-1,\bullet-1} \stackrel{\del\delbar}{\to} C_{\Gamma}^{\bullet,\bullet} \stackrel{\de}{\to} C_{\Gamma}^{\bullet+1,\bullet} \oplus C_{\Gamma}^{\bullet,\bullet+1} \right) \stackrel{\simeq}{\to}  H^{\bullet,\bullet}_{BC}(\solvmfd) \;.\]
\end{thm}

As an application, we will study the complex parallelizable Nakamura manifold in Example \ref{Nak}.

\subsection{Currents}

Let $X$ be a compact complex manifold, of complex dimension $n$. Denote the space of currents on $X$ by $\correnti^{\bullet,\bullet}X := \correnti_{n-\bullet,n-\bullet}X$, namely, the topological dual space of $\wedge^{n-\bullet,n-\bullet}X$; endow $\correnti^{\bullet,\bullet}X$ with a structure of double complex, by defining $\del\colon \correnti^{\bullet,\bullet}X \to \correnti^{\bullet+1,\bullet}X$ and $\delbar\colon \correnti^{\bullet,\bullet}X \to \correnti^{\bullet,\bullet+1}X$ by duality.

By means of the injective operator
$$ T_\sspace \colon \wedge^{\bullet,\bullet}X \to \correnti^{\bullet,\bullet}X \;, \qquad T_\eta \;:=\; \int_X \eta\wedge \sspace \;,$$
which satisfies $T \circ \del = \del \circ T$ and $T \circ \delbar = \delbar \circ T$, consider the de Rham double complex $\left(\wedge^{\bullet,\bullet}X,\, \del,\, \delbar\right)$ as a double sub-complex of $\left(\correnti^{\bullet,\bullet},\, \del,\, \delbar\right)$.

For $\left(p,q\right)\in\Z^2$, denote the sheaf of $p$-holomorphic forms on $X$ by $\Omega^p_X$, denote the sheaf of $(p,q)$-forms on $X$ by $\mathcal{A}^{p,q}_X$, and denote the sheaf of bi-degree $(p,q)$-currents by $\correntifascio^{p,q}_X$. Recall that, for any fixed $p\in\Z$, both
$$
0 \to \Omega^p_X \to \left(\mathcal{A}^{p,\bullet}_X,\, \delbar\right) \qquad \text{ and } \qquad 0 \to \Omega^p_X \to \left(\correntifascio^{p,\bullet}_X,\, \delbar\right)
$$
are fine (and hence acyclic, see, {\itshape e.g.} \cite[IV.4.19]{demailly-agbook}) resolutions of $\Omega^p_X$, and hence
$$
\frac{\ker\left(\delbar\colon \wedge^{p,\bullet}X \to \wedge^{p,\bullet+1}X\right)}{\imm\left(\delbar\colon \wedge^{p,\bullet-1}X \to \wedge^{p,\bullet}X\right)}
\;\simeq\;
\check{H}^{\bullet}\left(X;\Omega^{p}_X\right)
\;\simeq\;
\frac{\ker\left(\delbar\colon \correnti^{p,\bullet}X \to \correnti^{p,\bullet+1}X\right)}{\imm\left(\delbar\colon \correnti^{p,\bullet-1}X \to \correnti^{p,\bullet}X\right)}
\;,
$$
see, {\itshape e.g.} \cite[IV.6.4]{demailly-agbook}.

\begin{rem}\label{rem:t-qis-dolb}
 More precisely, given $X$ a compact complex manifold, for any $p\in\Z$ and for any $q\in\Z$, the maps $T_\sspace\colon \left(\wedge^{\bullet,q}X,\, \del\right) \to \left(\correnti^{\bullet,q}X,\, \del\right)$ and $T_\sspace\colon \left(\wedge^{p,\bullet}X,\, \delbar\right) \to \left(\correnti^{p,\bullet}X,\, \delbar\right)$ are quasi-isomorphisms.

 Indeed, firstly, we show that $T_\sspace\colon \left(\wedge^{p,\bullet}X,\, \delbar\right) \to \left(\correnti^{p,\bullet}X,\, \delbar\right)$ induces an injective map in cohomology. Fix $g$ a Hermitian metric on $X$. If $T_{[\alpha]}=\left[\delbar S\right]=[0]\in H^\bullet\left(\correnti^{p,\bullet}X,\, \delbar\right)$ with $\alpha$ the $\overline\square_g$-harmonic representative of $[\alpha]\in H^\bullet\left(\wedge^{p,\bullet}X,\, \delbar\right)$ and $S\in\correnti^{p,\bullet-1}X$, then in particular $T_\alpha\lfloor_{\ker\delbar}=0$. Since $\bar*_g\alpha\in\ker\delbar$, it follows that $0=T_\alpha\left(\bar*_g\alpha\right)=\int_X \alpha\wedge\bar*_g\alpha$ and hence $\alpha=0$. Now, since $\frac{\ker\left(\delbar\colon \wedge^{p,\bullet}X \to \wedge^{p,\bullet+1}X\right)}{\imm\left(\delbar\colon \wedge^{p,\bullet-1}X \to \wedge^{p,\bullet}X\right)}$ and $\frac{\ker\left(\delbar\colon \correnti^{p,\bullet}X \to \correnti^{p,\bullet+1}X\right)}{\imm\left(\delbar\colon \correnti^{p,\bullet-1}X \to \correnti^{
p,\bullet}X\right)}$ are isomorphic $\C$-vector spaces of finite dimension, it follows that $T_\sspace\colon \left(\wedge^{p,\bullet}X,\, \delbar\right) \to \left(\correnti^{p,\bullet}X,\, \delbar\right)$ is actually a quasi-isomorphism. By conjugation, also $T_\sspace\colon \left(\wedge^{\bullet,q}X,\, \del\right) \to \left(\correnti^{\bullet,q}X,\, \del\right)$ is a quasi-isomorphism.
\end{rem}

By applying Proposition \ref{spdb} to $\left(\wedge^{p,\bullet}X,\, \delbar\right) \hookrightarrow \left(\correnti^{p,\bullet}X,\, \delbar\right)$, or by noting that both $0 \to \underline{\C}_X \to \left(\mathcal{A}^\bullet_X\otimes \C,\, \de\right)$ and $0 \to \underline{\C}_X \to \left(\correntifascio^\bullet_X\otimes\C,\, \de\right)$ are acyclic resolutions of the constant sheaf $\underline{\C}_X$ over $X$
(where, for $k\in\Z$, the sheaf of $k$-forms on $X$ is denoted by $\mathcal{A}^k_X$, and the sheaf of degree $k$-currents is denoted by $\correntifascio^{k}_X$), one gets that
$$
\frac{\ker\left(\de\colon \wedge^{\bullet}X\otimes_\R\C \to \wedge^{\bullet+1}X\otimes_\R\C\right)}{\imm\left(\de \colon \wedge^{\bullet-1}X\otimes_\R\C \to \wedge^{\bullet}X\otimes_\R\C\right)}
\;\simeq\;
\check{H}^{\bullet}\left(X;\underline{\C}_X\right)
\;\simeq\;
\frac{\ker\left(\de\colon \correnti^{\bullet}X\otimes_\R\C \to \correnti^{\bullet+1}X\otimes_\R\C\right)}{\imm\left(\de\colon \correnti^{\bullet-1}X\otimes_\R\C \to \correnti^{\bullet}X\otimes_\R\C \right)}
\;.
$$

\begin{lem}\label{lem:Teta-surj-BCcoh}
 Let $X$ be a compact complex manifold. For any $\left(p,q\right)\in\Z^2$, the map $T_\sspace \colon \wedge^{\bullet,\bullet}X \to \correnti^{\bullet,\bullet}X$ induces the isomorphism
 $$ T_\sspace \colon \frac{\ker\left(\de\colon \wedge^{p,q}X \to \wedge^{p+q+1}X\otimes_\R\C\right)}{\imm\left(\de\colon \wedge^{p+q-1}X\otimes_\R\C \to \wedge^{p+q}X\otimes_\R\C\right)} \to \frac{\ker\left(\de\colon \correnti^{p,q}X \to \correnti^{p+q+1}X\otimes_\R\C\right)}{\imm\left(\de\colon \correnti^{p+q-1}X\otimes_\R\C \to \correnti^{p+q}X\otimes_\R\C\right)} \;.$$
\end{lem}

\begin{proof}
 Consider the regularization process in \cite[Theorem III.12]{derham}: there exist $R\colon \correnti^{\bullet,\bullet}X \to \wedge^{\bullet,\bullet}X$ and $A\colon \correnti^{\bullet}X\otimes_\R\C \to \correnti^{\bullet+1}X\otimes_\R\C$ linear operators such that
 $$ \id_{\correnti^{\bullet,\bullet}X} \;=\; R + \de A + A\de \;, \qquad \text{ and } \qquad R\lfloor_{\wedge^{\bullet,\bullet}X} \;=\; \id_{\wedge^{\bullet,\bullet}X} \text{ and } A\lfloor_{\wedge^{\bullet,\bullet}X}\;=\;0 \;. $$
 Take $S \in \frac{\ker\left(\de\colon \correnti^{p,q}X \to \correnti^{p+q+1}X\otimes_\R\C\right)}{\imm\left(\de\colon \correnti^{p+q-1}X\otimes_\R\C \to \correnti^{p+q}X\otimes_\R\C\right)}$.
 Since the map $T_\sspace \colon \wedge^{\bullet,\bullet}X \to \correnti^{\bullet,\bullet}X$ is a quasi-isomorphism, then there exist $\eta\in \ker\de \cap \wedge^{p,q}X$ and $U\in \correnti^{p+q-1}X\otimes_\R\C$ such that
 $$ S \;=\; T_\eta + \de U \;; $$
 hence one gets
 $$ RS \;=\; T_\eta + \de \left( U - AS \right) \;, $$
 and hence the lemma follows.
\end{proof}

As a consequence, by using Theorem \ref{surjBC}, we get another proof of the following result by M. Schweitzer: see \cite{schweitzer}, and also \cite[\S3.4]{kooistra}, where it is noticed as a consequence of the hypercohomological interpretation of the Bott-Chern cohomology, see also \cite[IV.12.1]{demailly-agbook}.

\begin{cor}[{see \cite[\S4.d]{schweitzer}}]
 Let $X$ be a compact complex manifold. Then, for any $\left(p,q\right)\in\Z^2$, the natural map
 $$ T_\sspace \colon \frac{\ker \left(\del+\delbar \colon \wedge^{p,q}X \to \wedge^{p+1,q}X\oplus\wedge^{p,q+1}X\right)}{\imm\left(\del\delbar\colon \wedge^{p-1,q-1}X \to \wedge^{p,q}X\right)} \to \frac{\ker \left(\del+\delbar \colon \correnti^{p,q}X \to \correnti^{p+1,q}X\oplus\correnti^{p,q+1}X\right)}{\imm\left(\del\delbar\colon \correnti^{p-1,q-1}X \to \correnti^{p,q}X\right)} $$
 induced by $T_\sspace \colon \wedge^{\bullet,\bullet}X \ni \eta \mapsto T_\eta:=\int_X \eta\wedge\sspace \in \correnti^{\bullet,\bullet}X$ is an isomorphism.
\end{cor}

\begin{proof}
 We firstly prove that $T_\sspace$ induces an injective map in Bott-Chern cohomology. Indeed, let $\mathfrak{a} = \left[ \alpha \right] \in H^{p,q}_{BC}(X)$ be such that $\left[T_\mathfrak{a}\right] = 0 \in \frac{\ker \left(\del+\delbar \colon \correnti^{p,q}X \to \correnti^{p+1,q}X\oplus\correnti^{p,q+1}X\right)}{\imm\left(\del\delbar\colon \correnti^{p-1,q-1}X \to \correnti^{p,q}X\right)}$.
 Choose $g$ a Hermitian metric on $X$, and let $\alpha\in\wedge^{p,q}X$ be the $\tilde\Delta^{BC}$-harmonic representative of $\mathfrak{a}$ with respect to $g$. Therefore, there exists $S\in \correnti^{p-1,q-1}X$ such that $T_\alpha=\del\delbar S$. In particular, $T_\alpha\lfloor_{\ker\del\delbar}=0$. Since $\bar*_g\alpha\in\ker\del\delbar$, it follows that $0=T_\alpha\left(\bar*_g\alpha\right)=\int_X\alpha\wedge\bar*_g\alpha$, and hence $\mathfrak{a}=\left[\alpha\right]=0$.

 We prove now that $T_\sspace$ induces a surjective map in Bott-Chern cohomology. Firstly, by Remark \ref{rem:t-qis-dolb}, for any $p\in\Z$ and for any $q\in\Z$, the maps $T_\sspace\colon \left(\wedge^{\bullet,q}X,\, \del\right) \to \left(\correnti^{\bullet,q}X,\, \del\right)$ and $T_\sspace\colon \left(\wedge^{p,\bullet}X,\, \delbar\right) \to \left(\correnti^{p,\bullet}X,\, \delbar\right)$ are quasi-isomorphisms.
 Furthermore, by Lemma \ref{lem:Teta-surj-BCcoh}, the induced map
 $$ T_\sspace\colon \frac{\ker\left(\de\colon \wedge^{\bullet}X\otimes \C\to \wedge^{\bullet+1}X\otimes\C\right) \cap \wedge^{p,q}X}{\imm\left(\de\colon \wedge^{\bullet-1}X \otimes\C \to \wedge^{\bullet}X\otimes\C\right)} \to \frac{\ker\left(\de\colon \correnti^{\bullet}X\otimes \C\to \correnti^{\bullet+1}X\otimes\C\right) \cap \correnti^{p,q}X}{\imm\left(\de\colon \correnti^{\bullet-1}X \otimes\C \to \correnti^{\bullet}X\otimes\C\right)} $$
 is surjective.
 Hence, Theorem \ref{surjBC} applies, yielding that the map $T_\sspace$ induces a surjective map in Bott-Chern cohomology.
\end{proof}

\begin{rem}
 Given $X$ a compact complex manifold of complex dimension $n$ and $G$ a finite group of biholomorphisms of $X$, consider the compact complex \emph{orbifold} $\tilde X:=\left.X \middle\slash G \right.$ of complex dimension $n$ (namely, \cite[Definition 2]{satake}, $\tilde X$ is a singular complex space whose singularities are locally isomorphic to quotient singularities $\left. \C^n \middle\slash G \right.$ with $G\subset \GL\left(\C^n\right)$ finite; see \cite[Theorem 1]{bochner}, see also \cite[Theorem 1.7.2]{raissy-master-thesis}).

 By extending the action of $G$ on $X$ to $\wedge^\bullet X$, respectively $\wedge^{\bullet,\bullet}X$, set $\wedge^\bullet \tilde X$ the space of $G$-invariant forms in $\wedge^\bullet X$, respectively $\wedge^{\bullet,\bullet}\tilde X$ the space of $G$-invariant forms in $\wedge^{\bullet,\bullet}X$. Analogously, consider $\correnti^\bullet \tilde X$ the space of $G$-invariant currents in $\correnti^\bullet X$, respectively $\correnti^{\bullet,\bullet}\tilde X$ the space of $G$-invariant currents in $\correnti^{\bullet,\bullet}X$.

 Consider the sub-complex $T_\sspace \colon \left(\wedge^{\bullet,\bullet}\tilde X,\, \del,\, \delbar\right) \hookrightarrow \left(\correnti^{\bullet,\bullet}\tilde X,\, \del,\, \delbar\right)$.
 By W.~L. Baily's result \cite[page 807]{baily}, and arguing as in Remark \ref{rem:other-cohom-inj} by means of a Hermitian metric on $\tilde X$, namely, a $G$-invariant Hermitian metric on $X$, it follows that, for any $p\in\Z$, the induced inclusion $T_\sspace \colon \left(\wedge^{p,\bullet} \tilde X ,\, \delbar\right) \hookrightarrow \left(\correnti^{p,\bullet}\tilde X ,\, \delbar\right)$ is a quasi-isomorphism; by conjugation, it follows also that, for any $q\in\Z$, the induced inclusion $T_\sspace \colon \left(\wedge^{\bullet,q} \tilde X ,\, \del\right) \hookrightarrow \left(\correnti^{\bullet,q}\tilde X ,\, \del\right)$ is a quasi-isomorphism. In particular, by using Proposition \ref{spdb}, one recovers that the induced inclusion $T_\sspace \colon \left(\wedge^\bullet \tilde X ,\, \de\right) \hookrightarrow \left(\correnti^\bullet\tilde X ,\, \de\right)$ is a quasi-isomorphism, as proved also by I. Satake, \cite[Theorem 1]{satake}.

 We note that the inclusion $T_\sspace\colon \wedge^{\bullet,\bullet}\tilde X \to \correnti^{\bullet,\bullet}\tilde X$ induces the surjective map
 \begin{eqnarray*}
 \lefteqn{ T_\sspace \colon \frac{\ker\left(\de\colon \wedge^{p+q}\tilde X\otimes_\R\C \to \wedge^{p+q+1}\tilde X\otimes_\R\C \right) \cap \wedge^{p,q}\tilde X}{\imm\left(\de\colon \wedge^{p+q-1}\tilde X\otimes_\R\C \to \wedge^{p+q}\tilde X\otimes_\R\C \right)} } \\[5pt]
 &\to& \frac{\ker\left(\de\colon \correnti^{p+q}\tilde X\otimes_\R\C \to \correnti^{p+q+1}\tilde X\otimes_\R\C \right) \cap \correnti^{p,q}\tilde X}{\imm\left(\de\colon \correnti^{p+q-1}\tilde X\otimes_\R\C \to \correnti^{p+q}\tilde X\otimes_\R\C \right)} \;;
 \end{eqnarray*}
 indeed, since $g^*\circ T \circ g^*=T$ for any $g\in G$, the regularization (see \cite[Theorem III.12]{derham}) of a $G$-invariant current of bidegree $(p,q)$ gives a $G$-invariant $(p,q)$-form.

 Hence, Theorem \ref{surjBC} applies, yielding that, for any $\left(p,q\right)\in\Z^2$, the inclusion $T_\sspace$ induces an isomorphism
 $$ T_\sspace \colon \frac{\ker\left(\de \colon \wedge^{p,q}\tilde X \to \wedge^{p+1,q}\tilde X \oplus \wedge^{p,q+1}\tilde X\right)}{\imm\left(\del\delbar \colon \wedge^{p-1,q-1}\tilde X \to \wedge^{p,q}\tilde X\right)} \stackrel{\simeq}{\to} \frac{\ker\left(\de \colon \correnti^{p,q}\tilde X \to \correnti^{p+1,q}\tilde X \oplus \correnti^{p,q+1}\tilde X\right)}{\imm\left(\del\delbar \colon \correnti^{p-1,q-1}\tilde X \to \correnti^{p,q}\tilde X\right)} \;, $$ 
 as proved also in \cite[Theorem 1]{angella-2}.

 Note that one can argue also by means of the sheaf-theoretic interpretation of the Bott-Chern and Aeppli cohomologies, developed by J.-P. Demailly, \cite[\S V I.12.1]{demailly-agbook} and M. Schweitzer, \cite[\S4]{schweitzer}, see also \cite[\S3.2]{kooistra}.
\end{rem}

\begin{rem}[{\cite{angella-kasuya-3}}]
  We note that the results in Section \ref{sec:subcplx} can be used also to investigate the symplectic Bott-Chern and Aeppli cohomologies, as introduced and studied by L.-S. Tseng and S.-T. Yau in \cite{tseng-yau-1, tseng-yau-2, tseng-yau-3}, for solvmanifolds endowed with left-invariant symplectic structures. In particular, one gets a different proof of the result in \cite[Theorem 3]{macri} by M. Macrì for completely-solvable solvmanifolds, and a generalization for (non-necessarily completely-solvable) solvmanifolds. The complex parallelizable Nakamura manifold $\left. \Gamma \middle\backslash G \right.$ can be investigated explicitly, also in relation with the validity of the $\de\de^{\Lambda}$-lemma, equivalently, the Hard Lefschetz Condition; see also \cite{kasuya-osaka}. We refer to \cite{angella-kasuya-3} for more details.
\end{rem}

\section{Examples}

\begin{ex}[The completely-solvable Nakamura manifold, {\cite[Example 1]{kasuya-mathz}}]\label{CSN}
The completely-solvable Nakamura manifold, firstly studied by I. Nakamura in \cite[page 90]{nakamura}, is an example of a cohomologically K\"ahler non-K\"ahler solvmanifold, \cite{deandres-fernandez-deleon-mencia}, \cite[Example 3.1]{fernandez-munoz-santisteban}, \cite[\S3]{debartolomeis-tomassini}.

Let $G:=\C\ltimes _{\phi}\C^{2}$, where
$$ \phi\left(x+\sqrt{-1}\,y\right) \;:=\;
\left(
\begin{array}{cc}
\esp^{x}& 0  \\
0&    \esp^{-x}  
\end{array}
\right) \in\GL\left(\C^2\right)\;.
$$
Then for some $a\in \R$  the matrix $\left(
\begin{array}{cc}
\esp^{x}& 0  \\
0&    \esp^{-x}  
\end{array}
\right)$
is conjugate to an element of $\mathrm{SL}(2;\Z)$.
We have a lattice $\Gamma := \left(a\,\Z+b\,\sqrt{-1}\,\Z\right)\ltimes_\phi \Gamma^{\prime\prime}$ such that $\Gamma^{\prime\prime} $ is a lattice of $\C^{2}$.
Consider the completely-solvable solvmanifold $\solvmfd$.

(As a matter of notation, we consider holomorphic coordinates $\left\{z_1,\, z_2,\, z_3\right\}$, where $\left\{ z_1:=x+\sqrt{-1}\,y \right\}$ is the holomorphic coordinate on $\C$, and we shorten, for example, $\esp^{-z_1}\de z_{12\bar1}:=\esp^{-z_1}\de z_{1}\wedge \de z_{2}\wedge \de\bar z_{1}$.)

By A. Hattori's theorem, \cite[Corollary 4.2]{hattori}, the de Rham cohomology of $\solvmfd$ does not depend on $\Gamma$ and can be computed using just $G$-left-invariant forms on $\solvmfd$; more precisely, one gets
\begin{eqnarray*}
 H^0_{dR}(\solvmfd;\R) &=& \R\left\langle 1 \right\rangle \;, \\[5pt]
 H^1_{dR}(\solvmfd;\R) &=& \R\left\langle \de z_1 ,\; \de \bar z_1 \right\rangle \;, \\[5pt]
 H^2_{dR}(\solvmfd;\R) &=& \R\left\langle \de z_{23} ,\; \de z_{1\bar1} ,\; \de z_{2\bar3} ,\; \de z_{3\bar2} ,\; \de z_{\bar2\bar3} \right\rangle \;, \\[5pt]
 H^3_{dR}(\solvmfd;\R) &=& \R\left\langle \de z_{123} ,\; \de z_{23\bar1} ,\; \de z_{12\bar3} ,\; \de z_{13\bar2} ,\; \de z_{1\bar2\bar3} ,\; \de z_{2\bar1\bar3} ,\; \de z_{3\bar1\bar2} ,\; \de z_{\bar1\bar2\bar3} \right\rangle \;, \\[5pt]
 H^4_{dR}(\solvmfd;\R) &=& \R\left\langle \de z_{123\bar1} ,\; \de z_{12\bar1\bar3} ,\; \de z_{23\bar2\bar3} ,\; \de z_{13\bar1\bar2} ,\; \de z_{1\bar1\bar2\bar3} \right\rangle \;, \\[5pt]
 H^5_{dR}(\solvmfd;\R) &=& \R\left\langle \de z_{123\bar2\bar3} ,\; \de z_{23\bar1\bar2\bar3} \right\rangle \;, \\[5pt]
 H^6_{dR}(\solvmfd;\R) &=& \R\left\langle \de z_{123\bar1\bar2\bar3} \right\rangle \;,
\end{eqnarray*}
where we have listed the harmonic representatives with respect to the $G$-left-invariant Hermitian metric $g := \de z_1 \odot \de \bar z_1 + \esp^{-z_1-\bar z_1}\de z_2 \odot \de \bar z_2 + \esp^{z_1+\bar z_1}\de z_3 \odot \de \bar z_3$ instead of their cohomology classes.

Here, in the notation as above, we have $\alpha_1(x+\sqrt{-1}\,y)=\exp(x)$ whence $\beta_1(x+\sqrt{-1}\,y)=\gamma_1(x+\sqrt{-1}\,y)=\exp(-\sqrt{-1}\,y)$, and $\alpha_2(x+\sqrt{-1}\,y)=\exp(-x)$ whence $\beta_2(x+\sqrt{-1}\,y)=\gamma_2(x+\sqrt{-1}\,y)=\exp(\sqrt{-1}\,y)$; so that $\alpha_1\beta_1^{-1}=\bar\alpha_1\gamma_1^{-1}=\exp(z)$ and $\alpha_2\beta_2^{-1}=\bar\alpha_2\gamma_2^{-1}=\exp(-z)$.

We consider  $C^{\bullet,\bullet} _{\Gamma}$ as in \eqref{eq:def-c}.
The bi-differential bi-graded algebra $B^{\bullet,\bullet} _{\Gamma}$ varies for a choice of $b$.
By using Theorem \ref{BCISO}, we compute $H^{\bullet,\bullet}_{BC}(\Gamma\backslash G)\simeq H^{\bullet,\bullet}_{BC}(C^{\bullet,\bullet}_{\Gamma})$, case by case:
\begin{enumerate}
\item\label{item:nakamura-1} $b=2m\pi$ for some integer $m\in\Z$;
\item\label{item:nakamura-2}  $b=(2m+1)\pi$ for some integer $m\in\Z$; 
\item\label{item:nakamura-3}$b\not=m\pi$ for any integer $m\in\Z$.
\end{enumerate}
Firstly, we write down $C^{\bullet,\bullet} _{\Gamma}$ case by case in Table \ref{table:c-nakamura-1}, Table \ref{table:c-nakamura-2}, and Table \ref{table:c-nakamura-3}.

\begin{center}
\begin{table}[tb]
 \centering
\begin{tabular}{>{$\mathbf\bgroup}l<{\mathbf\egroup$} || >{$}l<{$}}
\toprule
\text{case \ref{item:nakamura-1}} & C^{\bullet,\bullet} _{\Gamma}\\
\toprule
(0,0) & \C \left\langle 1 \right\rangle \\
\midrule[0.02em]
(1,0) & \C \left\langle \de z_{1},\; \esp^{-z_{1}}\de z_{2},\; \esp^{z_{1}}\de z_{3},\; \esp^{-\bar z_{1}}\de z_{2},\; \esp^{\bar z_{1}}\de z_{3} \right\rangle  \\[5pt]
(0,1) & \C \left\langle \de z_{\bar1},\; \esp^{-z_{1}}\de z_{\bar2},\; \esp^{z_{1}}\de z_{\bar3},\; \esp^{-\bar z_{1}}\de z_{\bar2},\; \esp^{\bar z_{1}}\de z_{\bar3} \right\rangle \\
\midrule[0.02em]
(2,0) & \C \left\langle \esp^{-z_{1}}\de z_{12},\; \esp^{z_{1}}\de z_{13},\; \de z_{23},\; \esp^{-\bar z_{1}}\de z_{12},\; \esp^{\bar z_{1}}\de z_{13} \right\rangle \\[5pt]
(1,1) & \C \left\langle \de z_{1\bar1},\; \esp^{-z_{1}}\de z_{1\bar2},\; \esp^{z_{1}}\de z_{1\bar3},\; \esp^{-z_{1}}\de z_{2\bar1},\; \esp^{-2z_{1}}\de z_{2\bar2},\; \de z_{2\bar3},\; \esp^{z_{1}}\de z_{3\bar1},\; \de z_{3\bar2},\; \esp^{2z_{1}}\de z_{3\bar3}, \right. \\[5pt]
& \left.\esp^{-\bar z_{1}}\de z_{2\bar1},\; \esp^{-\bar z_{1}}\de z_{1\bar2},\; \esp^{\bar z_{1}}\de z_{1\bar3},\; \esp^{\bar z_{1}}\de z_{3\bar1},\; \esp^{-2\bar z_{1}} \de z_{2\bar2},\; \esp^{2\bar z_{1}}\de z_{3\bar3} \right\rangle \\[5pt]
(0,2) & \C \left\langle \esp^{-z_{1}} \de z_{\bar1\bar2},\; \esp^{z_{1}} \de z_{\bar1\bar3},\; \de z_{\bar 2\bar3},\; \esp^{-\bar z_{1}}\de z_{\bar1\bar2},\; \esp^{\bar z_{1}}\de z_{\bar1\bar3} \right\rangle \\
\midrule[0.02em]
(3,0) & \C \left\langle \de z_{123} \right\rangle \\[5pt]
(2,1) & \C \left\langle \esp^{-z_{1}}\de z_{12\bar1},\; \esp^{-2 z_{1}}\de z_{12\bar2},\; \de z_{12\bar3},\; \esp^{z_{1}}\de z_{13\bar1},\; \de z_{13\bar2},\; \esp^{2z_{1}}\de z_{13\bar3},\; \de z_{23\bar1},\; \esp^{-z_{1}}\de z_{23\bar2},\; \esp^{z_{1}}\de z_{23\bar3}, \right. \\[5pt]
& \left. \esp^{-\bar z_{1}}\de z_{12\bar1},\; \esp^{\bar z_{1}}\de z_{13\bar1},\; \esp^{-2\bar z_{1}}\de z_{12\bar2},\; \esp^{-\bar z_{1}}\de z_{23\bar2},\; \esp^{2\bar z_{1}}\de z_{13\bar3},\; \esp^{\bar z_{1}}\de z_{23\bar3} \right\rangle \\[5pt]
(1,2) & \C \left\langle \esp^{-\bar z_{1}}\de z_{1\bar1\bar2},\; \esp^{-2\bar z_{1}}\de z_{2\bar1\bar2},\; \de z_{3\bar1\bar2},\; \esp^{\bar z_{1}}\de z_{1\bar1\bar3},\; \de z_{2\bar1\bar3},\; \esp^{2\bar z_{1}}\de z_{3\bar1\bar3},\; \de z_{1\bar2\bar3},\; \esp^{-\bar z_{1}}\de z_{2\bar2\bar3},\; \esp^{\bar z_{1}}\de z_{3\bar2\bar3}, \right. \\[5pt]
& \left. \esp^{ -z_{1}}\de z_{1\bar1\bar2},\; \esp^{z_{1}}\de z_{1\bar1\bar3},\; \esp^{-2 z_{1}}\de z_{2\bar1\bar2},\; \esp^{-z_{1}}\de z_{2\bar2\bar3},\; \esp^{2 z_{1}}\de z_{3\bar1\bar3},\; \esp^{ z_{1}}\de z_{3\bar2\bar3} \right\rangle \\[5pt]
(0,3) & \C \left\langle \de z_{\bar1\bar2\bar3} \right\rangle \\
\midrule[0.02em]
(3,1) & \C \left\langle \de z_{123\bar1},\; \esp^{-z_{1}}\de z_{123\bar2},\; \esp^{z_{1}}\de z_{123\bar3},\; \esp^{-\bar z_{1}}\de z_{123\bar2},\; \esp^{\bar z_{1}}\de z_{123\bar3} \right\rangle \\[5pt]
(2,2) & \C \left\langle \esp^{-2z_{1}}\de z_{12\bar1\bar2},\; \de z_{12\bar1\bar3},\; \esp^{-z_{1}}\de z_{12\bar2\bar3},\; \de z_{13\bar1\bar2},\; \esp^{2z_{1}}\de z_{13\bar1\bar3},\; \esp^{z_{1}}\de z_{13\bar2\bar3},\; \esp^{-z_{1}}\de z_{23\bar1\bar2},\; \esp^{z_{1}}\de z_{23\bar1\bar3}, \right.\\[5pt]
& \left. \de z_{23\bar2\bar3}, \; \esp^{-2\bar z_{1}}\de z_{12\bar1\bar2},\; \esp^{-\bar z_{1}}\de z_{23\bar1\bar2},\; \esp^{-\bar z_{1}}\de z_{12\bar2\bar3},\; \esp^{\bar z_{1}}\de z_{13\bar2\bar3},\; \esp^{2\bar z_{1}}\de z_{13\bar1\bar3},\; \esp^{\bar z_{1}}\de z_{23\bar1\bar3} \right\rangle \\[5pt]
(1,3) & \C \left\langle \de z_{1\bar1\bar2\bar3},\; \esp^{-\bar z_{1}}\de z_{2\bar1\bar2\bar3},\; \esp^{\bar z_{1}}\de z_{3\bar1\bar2\bar3},\; \esp^{- z_{1}}\de z_{2\bar1\bar2\bar3},\; \esp^{ z_{1}}\de z_{3\bar1\bar2\bar3} \right\rangle \\
\midrule[0.02em]
(3,2) & \C \left\langle \esp^{-z_{1}}\de z_{123\bar1\bar2},\; \esp^{z_{1}}\de z_{123\bar1\bar3},\; \de z_{123\bar2\bar3},\; \esp^{-\bar z_{1}}\de z_{123\bar1\bar2},\; \esp^{\bar z_{1}}\de z_{123\bar1\bar3} \right\rangle \\[5pt]
(2,3) & \C \left\langle \esp^{-z_{1}}\de z_{12\bar1\bar2\bar3},\; \esp^{z_{1}}\de z_{13\bar1\bar2\bar3},\; \de z_{23\bar1\bar2\bar3},\; 
\esp^{-\bar z_{1}}\de z_{12\bar1\bar2\bar3},\; \esp^{\bar z_{1}}\de z_{13\bar1\bar2\bar3} \right\rangle \\
\midrule[0.02em]
(3,3) & \C \left\langle \de z_{123\bar1\bar2\bar3} \right\rangle \\
\bottomrule
\end{tabular}
\caption{The double complex $C^{\bullet,\bullet}_\Gamma$ for the completely-solvable Nakamura manifold in case \ref{item:nakamura-1}.}
\label{table:c-nakamura-1}
\end{table}
\end{center}

\begin{center}
\begin{table}[tb]
 \centering
\begin{tabular}{>{$\mathbf\bgroup}l<{\mathbf\egroup$} || >{$}l<{$}}
\toprule
\text{case \ref{item:nakamura-2}} & C^{\bullet,\bullet}_{\Gamma} \\
\toprule
(0,0) & \C \left\langle 1 \right\rangle \\
\midrule[0.02em]
(1,0) & \C \left\langle \de z_{1} \right\rangle  \\[5pt]
(0,1) & \C \left\langle \de z_{\bar1} \right\rangle \\
\midrule[0.02em]
(2,0) & \C \left\langle \de z_{23} \right\rangle \\[5pt]
(1,1) & \C \left\langle \de z_{1\bar1},\; \esp^{-2z_{1}}\de z_{2\bar2},\; \esp^{-2\bar z_{1}}\de z_{2\bar2},\; \esp^{2z_{1}}\de z_{3\bar3},\; \esp^{2\bar z_{1}}\de z_{3\bar3},\; \de z_{2\bar3},\; \de z_{3\bar2} \right\rangle \\[5pt]
(0,2) & \C \left\langle \de z_{\bar2\bar3} \right\rangle \\
\midrule[0.02em]
(3,0) & \C \left\langle \de z_{123} \right\rangle \\[5pt]
(2,1) & \C \left\langle \de z_{23\bar1},\; \esp^{-2z_{1}}\de z_{12\bar2},\; \esp^{-2\bar z_{1}}\de z_{12\bar2},\; \esp^{2z_{1}}\de z_{13\bar3},\; \esp^{2\bar z_{1}}\de z_{13\bar3},\; \de z_{12\bar3},\; \de z_{13\bar2} \right\rangle \\[5pt]
(1,2) & \C \left\langle \de z_{1\bar2\bar3},\; \esp^{-2z_{1}}\de z_{2\bar1\bar2},\; \esp^{-2\bar z_{1}}\de z_{2\bar1\bar2},\; \esp^{2z_{1}}\de z_{3\bar1\bar3},\; \esp^{2\bar z_{1}}\de z_{3\bar1\bar3},\; \de z_{2\bar1\bar3},\; \de z_{3\bar1\bar2} \right\rangle \\[5pt]
(0,3) & \C \left\langle \de z_{\bar1\bar2\bar3} \right\rangle \\
\midrule[0.02em]
(3,1) & \C \left\langle \de z_{123\bar1} \right\rangle \\[5pt]
(2,2) & \C \left\langle \de z_{12\bar1\bar3},\; \esp^{-2z_{1}}\de z_{12\bar1\bar2},\; \esp^{-2\bar z_{1}}\de z_{12\bar1\bar2},\; \esp^{2z_{1}}\de z_{13\bar1\bar3},\; \esp^{2\bar z_{1}}\de z_{13\bar1\bar3},\; \de z_{23\bar2\bar3},\; \de z_{13\bar1\bar2} \right\rangle \\[5pt]
(1,3) & \C \left\langle \de z_{1\bar1\bar2\bar3} \right\rangle \\
\midrule[0.02em]
(3,2) & \C \left\langle \de z_{123\bar2\bar3}  \right\rangle \\[5pt]
(2,3) & \C \left\langle \de z_{23\bar1\bar2\bar 3} \right\rangle \\
\midrule[0.02em]
(3,3) & \C \left\langle \de z_{123\bar1\bar2\bar3} \right\rangle \\
\bottomrule
\end{tabular}
\caption{The double complex $C^{\bullet,\bullet}_\Gamma$ for the completely-solvable Nakamura manifold in case \ref{item:nakamura-2}.}
\label{table:c-nakamura-2}
\end{table}
\end{center}

\begin{center}
\begin{table}[tb]
 \centering
\begin{tabular}{>{$\mathbf\bgroup}l<{\mathbf\egroup$} || >{$}l<{$}}
\toprule
\text{case \ref{item:nakamura-3}} & C^{\bullet,\bullet}_{\Gamma} \\
\toprule
(0,0) & \C \left\langle 1 \right\rangle \\
\midrule[0.02em]
(1,0) & \C \left\langle \de z_{1} \right\rangle  \\[5pt]
(0,1) & \C \left\langle \de z_{\bar1} \right\rangle \\
\midrule[0.02em]
(2,0) & \C \left\langle \de z_{23} \right\rangle \\[5pt]
(1,1) & \C \left\langle \de z_{1\bar1},\;  \de z_{2\bar3},\; \de z_{3\bar2} \right\rangle \\[5pt]
(0,2) & \C \left\langle \de z_{\bar2\bar3} \right\rangle \\
\midrule[0.02em]
(3,0) & \C \left\langle \de z_{123} \right\rangle \\[5pt]
(2,1) & \C \left\langle \de z_{23\bar1},\; \de z_{12\bar3},\; \de z_{13\bar2} \right\rangle \\[5pt]
(1,2) & \C \left\langle \de z_{1\bar2\bar3},\; \de z_{2\bar1\bar3},\; \de z_{3\bar1\bar2} \right\rangle \\[5pt]
(0,3) & \C \left\langle \de z_{\bar1\bar2\bar3} \right\rangle \\
\midrule[0.02em]
(3,1) & \C \left\langle \de z_{123\bar1} \right\rangle \\[5pt]
(2,2) & \C \left\langle \de z_{12\bar1\bar3},\;  \de z_{23\bar2\bar3},\; \de z_{13\bar1\bar2} \right\rangle \\[5pt]
(1,3) & \C \left\langle \de z_{1\bar1\bar2\bar3} \right\rangle \\
\midrule[0.02em]
(3,2) & \C \left\langle \de z_{123\bar2\bar3} \right\rangle \\[5pt]
(2,3) & \C \left\langle \de z_{23\bar1\bar2\bar 3} \right\rangle \\
\midrule[0.02em]
(3,3) & \C \left\langle \de z_{123\bar1\bar2\bar3}\right\rangle \\
\bottomrule
\end{tabular}
\caption{The double complex $C^{\bullet,\bullet}_\Gamma$ for the completely-solvable Nakamura manifold in case \ref{item:nakamura-3}.}
\label{table:c-nakamura-3}
\end{table}
\end{center}

Note that, since $\del\delbar\left(C^{\bullet,\bullet} _{\Gamma}\right)=\{0\}$ for each case, we have, by using Theorem \ref{BCISO},
$$ H^{\bullet,\bullet}_{BC}(\solvmfd) \;\simeq\; H^{\bullet,\bullet}_{BC}\left(C^{\bullet,\bullet} _{\Gamma}\right) \;=\; \ker \de\lfloor_{C^{\bullet,\bullet} _{\Gamma}} \;. $$
Hence, we compute the Bott-Chern cohomology of the Nakamura manifold case by case in Table \ref{table:BC-nakamura-1} and Table \ref{table:BC-nakamura-2}; note that, in the case \ref{item:nakamura-3}, simply we have:
\begin{equation}\label{eq:BC-nakamura-3}
H_{BC}^{\bullet,\bullet}\left(\solvmfd\right)\simeq C^{\bullet,\bullet}_{\Gamma} \qquad \text{ in case \ref{item:nakamura-3}} \;.
\end{equation}

\begin{center}
\begin{table}[tb]
 \centering
\begin{tabular}{>{$\mathbf\bgroup}l<{\mathbf\egroup$} || >{$}l<{$}}
\toprule
\text{case \ref{item:nakamura-1}} & H^{\bullet,\bullet}_{BC}(\Gamma\backslash G) \\
\toprule
(0,0) & \C \left\langle 1 \right\rangle \\
\midrule[0.02em]
(1,0) & \C \left\langle [\de z_{1}] \right\rangle  \\[5pt]
(0,1) & \C \left\langle [\de z_{\bar1}]\right\rangle \\
\midrule[0.02em]
(2,0) & \C \left\langle [\esp^{-z_{1}}\de z_{12}],\; [\esp^{z_{1}}\de z_{13}],\; [\de z_{23}]\right\rangle \\[5pt]
(1,1) & \C \left\langle [\de z_{1\bar1}],\; [\esp^{-z_{1}}\de z_{1\bar2}],\; [\esp^{z_{1}}\de z_{1\bar3}],\; [ \de z_{2\bar3}],\; [\de z_{3\bar2}],\;
[\esp^{-\bar z_{1}}\de z_{2\bar1}],\;  [ \esp^{\bar z_{1}}\de z_{3\bar1}] \right\rangle \\[5pt]
(0,2) & \C \left\langle [\de z_{\bar 2\bar3}],\; [\esp^{-\bar z_{1}}\de z_{\bar1\bar2}],\; [\esp^{\bar z_{1}}\de z_{\bar1\bar3}] \right\rangle \\
\midrule[0.02em]
(3,0) & \C \left\langle [\de z_{123}] \right\rangle \\[5pt]
(2,1) & \C \left\langle [\esp^{-z_{1}}\de z_{12\bar1}],\; [\esp^{-2z_{1}}\de z_{12\bar2}],\; [\de z_{12\bar3}],\; [\esp^{z_{1}}\de z_{13\bar1}],\; [\de z_{13\bar2}],\; [\esp^{2z_{1}}\de z_{13\bar3}],\; [\de z_{23\bar1}], \right. \\[5pt]
      & \left. [\esp^{-\bar z_{1}}\de z_{12\bar1}],\; [\esp^{\bar z_{1}}\de z_{13\bar1}] \right\rangle \\[5pt]
(1,2) & \C \left\langle [\esp^{-\bar z_{1}}\de z_{1\bar1\bar2}],\; [\esp^{-2\bar z_{1}}\de z_{2\bar1\bar2}],\; [\de z_{3\bar1\bar2}],\; [\esp^{\bar z_{1}}\de z_{1\bar1\bar3}],\; [\de z_{2\bar1\bar3}],\; [\esp^{2\bar z_{1}}\de z_{3\bar1\bar3}],\; [\de z_{1\bar2\bar3}], \right. \\[5pt]
      & \left. [\esp^{ -z_{1}}\de z_{1\bar1\bar2}],\; [\esp^{z_{1}}\de z_{1\bar1\bar3}] \right\rangle \\[5pt]
(0,3) & \C \left\langle [\de z_{\bar1\bar2\bar3}] \right\rangle \\
\midrule[0.02em]
(3,1) & \C \left\langle [\de z_{123\bar1}],\; [\esp^{-z_{1}}\de z_{123\bar2}],\; [\esp^{z_{1}}\de z_{123\bar3}] \right\rangle \\[5pt]
(2,2) & \C \left\langle [\esp^{-2z_{1}}\de z_{12\bar1\bar2}],\; [\de z_{12\bar1\bar3}],\; [\esp^{-z_{1}}\de z_{12\bar2\bar3}],\; [\de z_{13\bar1\bar2}],\; [\esp^{2z_{1}}\de z_{13\bar1\bar3}],\; [\esp^{z_{1}}\de z_{13\bar2\bar3}],\;  [ \de z_{23\bar2\bar3}], \right.\\[5pt]
& \left. [\esp^{-2\bar z_{1}}\de z_{12\bar1\bar2}],\; [\esp^{-\bar z_{1}}\de z_{23\bar1\bar2}],\; [\esp^{2\bar z_{1}}\de z_{13\bar1\bar3}],\; [\esp^{\bar z_{1}}\de z_{23\bar1\bar3} ] \right\rangle \\[5pt]
(1,3) & \C \left\langle [\de z_{1\bar1\bar2\bar3}],\; [\esp^{-\bar z_{1}}\de z_{2\bar1\bar2\bar3}],\; [\esp^{\bar z_{1}}\de z_{3\bar1\bar2\bar3}] \right\rangle \\
\midrule[0.02em]
(3,2) & \C \left\langle [\esp^{-z_{1}}\de z_{123\bar1\bar2}],\; [\esp^{z_{1}}\de z_{123\bar1\bar3}],\; [\de z_{123\bar2\bar3}],\; [\esp^{-\bar z_{1}}\de z_{123\bar1\bar2}],\; [\esp^{\bar z_{1}}\de z_{123\bar1\bar3}] \right\rangle \\[5pt]
(2,3) & \C \left\langle [\esp^{-z_{1}}\de z_{12\bar1\bar2\bar3}],\; [\esp^{z_{1}}\de z_{13\bar1\bar2\bar3}],\; [\de z_{23\bar1\bar2\bar3}],\; 
[\esp^{-\bar z_{1}}\de z_{12\bar1\bar2\bar3}],\; [\esp^{\bar z_{1}}\de z_{13\bar1\bar2\bar3}] \right\rangle \\
\midrule[0.02em]
(3,3) & \C \left\langle [\de z_{123\bar1\bar2\bar3}] \right\rangle \\
\bottomrule
\end{tabular}
\caption{The Bott-Chern cohomology of the completely-solvable Nakamura manifold in case \ref{item:nakamura-1}.}
\label{table:BC-nakamura-1}
\end{table}
\end{center}

\begin{center}
\begin{table}[tb]
 \centering
\begin{tabular}{>{$\mathbf\bgroup}l<{\mathbf\egroup$} || >{$}l<{$}}
\toprule
\text{case \ref{item:nakamura-2}} & H_{BC}^{\bullet,\bullet}(\Gamma \backslash G)\\
\toprule
(0,0) & \C \left\langle 1 \right\rangle \\
\midrule[0.02em]
(1,0) & \C \left\langle [\de z_{1} ]\right\rangle  \\[5pt]
(0,1) & \C \left\langle [\de z_{\bar1}] \right\rangle \\
\midrule[0.02em]
(2,0) & \C \left\langle [\de z_{23}] \right\rangle \\[5pt]
(1,1) & \C \left\langle [\de z_{1\bar1}],\; [\de z_{2\bar3}],\; [\de z_{3\bar2}] \right\rangle \\[5pt]
(0,2) & \C \left\langle [\de z_{\bar2\bar3}] \right\rangle \\
\midrule[0.02em]
(3,0) & \C \left\langle [\de z_{123}] \right\rangle \\[5pt]
(2,1) & \C \left\langle [\de z_{23\bar1}],\; [\esp^{-2z_{1}}\de z_{12\bar2}],\; [ \esp^{2z_{1}}\de z_{13\bar3}],\; [\de z_{12\bar3}],\; [\de z_{13\bar2}] \right\rangle \\[5pt]
(1,2) & \C \left\langle [\de z_{1\bar2\bar3}],\; [\esp^{-2\bar z_{1}}\de z_{2\bar1\bar2}],\; [ \esp^{2\bar z_{1}}\de z_{3\bar1\bar3}],\; [\de z_{2\bar1\bar3}],\; [\de z_{3\bar1\bar2}] \right\rangle \\[5pt]
(0,3) & \C \left\langle [\de z_{\bar1\bar2\bar3}] \right\rangle \\
\midrule[0.02em]
(3,1) & \C \left\langle [\de z_{123\bar1}] \right\rangle \\[5pt]
(2,2) & \C \left\langle [\de z_{12\bar1\bar3}],\; [\esp^{-2z_{1}}\de z_{12\bar1\bar2}],\; [\esp^{-2\bar z_{1}}\de z_{12\bar1\bar2}],\; [\esp^{2z_{1}}\de z_{13\bar1\bar3}],\; [\esp^{2\bar z_{1}}\de z_{13\bar1\bar3}],\; [\de z_{23\bar2\bar3}],\; [\de z_{13\bar1\bar2}] \right\rangle \\[5pt]
(1,3) & \C \left\langle [\de z_{1\bar1\bar2\bar3}] \right\rangle \\
\midrule[0.02em]
(3,2) & \C \left\langle [\de z_{123\bar2\bar3}] \right\rangle \\[5pt]
(2,3) & \C \left\langle [\de z_{23\bar1\bar2\bar 3}] \right\rangle \\
\midrule[0.02em]
(3,3) & \C \left\langle [\de z_{123\bar1\bar2\bar3}] \right\rangle\\
\bottomrule
\end{tabular}
\caption{The Bott-Chern cohomology of the completely-solvable Nakamura manifold in case \ref{item:nakamura-2}.}
\label{table:BC-nakamura-2}
\end{table}
\end{center}

We summarize in Table \ref{table:delbar-BC-nakamura} the results of the computations of the Bott-Chern cohomology as done in Table \ref{table:BC-nakamura-1} and Table \ref{table:BC-nakamura-2} and \eqref{eq:BC-nakamura-3}, and of the Dolbeault cohomology, as done in \cite[Example 1]{kasuya-mathz}.

\begin{center}
\begin{table}[tb]
 \centering
\begin{tabular}{>{$\mathbf\bgroup}c<{\mathbf\egroup$} || >{$}c<{$} || >{$}c<{$} >{$}c<{$} | >{$}c<{$} >{$}c<{$} | >{$}c<{$} >{$}c<{$}}
\toprule
 & & \multicolumn{2}{c|}{case \ref{item:nakamura-1}} & \multicolumn{2}{c|}{case \ref{item:nakamura-2}} & \multicolumn{2}{c}{case \ref{item:nakamura-3}} \\
 & dR & \delbar & BC & \delbar & BC & \delbar & BC \\
\toprule
(0,0) & 1 &  1 & 1 & 1 & 1 & 1 & 1 \\
\midrule[0.02em]
(1,0) & \multirow{2}{*}{2} & 3 &  1 & 1 & 1 & 1 & 1 \\[5pt]
(0,1) & & 3 &  1 & 1 & 1 & 1 & 1 \\
\midrule[0.02em]
(2,0) & \multirow{3}{*}{5} & 3 &  3 & 1 & 1 & 1 & 1 \\[5pt]
(1,1) & & 9 &  7 & 5 & 3 & 3 & 3 \\[5pt]
(0,2) & & 3 &  3 & 1 & 1 & 1 & 1 \\
\midrule[0.02em]
(3,0) & \multirow{4}{*}{8} & 1 &  1 & 1 & 1 & 1 & 1 \\[5pt]
(2,1) & & 9 &  9 & 5 & 5 & 3 & 3 \\[5pt]
(1,2) & & 9 &  9 & 5 & 5 & 3 & 3 \\[5pt]
(0,3) & & 1 &  1 & 1 & 1 & 1 & 1 \\
\midrule[0.02em]
(3,1) & \multirow{3}{*}{5} & 3 &  3 & 1 & 1 & 1 & 1 \\[5pt]
(2,2) & & 9 & 11 & 5 & 7 & 3 & 3 \\[5pt]
(1,3) & & 3 &  3 & 1 & 1 & 1 & 1 \\
\midrule[0.02em]
(3,2) & \multirow{2}{*}{2} & 3 &  5 & 1 & 1 & 1 & 1 \\[5pt]
(2,3) & & 3 &  5 & 1 & 1 & 1 & 1 \\
\midrule[0.02em]
(3,3) & 1 &  1 & 1 & 1 & 1 & 1 & 1 \\
\bottomrule
\end{tabular}
\caption{The dimensions of the de Rham, Dolbeault, and Bott-Chern cohomologies of the completely-solvable Nakamura manifold.}
\label{table:delbar-BC-nakamura}
\end{table}
\end{center}

\begin{rem}\label{rem:nakamura-deldelbar-1}
Note that in any case the canonical map $\Tot^\bullet H_{BC}^{\bullet,\bullet}\left(\solvmfd\right)\to H^{\bullet}_{dR}\left(\solvmfd\right)$ is surjective. (With the notation of \cite{li-zhang, angella-tomassini-1}, this means that, in any case, $\solvmfd$ is \emph{complex-$\mathcal{C}^\infty$-pure-and-full at every stage}, namely, the de Rham cohomology admits a decomposition in pure-type subgroups with respect to the complex structure.)
In the case \ref{item:nakamura-3}, by Proposition \ref{spdb}, we have $H^{\bullet}_{dR}\left(\solvmfd\right) \simeq H^{\bullet}\left(\Tot^\bullet C^{\bullet,\bullet}_{\Gamma}\right)=\Tot^\bullet C^{\bullet,\bullet}_{\Gamma}$ and hence the canonical map $\Tot^\bullet H_{BC}^{\bullet,\bullet}\left(\solvmfd\right) \to H^{\bullet}_{dR}\left(\solvmfd\right)$ induced by the identity is in fact an isomorphism:
this implies that $\solvmfd$ in case \ref{item:nakamura-3} satisfies the $\partial\bar\partial$-Lemma (namely, every $\del$-closed $\delbar$-closed $\de$-exact form is $\del\delbar$-exact too, see \cite{deligne-griffiths-morgan-sullivan}).
In \cite{kasuya-mathz}, it is shown that for some left-invariant Hermitian metric the space of harmonic forms admits the Hodge decomposition and symmetry (see also \cite{kasuya-hodge} for higher dimensional examples with the Hodge decomposition and symmetry).
\end{rem}

\begin{rem}
 In view of \cite[Theorem A, Theorem B]{angella-tomassini-3}, stating that, for every compact complex manifold $X$, for any $k\in\Z$, the inequality
 $$ \sum_{p+q=k} \left( \dim_\C H^{p,q}_{BC}(X) + \dim_\C H^{p,q}_{A}(X) \right) \;\geq\; \sum_{p+q=k} \left( \dim_\C H^{p,q}_{\del}(X) + \dim_\C H^{p,q}_{\delbar}(X) \right) \;\geq\; 2\, \dim_\C H^k_{dR}(X;\C) $$
 holds, and that equalities hold for any $k\in\Z$ if and only if $X$ satisfies the $\del\delbar$-Lemma, one gets that the non-negative integer numbers $\sum_{p+q=k} \left( \dim_\C H^{p,q}_{BC}(X) + \dim_\C H^{p,q}_{A}(X) \right) - 2\, \dim_\C H^k_{dR}(X;\C) \in \N$, varying $k\in\Z$, provide a ``measure'' of the non-K\"ahlerianity of $X$.

 Note that, for the completely-solvable Nakamura manifold, in any case, one has
 $$ \dim_\C H^{p,q}_{BC}(X) + \dim_\C H^{p,q}_{A}(X) \;=\; \dim_\C H^{p,q}_{\del}(X) + \dim_\C H^{p,q}_{\delbar}(X) $$
 for any $(p,q)\in\Z^2$. On the other hand,
 $$
 \sum_{p+q=k} \left( \dim_\C H^{p,q}_{BC}(X) + \dim_\C H^{p,q}_{A}(X) \right) - 2\,\dim_\C H^k_{dR}(X;\C) \;=\;
 \left\{
 \begin{array}{ll}
   8 & \text{ for } k\in \{1,\, 5\} \\[5pt]
  20 & \text{ for } k\in \{2,\, 4\} \\[5pt]
  24 & \text{ for } k=3 \\[5pt]
  0  & \text{ otherwise}
 \end{array}
 \right. \qquad \text{ in case \ref{item:nakamura-1}} \;,
 $$
 and
  $$
 \sum_{p+q=k} \left( \dim_\C H^{p,q}_{BC}(X) + \dim_\C H^{p,q}_{A}(X) \right) - 2\,\dim_\C H^k_{dR}(X;\C) \;=\;
 \left\{
 \begin{array}{ll}
   0 & \text{ for } k\in \{1,\, 5\} \\[5pt]
   4 & \text{ for } k\in \{2,\, 4\} \\[5pt]
   8 & \text{ for } k=3 \\[5pt]
   0 & \text{ otherwise}
 \end{array}
 \right. \qquad \text{ in case \ref{item:nakamura-2}} \;,
 $$
 and
  $$
 \sum_{p+q=k} \left( \dim_\C H^{p,q}_{BC}(X) + \dim_\C H^{p,q}_{A}(X) \right) - 2\,\dim_\C H^k_{dR}(X;\C) \;=\;
 \left\{
 \begin{array}{ll}
   0 & \text{ for } k\in \{1,\, 5\} \\[5pt]
   0 & \text{ for } k\in \{2,\, 4\} \\[5pt]
   0 & \text{ for } k=3 \\[5pt]
   0 & \text{ otherwise}
 \end{array}
 \right. \qquad \text{ in case \ref{item:nakamura-3}} \;.
 $$
 In particular, by \cite[Theorem B]{angella-tomassini-3}, one gets that $\solvmfd$ in case \ref{item:nakamura-3} satisfies the $\del\delbar$-Lemma, as noticed also in Remark \ref{rem:nakamura-deldelbar-1}.
\end{rem}

\end{ex}

\medskip

\begin{ex}[The complex parallelizable Nakamura manifold]\label{Nak}
Let $G=\C\ltimes_{\phi} \C^{2}$ be such that \[\phi(z)=\left(
\begin{array}{cc}
\esp^{z}& 0  \\
0&    \esp^{-z}  
\end{array}
\right) \;.\]
Then there exist $a+\sqrt{-1}\,b \in \C$ and $c+\sqrt{-1}\,d\in \C$ such that $ \Z(a+\sqrt{-1}\,b)+\Z(c+\sqrt{-1}\,d)$ is a lattice in $\C$ and $\phi(a+\sqrt{-1}\,b)$ and $\phi(c+\sqrt{-1}\,d)$ are conjugate to elements of $\mathrm{SL}(4;\Z)$, where we regard $\mathrm{SL}(2;\C)\subset \mathrm{SL}(4;\R)$, see \cite{hasegawa-dga}.
Hence we have a lattice $\Gamma := \left( \Z \left( a+\sqrt{-1}\,b \right) + \Z \left( c+\sqrt{-1}\,d \right) \right) \ltimes_{\phi} \Gamma^{\prime\prime}$ of $G$ such that $\Gamma^{\prime\prime}$ is a lattice of $\C^{2}$. Let $X := \left. \Gamma \middle\backslash G \right.$ be the \emph{complex parallelizable Nakamura manifold}, \cite[\S2]{nakamura}.

We take the connected simply-connected complex nilpotent subgroup $C:=\C \subseteq G$ such that $G = C \cdot N$, where $N$ is the nilradical of $G$. Recall that $\g_+$ denotes the Lie algebra of the $G$-left-invariant holomorphic vector fields on $G$.
For a coordinate set $(z_{1},z_{2},z_{3})$ of $\C\ltimes_{\phi} \C^{2}$, we have the basis $\left\{ \frac{\partial}{\partial z_{1}}, \;  \esp^{z_{1}}\,\frac{\partial}{\partial z_{2}}, \; \esp^{-z_{1}}\,\frac{\partial}{\partial z_{3}} \right\}$ of $\g_{+}$ such that
$$ \left( \Ad_{(z_{1},z_{2},z_{3})} \right)_{\mathrm{s}} \;=\; \diag\left(1,\, \esp^{z_{1}},\, \esp^{-z_{1}}\right) \;\in\; \Aut(\g_+) \;. $$

Here, in the notation as above, we have $\alpha_1(z_1)=1$, $\alpha_2(z_1)=\exp(z_1)$, and $\alpha_3(z_1)=\exp(-z_1)$.

\begin{enumerate}[{\itshape (a)}]
 \item\label{item:ex-nakamura2-1}
If $b \in \pi\,\Z$ and $d \in\pi\,\Z$, then, 
for $z\in \left(a+\sqrt{-1}\,b\right)\,\Z + \left(c+\sqrt{-1}\,d\right)\,\Z$, we have $\phi(z)\in \mathrm{SL}(2;\R)$.
Since $\left.\left(\frac{\esp^{z_{1}}}{\esp^{\bar z_{1}}}\right)\middle\lfloor_{\Gamma}\right.=\left.\left(\esp^{z_{1}-\bar z_{1}}\right)\middle\lfloor_{\Gamma}\right.=1$, we have
$$ B_{\Gamma}^{\bullet} \;=\; \wedge^{\bullet} \C \left\langle \de z_{\bar 1},\; \esp^{ z_1}\de z_{\bar 2},\; \esp^{ z_1}\de z_{\bar 3} \right\rangle \;. $$

Hence the double complex $C^{\bullet,\bullet}_{\Gamma}$ in case {\itshape(\ref{item:ex-nakamura2-1})} is the one in Table \ref{table:c-ex-nakamura2-1}. (We recall that, in order to shorten the notation, we write, for example, $\esp^{\bar z_1}\de z_{1\bar3}:=\esp^{\bar z_1}\de z_{1}\wedge\de \bar z_{3}$.)

\begin{center}
\begin{table}[tb]
 \centering
\begin{tabular}{>{$\mathbf\bgroup}l<{\mathbf\egroup$} || >{$}l<{$}}
\toprule
\text{case {\itshape(\ref{item:ex-nakamura2-1})}} & C^{\bullet,\bullet} _{\Gamma}\\
\toprule
(0,0) & \C \left\langle 1 \right\rangle \\
\midrule[0.02em]
(1,0) & \C \left\langle \de z_{1},\; \esp^{-z_{1}}\de z_{2},\; \esp^{z_{1}}\de z_{3},\; \esp^{-\bar z_{1}}\de z_{2},\; \esp^{\bar z_{1}}\de z_{3} \right\rangle  \\[5pt]
(0,1) & \C \left\langle \de z_{\bar1},\; \esp^{-z_{1}}\de z_{\bar2},\; \esp^{z_{1}}\de z_{\bar3},\; \esp^{-\bar z_{1}}\de z_{\bar2},\; \esp^{\bar z_{1}}\de z_{\bar3} \right\rangle \\
\midrule[0.02em]
(2,0) & \C \left\langle \esp^{-z_{1}}\de z_{12},\; \esp^{z_{1}}\de z_{13},\; \de z_{23},\; \esp^{-\bar z_{1}}\de z_{12},\; \esp^{\bar z_{1}}\de z_{13} \right\rangle \\[5pt]
(1,1) & \C \left\langle \de z_{1\bar1},\; \esp^{-z_{1}}\de z_{1\bar2},\; \esp^{z_{1}}\de z_{1\bar3},\; \esp^{-z_{1}}\de z_{2\bar1},\; \esp^{-2z_{1}}\de z_{2\bar2},\; \de z_{2\bar3},\; \esp^{z_{1}}\de z_{3\bar1},\; \de z_{3\bar2},\; \esp^{2z_{1}}\de z_{3\bar3}, \right. \\[5pt]
& \left.\esp^{-\bar z_{1}}\de z_{2\bar1},\; \esp^{-\bar z_{1}}\de z_{1\bar2},\; \esp^{\bar z_{1}}\de z_{1\bar3},\; \esp^{\bar z_{1}}\de z_{3\bar1},\; \esp^{-2\bar z_{1}} \de z_{2\bar2},\; \esp^{2\bar z_{1}}\de z_{3\bar3} \right\rangle \\[5pt]
(0,2) & \C \left\langle \esp^{-z_{1}} \de z_{\bar1\bar2},\; \esp^{z_{1}} \de z_{\bar1\bar3},\; \de z_{\bar 2\bar3},\; \esp^{-\bar z_{1}}\de z_{\bar1\bar2},\; \esp^{\bar z_{1}}\de z_{\bar1\bar3} \right\rangle \\
\midrule[0.02em]
(3,0) & \C \left\langle \de z_{123} \right\rangle \\[5pt]
(2,1) & \C \left\langle \esp^{-z_{1}}\de z_{12\bar1},\; \esp^{-2 z_{1}}\de z_{12\bar2},\; \de z_{12\bar3},\; \esp^{z_{1}}\de z_{13\bar1},\; \de z_{13\bar2},\; \esp^{2z_{1}}\de z_{13\bar3},\; \de z_{23\bar1},\; \esp^{-z_{1}}\de z_{23\bar2},\; \esp^{z_{1}}\de z_{23\bar3}, \right. \\[5pt]
& \left. \esp^{-\bar z_{1}}\de z_{12\bar1},\; \esp^{\bar z_{1}}\de z_{13\bar1},\; \esp^{-2\bar z_{1}}\de z_{12\bar2},\; \esp^{-\bar z_{1}}\de z_{23\bar2},\; \esp^{2\bar z_{1}}\de z_{13\bar3},\; \esp^{\bar z_{1}}\de z_{23\bar3} \right\rangle \\[5pt]
(1,2) & \C \left\langle \esp^{-\bar z_{1}}\de z_{1\bar1\bar2},\; \esp^{-2\bar z_{1}}\de z_{2\bar1\bar2},\; \de z_{3\bar1\bar2},\; \esp^{\bar z_{1}}\de z_{1\bar1\bar3},\; \de z_{2\bar1\bar3},\; \esp^{2\bar z_{1}}\de z_{3\bar1\bar3},\; \de z_{1\bar2\bar3},\; \esp^{-\bar z_{1}}\de z_{2\bar2\bar3},\; \esp^{\bar z_{1}}\de z_{3\bar2\bar3}, \right. \\[5pt]
& \left. \esp^{ -z_{1}}\de z_{1\bar1\bar2},\; \esp^{z_{1}}\de z_{1\bar1\bar3},\; \esp^{-2 z_{1}}\de z_{2\bar1\bar2},\; \esp^{-z_{1}}\de z_{2\bar2\bar3},\; \esp^{2 z_{1}}\de z_{3\bar1\bar3},\; \esp^{ z_{1}}\de z_{3\bar2\bar3} \right\rangle \\[5pt]
(0,3) & \C \left\langle \de z_{\bar1\bar2\bar3} \right\rangle \\
\midrule[0.02em]
(3,1) & \C \left\langle \de z_{123\bar1},\; \esp^{-z_{1}}\de z_{123\bar2},\; \esp^{z_{1}}\de z_{123\bar3},\; \esp^{-\bar z_{1}}\de z_{123\bar2},\; \esp^{\bar z_{1}}\de z_{123\bar3} \right\rangle \\[5pt]
(2,2) & \C \left\langle \esp^{-2z_{1}}\de z_{12\bar1\bar2},\; \de z_{12\bar1\bar3},\; \esp^{-z_{1}}\de z_{12\bar2\bar3},\; \de z_{13\bar1\bar2},\; \esp^{2z_{1}}\de z_{13\bar1\bar3},\; \esp^{z_{1}}\de z_{13\bar2\bar3},\; \esp^{-z_{1}}\de z_{23\bar1\bar2},\; \esp^{z_{1}}\de z_{23\bar1\bar3}, \right.\\[5pt]
& \left. \de z_{23\bar2\bar3}, \; \esp^{-2\bar z_{1}}\de z_{12\bar1\bar2},\; \esp^{-\bar z_{1}}\de z_{23\bar1\bar2},\; \esp^{-\bar z_{1}}\de z_{12\bar2\bar3},\; \esp^{\bar z_{1}}\de z_{13\bar2\bar3},\; \esp^{2\bar z_{1}}\de z_{13\bar1\bar3},\; \esp^{\bar z_{1}}\de z_{23\bar1\bar3} \right\rangle \\[5pt]
(1,3) & \C \left\langle \de z_{1\bar1\bar2\bar3},\; \esp^{-\bar z_{1}}\de z_{2\bar1\bar2\bar3},\; \esp^{\bar z_{1}}\de z_{3\bar1\bar2\bar3},\; \esp^{- z_{1}}\de z_{2\bar1\bar2\bar3},\; \esp^{ z_{1}}\de z_{3\bar1\bar2\bar3} \right\rangle \\
\midrule[0.02em]
(3,2) & \C \left\langle \esp^{-z_{1}}\de z_{123\bar1\bar2},\; \esp^{z_{1}}\de z_{123\bar1\bar3},\; \de z_{123\bar2\bar3},\; \esp^{-\bar z_{1}}\de z_{123\bar1\bar2},\; \esp^{\bar z_{1}}\de z_{123\bar1\bar3} \right\rangle \\[5pt]
(2,3) & \C \left\langle \esp^{-z_{1}}\de z_{12\bar1\bar2\bar3},\; \esp^{z_{1}}\de z_{13\bar1\bar2\bar3},\; \de z_{23\bar1\bar2\bar3},\; 
\esp^{-\bar z_{1}}\de z_{12\bar1\bar2\bar3},\; \esp^{\bar z_{1}}\de z_{13\bar1\bar2\bar3} \right\rangle \\
\midrule[0.02em]
(3,3) & \C \left\langle \de z_{123\bar1\bar2\bar3} \right\rangle \\
\bottomrule
\end{tabular}
\caption{The double complex $C^{\bullet,\bullet}_{\Gamma}$ in \eqref{eq:def-c-cplx-par} for the complex parallelizable Nakamura manifold in case {\itshape(\ref{item:ex-nakamura2-1})}.}
\label{table:c-ex-nakamura2-1}
\end{table}
\end{center}

We compute the Bott-Chern cohomology for the complex parallelizable Nakamura manifold in case {\itshape(\ref{item:ex-nakamura2-1})} in Table \ref{table:BC-ex-nakamura2-1}.

\begin{center}
\begin{table}[tb]
 \centering
\begin{tabular}{>{$\mathbf\bgroup}l<{\mathbf\egroup$} || >{$}l<{$}}
\toprule
\text{case {\itshape(\ref{item:ex-nakamura2-1})}} & H^{\bullet,\bullet}_{BC}(\Gamma\backslash G) \\
\toprule
(0,0) & \C \left\langle 1 \right\rangle \\
\midrule[0.02em]
(1,0) & \C \left\langle [\de z_{1}] \right\rangle  \\[5pt]
(0,1) & \C \left\langle [\de z_{\bar1}]\right\rangle \\
\midrule[0.02em]
(2,0) & \C \left\langle [\esp^{-z_{1}}\de z_{12}],\; [\esp^{z_{1}}\de z_{13}],\; [\de z_{23}]\right\rangle \\[5pt]
(1,1) & \C \left\langle [\de z_{1\bar1}],\; [\esp^{-z_{1}}\de z_{1\bar2}],\; [\esp^{z_{1}}\de z_{1\bar3}],\; [ \de z_{2\bar3}],\; [\de z_{3\bar2}],\;
[\esp^{-\bar z_{1}}\de z_{2\bar1}],\;  [ \esp^{\bar z_{1}}\de z_{3\bar1}] \right\rangle \\[5pt]
(0,2) & \C \left\langle [\de z_{\bar 2\bar3}],\; [\esp^{-\bar z_{1}}\de z_{\bar1\bar2}],\; [\esp^{\bar z_{1}}\de z_{\bar1\bar3}] \right\rangle \\
\midrule[0.02em]
(3,0) & \C \left\langle [\de z_{123}] \right\rangle \\[5pt]
(2,1) & \C \left\langle [\esp^{-z_{1}}\de z_{12\bar1}],\; [\esp^{-2z_{1}}\de z_{12\bar2}],\; [\de z_{12\bar3}],\; [\esp^{z_{1}}\de z_{13\bar1}],\; [\de z_{13\bar2}],\; [\esp^{2z_{1}}\de z_{13\bar3}], \right. \\[5pt]
& \left. [\de z_{23\bar1}],\;  [\esp^{-\bar z_{1}}\de z_{12\bar1}],\; [\esp^{\bar z_{1}}\de z_{13\bar1}] \right\rangle \\[5pt]
(1,2) & \C \left\langle [\esp^{-\bar z_{1}}\de z_{1\bar1\bar2}],\; [\esp^{-2\bar z_{1}}\de z_{2\bar1\bar2}],\; [\de z_{3\bar1\bar2}],\; [\esp^{\bar z_{1}}\de z_{1\bar1\bar3}],\; [\de z_{2\bar1\bar3}],\; [\esp^{2\bar z_{1}}\de z_{3\bar1\bar3}], \right. \\[5pt]
& \left. [\de z_{1\bar2\bar3}],\;  [\esp^{ -z_{1}}\de z_{1\bar1\bar2}],\; [\esp^{z_{1}}\de z_{1\bar1\bar3}] \right\rangle \\[5pt]
(0,3) & \C \left\langle [\de z_{\bar1\bar2\bar3}] \right\rangle \\
\midrule[0.02em]
(3,1) & \C \left\langle [\de z_{123\bar1}],\; [\esp^{-z_{1}}\de z_{123\bar2}],\; [\esp^{z_{1}}\de z_{123\bar3}] \right\rangle \\[5pt]
(2,2) & \C \left\langle [\esp^{-2z_{1}}\de z_{12\bar1\bar2}],\; [\de z_{12\bar1\bar3}],\; [\esp^{-z_{1}}\de z_{12\bar2\bar3}],\; [\de z_{13\bar1\bar2}],\; [\esp^{2z_{1}}\de z_{13\bar1\bar3}],\; [\esp^{z_{1}}\de z_{13\bar2\bar3}], \right.\\[5pt]
& \left. [\de z_{23\bar2\bar3}], \; [\esp^{-2\bar z_{1}}\de z_{12\bar1\bar2}],\; [\esp^{-\bar z_{1}}\de z_{23\bar1\bar2}],\; [\esp^{2\bar z_{1}}\de z_{13\bar1\bar3}],\; [\esp^{\bar z_{1}}\de z_{23\bar1\bar3} ] \right\rangle \\[5pt]
(1,3) & \C \left\langle [\de z_{1\bar1\bar2\bar3}],\; [\esp^{-\bar z_{1}}\de z_{2\bar1\bar2\bar3}],\; [\esp^{\bar z_{1}}\de z_{3\bar1\bar2\bar3}] \right\rangle \\
\midrule[0.02em]
(3,2) & \C \left\langle [\esp^{-z_{1}}\de z_{123\bar1\bar2}],\; [\esp^{z_{1}}\de z_{123\bar1\bar3}],\; [\de z_{123\bar2\bar3}],\; [\esp^{-\bar z_{1}}\de z_{123\bar1\bar2}],\; [\esp^{\bar z_{1}}\de z_{123\bar1\bar3}] \right\rangle \\[5pt]
(2,3) & \C \left\langle [\esp^{-z_{1}}\de z_{12\bar1\bar2\bar3}],\; [\esp^{z_{1}}\de z_{13\bar1\bar2\bar3}],\; [\de z_{23\bar1\bar2\bar3}],\; 
[\esp^{-\bar z_{1}}\de z_{12\bar1\bar2\bar3}],\; [\esp^{\bar z_{1}}\de z_{13\bar1\bar2\bar3}] \right\rangle \\
\midrule[0.02em]
(3,3) & \C \left\langle [\de z_{123\bar1\bar2\bar3}] \right\rangle \\
\bottomrule
\end{tabular}
\caption{The Bott-Chern cohomology of the complex parallelizable Nakamura manifold in case {\itshape (\ref{item:ex-nakamura2-1})}.}
\label{table:BC-ex-nakamura2-1}
\end{table}
\end{center}

The differential algebra $A_{\Gamma}^{\bullet}$ for the complex parallelizable Nakamura manifold in case {\itshape (\ref{item:ex-nakamura2-1})} is summarized in Table \ref{table:a-nak-cplx-paral-1}.

\begin{center}
\begin{table}[tb]
 \centering
\begin{tabular}{>{$\mathbf\bgroup}l<{\mathbf\egroup$} || >{$}l<{$}}
\toprule
\text{case {\itshape (\ref{item:ex-nakamura2-1})}} & A^{\bullet}_{\Gamma} \\
\toprule
0 & \C \left\langle 1 \right\rangle \\
\midrule[0.02em]
1 & \C \left\langle \de z_{1},\;\de  z_{\bar 1}   \right\rangle \\
\midrule[0.02em]
2 & \C \left\langle  \de z_{1\bar1},\; \de z_{23},\; \de z_{2\bar 3},\; \de z_{3\bar 2},\; \de z_{\bar2\bar3} \right\rangle \\
\midrule[0.02em]
3 & \C \left\langle \de z_{123},\; \de z_{12\bar3},\; \de z_{13\bar2},\; \de z_{3\bar1\bar2},\;\de z_{2\bar1\bar3},\;\de z_{\bar1\bar2\bar3},\;  \de z_{\bar123},\;  \de z_{1\bar2\bar3} \right\rangle \\
\midrule[0.02em]
4 & \C \left\langle \de z_{123\bar1},\;\de z_{13\bar1\bar2},\;\de z_{23\bar2\bar3},\; \de z_{12\bar1\bar3},\; \de z_{1\bar1\bar2\bar3} \right\rangle \\
\midrule[0.02em]
5 & \C \left\langle \de z_{23\bar1\bar2\bar 3} ,\; \de z_{123\bar2\bar3} \right\rangle \\
\midrule[0.02em]
6 & \C \left\langle \de z_{123\bar1\bar2\bar3} \right\rangle \\
\bottomrule
\end{tabular}
\caption{The cochain complex $A_{\Gamma}^{\bullet}$ in \eqref{eq:def-a1} for the complex parallelizable Nakamura manifold in case {\itshape (\ref{item:ex-nakamura2-1})}.}
\label{table:a-nak-cplx-paral-1}
\end{table}
\end{center}

\begin{rem}
Suppose  $b \in 2\pi\,\Z$ and $d \in2\pi\,\Z$.
Considering another Lie group $H:=\C\ltimes_{\psi}\C^{2}$ such that
\[\psi(z) \;:=\; \left(
\begin{array}{cc}
\esp^{\frac{1}{2}(z_{1}+\bar z_{1})}& 0  \\
0&    \esp^{-\frac{1}{2}(z_{1}+\bar z_{1})}  
\end{array}
\right),\]
the correspondence $G \in \left( z_{1},\, z_{2},\, z_{3} \right) \mapsto \left( z_{1},\, z_{2},\, z_{3} \right) \in H$ gives an embedding $\Gamma \hookrightarrow H$ as a lattice and hence we can identify $\solvmfd$ with $\Gamma\backslash H$, see \cite[Section 3]{Yamada-GD}.
Since $H$ is equal to the solvable completely-solvable Lie group in Example \ref{CSN}, this case is identified with case \ref{item:nakamura-1} in Example \ref{CSN}.
Note that $A^\bullet_{\Gamma}$ is not $G$-left-invariant in this case (for example the $2$-form $\de z_{2\bar 3}$ is not $G$-left-invariant)  and hence $H^{\bullet}\left(\wedge^\bullet\g^*,\, \de\right) \not\simeq H^{\bullet}_{dR}\left(\solvmfd;\R\right)$, see also \cite[Corollary 4.2]{debartolomeis-tomassini}.
On the other hand, we have $H^{\bullet}\left(\wedge^\bullet\mathfrak{h}^* ,\, \de\right) \simeq H^{\bullet}_{dR}\left(\left.\Gamma\middle\backslash H\right.;\R\right)$, where $\mathfrak{h}$ is the Lie algebra of $H$.
In \cite[Main Theorem]{console-fino-solvmanifolds}, it is proven that, for any solvmanifold $\solvmfd$, there exist a connected simply-connected solvable Lie group $\tilde G$ and a finite index subgroup $\tilde \Gamma\subseteq \Gamma$ such that $H^{\bullet}\left(\wedge^\bullet{\tilde\g}^* ,\, \de\right) \simeq H^{\bullet}_{dR}\left(\left.\tilde \Gamma\middle\backslash G\right.;\R\right)$, where $\tilde \g$ is the Lie algebra of $\tilde G$.
\end{rem}

 \item\label{item:ex-nakamura2-2}
If $b\not \in \pi\,\Z$ or $d\not \in\pi\,\Z$, then the sub-complex $B_\Gamma^\bullet$ defined in \eqref{eq:def-b-holpar} is
\begin{eqnarray*}
B_{\Gamma}^{1} &=& \C \left\langle \de\bar z_{1} \right\rangle \;, \\[5pt]
B_{\Gamma}^{2} &=& \C \left\langle \de\bar z_{2}\wedge \de\bar z_{3} \right\rangle \;, \\[5pt]
B_{\Gamma}^{3} &=& \C \left\langle \de\bar z_{1} \wedge \de\bar z_{2}\wedge \de\bar z_{3} \right\rangle \;.
\end{eqnarray*}
Then the double complex $C_{\Gamma}^{\bullet,\bullet}$ is given in Table \ref{table:nakamura-cplx-par-2}.
\end{enumerate}

\begin{center}
\begin{table}[tb]
 \centering
\begin{tabular}{>{$\mathbf\bgroup}l<{\mathbf\egroup$} || >{$}l<{$}}
\toprule
\text{case {\itshape (b)}} & C^{\bullet,\bullet}_{\Gamma} \\
\toprule
(0,0) & \C \left\langle 1 \right\rangle \\
\midrule[0.02em]
(1,0) & \C \left\langle \de z_{1},\; \esp^{-z_{1}}\de z_{2},\; \esp^{z_{1}}\de z_{3} \right\rangle  \\[5pt]
(0,1) & \C \left\langle \de  z_{\bar 1},\; \esp^{-\bar z_{1}}\de z_{\bar 2},\; \esp^{\bar z_{1}}\de \bar z_{3}   \right\rangle \\
\midrule[0.02em]
(2,0) & \C \left\langle \esp^{-z_{1}}\de z_{12},\; \esp^{z_{1}}\de z_{13},\; \de z_{23} \right\rangle \\[5pt]
(1,1) & \C \left\langle \de z_{1\bar1},\; \esp^{-z_{1}}\de z_{2\bar1},\; \esp^{z_{1}}\de z_{3\bar1},\; \esp^{-\bar z_1}\de z_{1\bar2},\; \esp^{\bar z_{1}}\de z_{1\bar3} \right\rangle \\[5pt]
(0,2) & \C \left\langle \esp^{-\bar z_{1}} \de z_{\bar1\bar2},\; \esp^{\bar z_{1}}\de z_{\bar1\bar3},\; \de z_{\bar2\bar3} \right\rangle \\
\midrule[0.02em]
(3,0) & \C \left\langle \de z_{123} \right\rangle \\[5pt]
(2,1) & \C \left\langle \esp^{-z_{1}}\de z_{12\bar1},\; \esp^{ z_{1}}\de z_{13\bar1},\; \de z_{23\bar1},\;  \esp^{-\bar z_{1}}\de z_{23\bar2},\; \esp^{\bar z_{1}}\de z_{23\bar3} \right\rangle \\[5pt]
(1,2) & \C \left\langle \esp^{-\bar z_{1}}\de z_{1\bar1\bar2},\; \esp^{\bar z_{1}}\de z_{1\bar1\bar3},\; \de z_{1\bar2\bar3},\;  \esp^{- z_{1}}\de z_{2\bar2\bar3},\; \esp^{ z_{1}}\de z_{3\bar2\bar3} \right\rangle \\[5pt]
(0,3) & \C \left\langle \de z_{\bar1\bar2\bar3} \right\rangle \\
\midrule[0.02em]
(3,1) & \C \left\langle \de z_{123\bar1},\;\esp^{-\bar z_{1}}\de z_{123\bar2},\; \esp^{\bar z_{1}}\de z_{123\bar3} \right\rangle \\[5pt]
(2,2) & \C \left\langle \esp^{- z_{1}}\de z_{12\bar2\bar3},\; \esp^{z_{1}}\de z_{13\bar2\bar3},\;\de z_{23\bar2\bar3},\; \esp^{-\bar z_{1}}\de z_{23\bar1\bar2},\; \esp^{\bar z_{1}}\de z_{23\bar1\bar3}\right\rangle \\[5pt]
(1,3) & \C \left\langle \de z_{1\bar1\bar2\bar3} \;\esp^{-  z_{1}}\de z_{2\bar1\bar2\bar3},\; \esp^{  z_{1}}\de z_{3\bar1\bar2\bar3}\right\rangle \\
\midrule[0.02em]
(3,2) & \C \left\langle \esp^{- \bar z_{1}}\de z_{123\bar1\bar2},\;\esp^{\bar z_{1}}\de z_{123\bar1\bar 3},\; \de z_{123\bar2\bar 3} \right\rangle \\[5pt]
(2,3) & \C \left\langle \esp^{- z_{1}}\de z_{12\bar1\bar2\bar3},\; \esp^{z_{1}}\de z_{13\bar1\bar2\bar3} ,\; \de z_{23\bar1\bar2\bar3} \right\rangle \\
\midrule[0.02em]
(3,3) & \C \left\langle \de z_{123\bar1\bar2\bar3} \right\rangle \\
\bottomrule
\end{tabular}
\caption{The double complex $C^{\bullet,\bullet}_{\Gamma}$ in \eqref{eq:def-c-cplx-par} for the complex parallelizable Nakamura manifold in case {\itshape (\ref{item:ex-nakamura2-2})}.}
\label{table:nakamura-cplx-par-2}
\end{table}
\end{center}

We compute $H^{\bullet,\bullet}_{BC}(\solvmfd)$ in case {\itshape (\ref{item:ex-nakamura2-2})}, summarizing the results in Table \ref{table:bc-nakamura-cplx-parall}.

\begin{center}
\begin{table}[tb]
 \centering
\begin{tabular}{>{$\mathbf\bgroup}l<{\mathbf\egroup$} || >{$}l<{$}}
\toprule
\text{case {\itshape (b)}} & H^{\bullet,\bullet}_{BC}(\solvmfd)\\
\toprule
(0,0) & \C \left\langle 1 \right\rangle \\
\midrule[0.02em]
(1,0) & \C \left\langle [\de z_{1}] \right\rangle  \\[5pt]
(0,1) & \C \left\langle [\de  z_{\bar 1}]   \right\rangle \\
\midrule[0.02em]
(2,0) & \C \left\langle [\esp^{-z_{1}}\de z_{12}],\; [\esp^{z_{1}}\de z_{13}],\; [\de z_{23}] \right\rangle \\[5pt]
(1,1) & \C \left\langle [\de z_{1\bar1}] \right\rangle \\[5pt]
(0,2) & \C \left\langle [\esp^{-\bar z_{1}} \de z_{\bar1\bar2}],\; [\esp^{\bar z_{1}}\de z_{\bar1\bar3}],\; [\de z_{\bar2\bar3}] \right\rangle \\
\midrule[0.02em]
(3,0) & \C \left\langle [\de z_{123}] \right\rangle \\[5pt]
(2,1) & \C \left\langle [\esp^{-z_{1}}\de z_{12\bar1}],\; [\esp^{ z_{1}}\de z_{13\bar1}],\; [\de z_{23\bar1}] \right\rangle \\[5pt]
(1,2) & \C \left\langle [\esp^{-\bar z_{1}}\de z_{1\bar1\bar2}],\; [\esp^{\bar z_{1}}\de z_{1\bar1\bar3}],\; [\de z_{1\bar2\bar3}],\right\rangle \\[5pt]
(0,3) & \C \left\langle [\de z_{\bar1\bar2\bar3} ]\right\rangle \\
\midrule[0.02em]
(3,1) & \C \left\langle [\de z_{123\bar1}] \right\rangle \\[5pt]
(2,2) & \C \left\langle [\esp^{- z_{1}}\de z_{12\bar2\bar3}],\; [\esp^{z_{1}}\de z_{13\bar2\bar3}],\; [\de z_{23\bar2\bar3}],\; [\esp^{-\bar z_{1}}\de z_{23\bar1\bar2}],\; [\esp^{\bar z_{1}}\de z_{23\bar1\bar3}]\right\rangle \\[5pt]
(1,3) & \C \left\langle [\de z_{1\bar1\bar2\bar3} ]\right\rangle \\
\midrule[0.02em]
(3,2) & \C \left\langle [\esp^{- \bar z_{1}}\de z_{123\bar1\bar2}],\;[\esp^{\bar z_{1}}\de z_{123\bar1\bar 3}],\; [\de z_{123\bar2\bar 3}] \right\rangle \\[5pt]
(2,3) & \C \left\langle [\esp^{- z_{1}}\de z_{12\bar1\bar2\bar3}],\; [\esp^{ z_{1}}\de z_{13\bar1\bar2\bar3}] ,\; [\de z_{23\bar1\bar2\bar3}] \right\rangle \\
\midrule[0.02em]
(3,3) & \C \left\langle [\de z_{123\bar1\bar2\bar3}] \right\rangle \\
\bottomrule
\end{tabular}
\caption{The Bott-Chern cohomology of the complex parallelizable Nakamura manifold in case {\itshape (\ref{item:ex-nakamura2-2})}.}
\label{table:bc-nakamura-cplx-parall}
\end{table}
\end{center}

The cochain complex $A_{\Gamma}^{\bullet}$ in \eqref{eq:def-a1} in case {\itshape (\ref{item:ex-nakamura2-2})} is given in Table \ref{table:a-nak-cplx-parall}.

\begin{center}
\begin{table}[tb]
 \centering
\begin{tabular}{>{$\mathbf\bgroup}l<{\mathbf\egroup$} || >{$}l<{$}}
\toprule
\text{case {\itshape (\ref{item:ex-nakamura2-2})}} & A^{\bullet}_{\Gamma} \\
\toprule
0 & \C \left\langle 1 \right\rangle \\
\midrule[0.02em]
1 & \C \left\langle \de z_{1},\;\de  z_{\bar 1}   \right\rangle \\
\midrule[0.02em]
2 & \C \left\langle  \de z_{1\bar1},\; \de z_{23},\;  \de z_{\bar2\bar3} \right\rangle \\
\midrule[0.02em]
3 & \C \left\langle \de z_{123},\;  \de z_{\bar1\bar2\bar3},\;  \de z_{\bar123},\;  \de z_{1\bar2\bar3} \right\rangle \\
\midrule[0.02em]
4 & \C \left\langle \de z_{123\bar1},\;\de z_{23\bar2\bar3},\;  \de z_{1\bar1\bar2\bar3} \right\rangle \\
\midrule[0.02em]
5 & \C \left\langle \de z_{23\bar1\bar2\bar 3} ,\; \de z_{123\bar2\bar3} \right\rangle \\
\midrule[0.02em]
6 & \C \left\langle \de z_{123\bar1\bar2\bar3} \right\rangle \\
\bottomrule
\end{tabular}
\caption{The cochain complex $A_{\Gamma}^{\bullet}$ in \eqref{eq:def-a1} for the complex parallelizable Nakamura manifold in case {\itshape (\ref{item:ex-nakamura2-2})}.}
\label{table:a-nak-cplx-parall}
\end{table}
\end{center}

Finally, we summarize the results of the computations of the dimensions of the de Rham, the Dolbeault, and the Bott-Chern cohomologies in Table \ref{table:cohom-nakamura-cplx-par} (see \cite[Example 2]{kasuya-mathz} for the Dolbeault cohomology).

\begin{center}
\begin{table}[tb]
 \centering
\begin{tabular}{>{$\mathbf\bgroup}c<{\mathbf\egroup$} || >{$}c<{$} | >{$}c<{$} >{$}c<{$}  |>{$}c<{$} | >{$}c<{$} >{$}c<{$}}
\toprule
\dim_\C H_{\sharp}^{\bullet,\bullet}\left(\solvmfd\right) & \multicolumn{3}{c|}{\text{case {\itshape (\ref{item:ex-nakamura2-1})}}} & \multicolumn{3}{c}{\text{case {\itshape (\ref{item:ex-nakamura2-2})}}} \\
& dR & \delbar & BC & dR& \delbar & BC  \\
\toprule
(0,0) & 1 &  1 & 1 & 1& 1 & 1  \\
\midrule[0.02em]
(1,0) & \multirow{2}{*}{2} & 3 &  1 &\multirow{2}{*}{2}& 3 & 1  \\[5pt]
(0,1) & & 3 &  1 && 1 & 1  \\
\midrule[0.02em]
(2,0) & \multirow{3}{*}{5} & 3 &  3 & \multirow{3}{*}{3} & 3&3  \\[5pt]
(1,1) & & 9 &  7 & &3 & 1  \\[5pt]
(0,2) & & 3 &  3 & &1 & 3  \\
\midrule[0.02em]
(3,0) & \multirow{4}{*}{8} & 1 &  1 &\multirow{4}{*}{4} & 1 & 1  \\[5pt]
(2,1) & & 9 &  9 & & 3& 3  \\[5pt]
(1,2) & & 9 &  9 & &3 & 3  \\[5pt]
(0,3) & & 1 &  1 & &1 & 1  \\
\midrule[0.02em]
(3,1) & \multirow{3}{*}{5} & 3 &  3 & \multirow{3}{*}{3} & 1 & 1  \\[5pt]
(2,2) & & 9 & 11 && 3 & 5  \\[5pt]
(1,3) & & 3 &  3 & &3 & 1  \\
\midrule[0.02em]
(3,2) & \multirow{2}{*}{2} & 3 &  5 &\multirow{2}{*}{2}& 1 & 3 \\[5pt]
(2,3) & & 3 &  5 && 3 & 3  \\
\midrule[0.02em]
(3,3) & 1 &  1 & 1&1 & 1 & 1  \\
\bottomrule
\end{tabular}
\caption{Summary of the dimensions of the cohomologies of the complex parallelizable Nakamura manifold.}
\label{table:cohom-nakamura-cplx-par}
\end{table}
\end{center}

\begin{rem}
 Note that, for any $(p,q)\in\Z^2$,
 $$ \dim_\C H^{p,q}_{BC}(X) + \dim_\C H^{p,q}_{A}(X) \;=\; \dim_\C H^{p,q}_{\del}(X) + \dim_\C H^{p,q}_{\delbar}(X) $$
 in both case {\itshape (\ref{item:ex-nakamura2-1})} and case {\itshape (\ref{item:ex-nakamura2-2})}; note also that
 $$
 \sum_{p+q=k} \left( \dim_\C H^{p,q}_{BC}(X) + \dim_\C H^{p,q}_{A}(X) \right) - 2\,\dim_\C H^k_{dR}(X;\C) \;=\;
 \left\{
 \begin{array}{ll}
   8 & \text{ for } k \in \{1 ,\, 5\} \\[5pt]
  20 & \text{ for } k \in \{2 ,\, 4\} \\[5pt]
  24 & \text{ for } k = 3 \\[5pt]
   0 & \text{ otherwise}
 \end{array}
 \right. \qquad \text{ in case {\itshape (\ref{item:ex-nakamura2-1})}} \;,
 $$
 and
  $$
 \sum_{p+q=k} \left( \dim_\C H^{p,q}_{BC}(X) + \dim_\C H^{p,q}_{A}(X) \right) - 2\,\dim_\C H^k_{dR}(X;\C) \;=\;
 \left\{
 \begin{array}{ll}
   4 & \text{ for } k \in \{1 ,\, 5\} \\[5pt]
   8 & \text{ for } k \in \{2 ,\, 4\} \\[5pt]
   8 & \text{ for } k = 3 \\[5pt]
   0 & \text{ otherwise}
 \end{array}
 \right. \qquad \text{ in case {\itshape (\ref{item:ex-nakamura2-2})}} \;.
 $$ 
\end{rem}

\end{ex}

\appendix

\section{Tables}
\processdelayedfloats

\FloatBarrier

\end{document}